\documentclass[11pt]{amsart}

\usepackage[all,cmtip]{xy}
\usepackage{amsmath, amssymb}
\usepackage{mathtools}
\usepackage[T1]{fontenc}
\usepackage[english]{babel}
\usepackage[hidelinks]{hyperref}
\usepackage[marginpar=2.5cm]{geometry}
\usepackage{xcolor}\usepackage{marginnote}
\usepackage{amsrefs}
\usepackage{orcidlink}

\newtheorem{thm}{Theorem}[section]

\newtheorem*{thm*}{Theorem}
\newtheorem{lem}[thm]{Lemma}

\newtheorem{prop}[thm]{Proposition}
\newtheorem*{prop*}{Proposition}

\newtheorem{cor}[thm]{Corollary}

\theoremstyle{definition}
\newtheorem{defn}[thm]{Definition}
\newtheorem{notation}[thm]{Notation}
\newtheorem{remark}[thm]{Remark}

\newtheorem{question}[thm]{Question}

\newtheorem{example}[thm]{Example}
\newtheorem{claim}[thm]{Claim}

\newcommand{\dminus}{ 
\buildrel\textstyle\ .\over{\hbox{ 
\vrule height3pt depth0pt width0pt}{\smash-} 
}}

\def\e{\epsilon}

\def\al{\alpha}
\def\bb{\mathbb}

\def\ve{\varepsilon}
\def\vp{\varphi}
\def\de{\delta}

\def\Sg{\Sigma}
\def\Om{\Omega}
\def\bb{\mathbb}

\def\ff{\mathfrak}
\def\g{\gamma}
\def\sg{\sigma}
\def\G{\Gamma}
\def\cc{\mathcal}

\DeclareMathOperator{\supp}{supp}

\DeclareMathOperator{\Mod}{Mod}

\newcommand{\fk}{\mathfrak}

\numberwithin{equation}{section}

\begin{document}

%%%%%%%%%%%%%%%%%%%%%%%%%%%%%%%%%%%%%%%%%%%%%%
\title{The model theory of metric lattices: pseudofinite partition lattices}

\author[Contreras-Mantilla]{Jose Contreras-Mantilla \orcidlink{0009-0004-2071-9925}}

\address{Mathematics Department, Purdue University, 150 N. University Street, West Lafayette, IN 47907-2067}
\email{contre56@purdue.edu}

\author[Sinclair]{Thomas Sinclair \orcidlink{0000-0003-0401-7232}}
%\thanks{T. Sinclair was partially supported by NSF grant DMS-1600857.}

\address{Mathematics Department, Purdue University, 150 N. University Street, West Lafayette, IN 47907-2067}
\email{tsincla@purdue.edu}
\urladdr{http://www.math.purdue.edu/~tsincla/}

\begin{abstract}
    We initiate the study of general metric lattices in the context of the model theory of metric structures. As an application we develop a theory of pseudofinite limits of partition lattices and connect this theory with the theory of continuous limits of partition lattices due to Bj\"orner and Lov\'asz.
\end{abstract}



%%%%%%%%%%%%%%%%%%%%%%%%%%%%%%%%%%%%%%%%%%%%%%

\maketitle

\section{Introduction}

The development of theories of continuous limiting structures of combinatorial objects is a subject with a rich history and many actively developing directions \cites{elek, goldbring-towsner, kardos, lovasz-large, nesetril, razborov}. In many of these approaches, the interface of model-theoretic techniques, especially ultraproduct constructions and ultralimits, with analytic concepts such as measurability plays a central role. 
%The goal of this paper is to develop a model-theoretic approach to the structure of geometric lattices within the framework of the model theory of metric structures. 
Though still a young subject, the model theory of metric structures as developed in the seminal work of Ben Yaacov, Berenstein, Henson, and Usvyatsov \cite{mtfms} has provided a powerful toolkit of model-theoretic methods which are well adapted for a wide range of analytic objects such as Banach spaces, dynamical systems, and operator algebras \cites{mtfms, farah, mtoa}. In fact, many of the widely studied classes of metric structures possess a lattice structure which is either crucial to their axiomatization, as in the case of Boolean algebras and $L^p$-spaces \cites{mtfms, henson-raynaud-07}, or is at least central to the general theory, such as the lattice of projections in the theory von Neumann algebras \cite{fhs-ii}.  The goal of the present manuscript is then to isolate and study metric lattices as structures in their own right, with a more explicitly combinatorial bent, and with the thesis that the framework of the model theory of metric structures can be fruitfully used to explore notions of analytic limiting structures. After developing the general theory, we provide applications to the special case of finite partition lattices, investigating which properties, from the model-theoretic point of view, limits of such lattices ought to have. Our analysis complements and clarifies earlier work on inductive limits of partition lattices by Bj\"orner \cite{Bjorner-part, Bjorner-continuous}.

A finite lattice $L$ with a minimal element $0$ comes equipped with a height function, where the height $|x|$ of an element $x$ is the smallest number of atoms (minimal nonzero elements) whose join is $x$. If this height function is an order-preserving map to $\bb Z$ and satisfies the submodularity condition $|x+y| + |xy|\leq |x| + |y|$, where $x+y$ and $xy$ are the meet and join, respectively, of $x$ and $y$, then lattice is said to be geometric \cite[Chapter IV]{birkhoff}. Wilcox and Smiley \cite{Wilcox} realized that the height function on a geometric lattice could be used to define a metric on the lattice and that algebraic properties of the lattice, such as modularity, could be phrased in terms of properties of the metric. This point of view was greatly expanded and developed in works of Sachs \cite{Sachs}, Bj\"orner \cite{Bjorner-part, Bjorner-continuous}, and Bj\"orner  and Lov\'asz  \cite{bjorner-lovasz} where a study of continuous inductive limits of certain classes of geometric lattices was initiated, inspired by von Neumann's theory of continuous geometries \cite{vN-geometry}. A very much related line of inquiry was recently launched by Lov\'asz on the limiting theory of submodular functions \cite{lovasz-submod-setfunction} and continuous limits of matroids \cites{lovasz-matroid}: See also \cites{berczi2024quotient, kardos}. We note that there is a correspondence between matroids and geometric lattices, though this correspondence is not functorial in the sense that many important constructions on either side do not carry over to the other. On the model-theoretic side, the model theory of probability spaces, that is, of metrized Boolean lattices, was fully developed in \cite{mtfms}: See also the expository paper \cite{Mtfps}.  These works form the point of departure for what follows.

After recalling some background in lattice theory and the model theory of metric structures in Section \ref{sec:prelim}, our first task is to describe a first order theory for metric lattices which is suitably general to capture all geometric lattices with metrics which come from their height (rank) functions. Immediately, we are confronted with the fact that while the join function is jointly contractive in the metric, the meet function does not possess a uniform modulus of continuity over all geometric lattices \cite[p. 23]{Bjorner-continuous}. From the perspective of model theory, this means that we are forbidden from using the meet operation in any expression. While Bj\"orner and Lov\'asz in \cite{bjorner-lovasz} address this issue by studying limits for a more restrictive class of lattices which possess a continuous ``pseudointersection'' operation, we take up the theory of metric lattices from scratch and show that most of the meaningful properties which contain meet operations in their definitions (for instance, submodularity) actually imply, and therefore can be rephrased in terms of, expressions only involving the join operation. The tradeoff is that one must modestly increase the quantifier complexity of these expressions by taking infima and/or suprema over additional variables, rather than simple inequalities. In the case of metrically complete metric lattices, we then show that the meet can be canonically and uniquely reconstructed from the join. This is the content of Section \ref{sec:metric-lattices}. 

We note that while the morphisms in our category of metric lattices need only preserve the join operation, this is likely the correct perspective, as this provides a more natural framework for studying quotients. For instance, sequences of ``sparse'' lattices might be suitably characterized as uniformly continuous quotients of Boolean or modular objects. A discussion of this, along with general constructions and examples, is found in Section \ref{sec:examples}. The section culminates in an interesting connection between lattice metrics and conditionally negative definite kernels induced by the join operation which was independently investigated by Lov\'asz \cite{lovasz-submod-setfunction} in the case of Boolean lattices. 

Section \ref{sec:model-theory} is the part of the paper most directly focused on model theory proper. The main result of the section is to show that in the case of (metrically) modular lattices, the meet operation is definable which essentially means that it is permissible to write formal expressions using both the meet and join operations for this class. While the utility of our theory mainly lies in handling non-modular lattices, this serves as an important check that our axioms recover the right framework in the ``tame'' case of modular lattices where the meet operation is also uniformly contractive in the metric. Subsequently, we provide additional axioms for recovering distributive and Boolean metric lattices. Moreover, we show that the model theory of probability spaces as defined in \cite{mtfms} can be recovered as a special case (that is, an elementary class) within our more general metric lattice framework.

In Section \ref{sec:partition} we apply our analysis to the set of finite partition lattices. From the model-theoretic perspective a good limiting object for this category ought to satisfy every formula satisfied by all finite partition lattices. We call such a complete metric lattice a \emph{pseudofinite partition lattice} and seek to understand their structure. As emphasized in the prior work of Bj\"orner \cite{Bjorner-part}, the modular elements form a core for the partition lattice which is crucial to understanding the limiting structure. Therefore, the first major result in this section, Proposition \ref{prop:distance-to-singular-controls-modularity}, demonstrates that the distance to the modular elements can be uniformly controlled by a first-order formula over all finite partition lattices, so the modular elements form a definable set of elements. While the modular elements of a finite partition lattice do not form a sublattice, they do so approximately as the size of the partition set tends to infinity. As a consequence, for any infinite pseudofinite partition lattice, the modular elements form a complete Boolean sublattice (Proposition \ref{prop:boolean-core}) which is consistent with the continuous partition lattice constructed by Bj\"orner. We then initiate a study of the full lattice relative to this Boolean core by assigning to every $x$ element its set of modular complements $\G(x)$ or ``selectors,'' which are so termed as they select a unique element of each block of $x$ in the case of a finite partition lattice. We then show, through a series of intricate combinatorial arguments culminating in Proposition \ref{prop:selectors-hausdorff}, that the selectors can effectively reconstruct any infinite pseudofinite partition lattice as a meet subsemilattice of the lattice of subsets of a complete Boolean lattice. In this way our investigations may provide a suitable framework for a continuous limiting theory of matroids.

\section{Model Theory of Metric Structures}\label{sec:prelim}

In this paper we will use the model-theoretical framework as detailed in \cite{hart} and \cite{mtfms} for the case of one-sorted structures. In particular, we define the following terms.
\begin{itemize}
    
    \item A language, or signature, $\fk L$ contains the usual sets of function, relation and constant symbols with their respective arity, as in first order logic. Additionally, the language will associate to each relation symbol $R$ a function $\Delta_R:[0,1]\rightarrow[0,1]$, such that $\lim_{\e\rightarrow 0^+}\Delta(\e)=0$, which we call the uniform continuity modulus of $R$, and similarly for the function symbols.
    
    \item The logical symbols used in first order logic of $x=y$, $\forall_x$, and $\exists_x$ are replaced by $d(x,y)=0$, $\sup_x$, and $\inf_x$, respectively. Connectives are now all continuous functions $u:[0,1]^n\to [0,1]$ for every $n$.
    
    \item Given a language $\fk L$ we say that $\cc M$ is an \emph{$\fk L$-structure} if $\cc M$ consists of a complete metric space of diameter one along with a collection of functions and distinguished elements satisfying the following. For every function symbol $f\in\fk L$ there is a corresponding uniformly continuous function $f^{\cc M}: M^{a(f)}\rightarrow M$ where $a(f)$ is the arity of $f$ and $\Delta_f$ is a witness of the uniform continuity of $f^{\cc M}$. For every relation symbol $R\in\cc L$ there is a corresponding uniformly continuous function $R^{\cc M}: M^{a(R)}\rightarrow [0,1]$ where $a(R)$ is the arity of $R$ and $\Delta_R$ is a witness of the uniform continuity of $R^{\cc M}$. For every constant symbol $c\in\fk L$ there is a corresponding distinguished element $c^\cc M\in M$.
    
    \item The notions of substructure and embedding between $\cc L$-structures are defined in an analogous way to first order logic.
    
    \item The $\fk L$-terms are constructed analogously to classical first order logic. Similarly, the $\fk L$-formulas are constructed following the same process as in first order logic. A formula with no unquantified variables is called a \emph{sentence}. A formula $\vp(x)$ is \emph{satisfied} for some tuple $a\in M$ for $\cc M$ an $\cc L$-structure if $\vp^{\cc M}(a)=0$.

    \item Let $\Sg$ be a set of $\cc L$-sentences. We say that an $\fk L$-structure \emph{models} $\Sg$, written $\cc M\models \Sg$ if $\sg^{\cc M}=0$ for all $\sg\in \Sg$. The set $\Sg$ is \emph{satisfiable} if it admits a model, in which case we write $\Mod(\Sg)$ for the class of all models of $\Sg$. 

    \item An $\fk L$-theory is a collection of satisfiable $\fk L$-sentences.

    \item A \emph{definable predicate} or \emph{$T$-formula} is the limit of a sequence of  formulas that converge uniformly over all models $\cc M$ of a $\fk L$-theory $T$.

    \item A homomorphism $\rho: \cc M\to \cc N$ of $\fk L$-structures is said to be \emph{elementary} if it preserves the values of all sentences. 
\end{itemize}

For a complete treatment of the model theory of metric structures the reader may consult \cite{farah}, \cite{hart}, or \cite{mtfms}. Besides the general notions mentioned before, we will also use the following model-theoretical results about definability.

Let $\fk L$ be a signature and $T$ an $\fk L$-theory. We say that a functor $\cc X:\Mod(T)\to {\sf Met}$, where ${\sf Met}$ is the category of metric spaces with isometric embeddings, is a \emph{$T$-functor} if $\cc X(\cc M)$ is a closed subspace of $M$ and if for every elementary map $\rho:\cc M\to\cc N$ of $T$-models, $\cc X(\rho)=\rho \upharpoonright_{\cc X(\cc M)}$. 

Roughly, a $T$-functor is said to be definable if the distance to $\cc X(\cc M)$ is computable in terms of a formula: see the first item in the theorem below. The following theorem found in \cite{hart} gives a full picture of the meaning of definability in the context of the model theory of metric structures.

\begin{thm}[\cite{hart}, Theorem 8.2]\label{thm:definability}
    Let's consider a signature $\cc L$, a $\cc L$-theory $T$ and a $T$-functor $\cc X$. The following are equivalent:
    \begin{enumerate}
        \item There is a $T$-formula $\vp(\bar{x})$ such that for any model $\cc M$ of $T$ and any $\bar{a}\in\cc M$,
        $$\vp^\cc M(\bar{a})=d(\bar{a},\cc X(\cc M)).$$
        \item For any $T$-formula $\psi(\bar{x},\bar{y})$, there is a $T$-formula $\sigma(\bar{y})$ such that for any model $\cc M$ of $T$ and any $\bar{a}\in\cc M$,
        $$\sigma^\cc M(\bar{a})=\inf_{\bar{x}\in\cc X(\cc M)}\psi^\cc M(\bar{x},\bar{a})$$
        (or equivalently the analogous statement replacing $\inf$ by $\sup$).
        
        \item For every $\varepsilon>0$, there is a $T$-formula $\vp(\bar{x})$ and $\delta>0$ such that, for all models $\cc M$ of $T$:
        \begin{itemize}
            \item $\cc X(\cc M)\subset \{\bar{a}\in \cc M:\vp^\cc M(\bar{a})=0\}$,
            \item For all $\bar{a}\in\cc M$, if $\vp^\cc M(\bar{a})<\delta$, then $d(\bar{a},\cc X(\cc M))\leq\varepsilon$.
       \end{itemize}
       
       \item For every $\varepsilon>0$, there is a basic $\fk L$-formula $\vp(\bar{x})$ and $\delta>0$ such that, for all models $\cc M$ of $T$:
        \begin{itemize}
            \item $\cc X(\cc M)\subset \{\bar{a}\in \cc M:\vp^\cc M(\bar{a})=0\}$,
            \item For all $\bar{a}\in\cc M$, if $\vp^\cc M(\bar{a})<\delta$, then $d(\bar{a},\cc X(\cc M))\leq\varepsilon$.
        \end{itemize}
        
        \item For any set $I$, family of $\fk L$-structures $\cc M_i$ for $i\in I$, and ultrafilter $\cc U$ on $I$,
        $$\cc X(\cc N)=\prod_\cc U \cc X(\cc M_i)\text{, where } \cc N=\prod_\cc U \cc M_i.$$
        \end{enumerate}
\end{thm}

The Beth Definability Theorem is a powerful tool in model theory since it offers a criterion to identify relations and predicates that can be expressed by formulas, and it establishes a connection between the explicit sense of definability, using formulas, to the implicit sense of definability, using models of a theory. The following formulation of the theorem appears in \cite{hart}.
\begin{defn}
    Let $\fk L$ be a continuous language, $T$ an $\fk L$-theory and $\vp(\bar{x}),\psi(\bar{x})$ two $T$-formulas. We say that $\vp(\bar{x})$ is $T$-equivalent to $\psi(\bar{x})$ if and only if for every model $\cc M$ of $T, \bar{a}\in M, \psi^{\cc M}(\bar{a})=\vp^{\cc M}(\bar{a})$, i.e., their interpretations for any model of the theory are the same function. 
\end{defn}

\begin{thm}[Beth Definability Theorem]
    Suppose that $\fk L'\subseteq\fk L$ are two continuous languages with the same sorts. Further, suppose $T$ is an $\fk L$-theory. If the forgetful functor $F$ from models of $T$ to $\fk L'$-structures given by restriction is an equivalence of categories onto the image of $F$, then every $\fk L$-formula is $T$-equivalent to an $\fk L'$-formula.
\end{thm}

As it is claimed in \cite{hart}, this result is even more important in the context of continuous logic given the challenge of identifying definable predicates as limits of explicit formulas. We will use this to prove the definability of the \textit{meet} operation under certain circumstances. Specifically, we will use the following corollary of the Beth Definability Theorem found in \cite[Corollary 4.2.3]{farah}.

\begin{cor}\label{cor:beth-expanded-language}
    Suppose that $\mathcal{C}$ is an elementary class of structures in a language $\fk L$ and, for every $A\in\mathcal{C}$, the structure $A$ is expanded by a predicate $P_A$ which is uniformly continuous with uniform continuity modulus independent of our choice of A. Let $\mathcal{C}' := \{(A,P_A): A \in \mathcal{C}\}$ be a class of structures for an expanded language $\mathcal{L}'$ with a predicate for $P$. If $\mathcal{C}'$ is an elementary class for a theory $T'$ in the language $\fk L'$, then $P$ is $T'$-equivalent to a definable predicate in $\fk L$.
\end{cor}

\begin{remark}
    For the purposes of this paper, we will use ``formula'' to refer to either an $\fk L$-formula, $T$-formula, or definable predicate in general. For a more detailed explanation of each one and their differences we recommend the reader to check \cite{hart}.
\end{remark}

\section{Metric Lattices}\label{sec:metric-lattices}

A \emph{semilattice} is a set $L$ equipped with a commutative, associative, idempotent binary operation $+$. We can see that this operation defines a partial order $\leq$ on $L$ by $x\leq y$ if $x+y=y$, and a strict partial order by setting $x<y$ if and only if $x\leq y$ and $x\neq y$. This choice of partial order creates what is often referred to as an \emph{upper semilattice} in the literature: when we refer to a partial order on a semilattice, we will always mean this one. Further, we will always assume that semilattices are equipped with (unique) elements $0$ and $1$ so that $0+x=x$ and $1+x=1$ for all $x\in L$.

\begin{defn}\label{defn:lattice}
    A tuple $\cc L =(L, <, +, \cdot, 0, 1)$ a \emph{lattice} if:
    \begin{enumerate}
        \item $(L,<)$ is a partially ordered set with a maximal element $1$ and a minimal element $0$;
        \item $+,\cdot$ are commutative, associative, idempotent binary operations;
        \item $x\leq y$ if and only if $x+y=y$ if and only if $xy=x$.       
    \end{enumerate}
\end{defn}

Following convention, we will refer to $+$ as the \emph{join} and $\cdot$ as the \emph{meet}.
Note that the conditions imply that $\sum_{i=1}^k x_i$ is the least upper bound of $\{x_1,\dotsc,x_k\}$ and $\prod_{i=1}^k x_i$ is the greatest lower bound. Furthermore, if for any subset $S\subset L$ we can find its lowest upper bound and its greatest lower bound we say that the lattice is \emph{complete} and we will denote them by $\Sigma_{x\in S} x$ and $\prod_{x\in S}x$ respectively.

If $\cc L$ is a finite semilattice, more generally, if the partial order is complete, then the meet operation $xy$ can be defined in terms of this order as the maximal element simultaneously below both $x$ and $y$. When generalizing beyond finite lattices, this poses an interesting choice of category in terms of whether to consider lattice or semilattice morphisms.

\begin{defn}\label{defn:metric-lattice}
    A \emph{metric lattice} is a tuple $\cc L = (L, <, +, \cdot, 0, 1, d)$ so that $(L, <, +, \cdot, 0, 1)$ is a lattice and $d$ is a metric on $L$ satisfying:
    \begin{enumerate}
        \item $d(1,0) =1$;
        \item $d(x+z,y+z)\leq d(x,y)$;
        \item $d(x,y)\leq d(x+y,0) - d(xy,0)$
    \end{enumerate}
    for all $x,y,z,\in L$.
    If $(L,d)$ is a complete metric space, we say that $\cc L$ is a \emph{complete metric lattice}.
    We call a metric satisfying the last condition \emph{semi-modular}.
\end{defn}

\begin{remark}\label{rmk:uniform-continuity}
    Notice that condition (2) implies that 
    \begin{equation}\label{eq:uniform-continuity}
        d(x+z,y+w)\leq d(x,y) + d(z,w)
    \end{equation}
    by the triangle inequality, so $L\times L\ni (x,y)\mapsto x+y\in L$ is jointly uniformly continuous.
\end{remark}

Given a metric lattice $\cc L = (L, <, +, \cdot, 0, 1, d)$, we can define a \emph{rank function} $|\cdot|: L\to \bb R_{\geq 0}$ by
\begin{equation}
    |x| := d(x,0).
\end{equation}
The semi-modular condition on the metric then becomes
\begin{equation}\label{eq:semi-mod-metric}
    d(x,y)\leq |x+y| - |xy|.
\end{equation}   

\begin{prop}\label{prop:metric-norm}
    If $\cc L = (L, <, +, \cdot, 0, 1, d)$ is a metric lattice, then:
    \begin{enumerate}
        \item $x<y$ implies that $|x| < |y|$. 
        \item $d(x,y)\in [0,1]$ for all $x,y\in L$.
    \end{enumerate}
\end{prop}

\begin{proof}
    If $x < y$, then $xy = x < y = x+y$, so $0 < d(x,y) \leq |y| - |x|$ by (\ref{eq:semi-mod-metric}). Since $0\leq x\leq 1$ for all $x\in L$, it follows that $0\leq |x|\leq 1$. As a consequence 
    \[d(x,y) \leq |x+y| - |xy|\leq |x+y|\leq 1.\qedhere\]
\end{proof}

\begin{remark}
    This definition differs from the definition given in \cite{Wilcox}. It can be seen that a metric lattice as defined therein is also a metric lattice in our sense. Indeed, the only difference is the insistence that $d(x,y) = 2|x+y| - |x| - |y|$. For a metric lattice we define 
    \begin{equation}
        d'(x,y) := d(x+y,x) + d(x+y,y).
    \end{equation}
\end{remark}

\begin{prop}\label{prop:metric-rank}
    For a metric lattice $\cc L$ the following are true.
    \begin{enumerate}
        \item $||x|-|y||\leq d(x,y)$.
        \item $d(x,y) = |y| - |x|$ for $x<y$.
        \item $d'(x,y) = 2|x+y| - |x| - |y|$.
        \item $|x+y|+|xy|\leq|x|+|y|$.
        \item If $z\leq x$, then $|z|+|y+x|\leq |z+y|+|x|$.
        \item $d'$ is a metric and $\cc L' = (L,d')$ is a metric lattice.
        \item $d(x,y)\leq d'(x,y)\leq 2\,d(x,y)$.
    \end{enumerate}
\end{prop}
\begin{proof} We prove each item sequentially.
    \begin{enumerate}
        \item 
            Since $d$ is a metric, by the reverse triangle inequality, we have that 
            $$||x|-|y||=|d(x,0)-d(y,0)|\leq d(x,y).$$
            
        \item 
            If $x<y$, by the previous point and the semi-modular condition, we conclude that 
            \[||y|-|x||=|y|-|x|\leq d(x,y)\leq |x+y|-|xy|=|y|-|x|;\] thus, $d(x,y)=|y|-|x|$.
            
        \item 
            From item $(2)$, we have that 
            \begin{equation*}
                \begin{split}
                    d'(x,y)=d(x+y,x)+d(x+y,y) &= \left(|x+y|-|x|\right) + \left(|x+y|-|y|\right)\\
                     &=2|x+y|-|x|-|y|.
                \end{split}
            \end{equation*}
            
        \item 
            Again by item $(2)$, we have that $d(x+y,x)=|x+y|-|x|$ and $d(y,xy)=|y|-|xy|$. Also, from condition $(2)$ in the definition of a metric lattice, we have that 
            \[|x+y|-|x|=d(x+y,x)=d(x+y,x+xy)\leq d(y,xy)=|y|-|xy|;\] 
            thus, we get that $|x+y|+|xy|\leq|x|+|y|$.
            
        \item 
            By item $(4)$, we have that for $x,y,z\in L$, 
            \[|(z+y)x|+|z+y+x|\leq |z+y|+|x|.\]
            Since $z\leq z+y,x$, we have that $z\leq (z+y)x$ and $z+y+x=y+x$; thus, by Proposition \ref{prop:metric-norm}.1 we get that
            \[|z|+|y+x|\leq |(z+y)x|+|z+y+x|\leq |z+y|+|x|.\]
            
        \item 
            We first prove that $d'$ is a metric. Since $d$ is a metric, it is clear that $d'(x,y)\geq0$, $d'(x,x)=0$, and $d'(x,y)=d'(y,x)$. 
            
            Using our definition of a metric lattice and items $(2)$ and $(3)$, we can follow the strategy given in \cite[Theorem 1.1]{Wilcox} to prove the triangle inequality. Indeed, by item $(3)$ and Proposition \ref{prop:metric-norm}.1, we get that
            \begin{align*}
                d'(x,z)+d'(z,y)-d'(x,y)&=2(|x+z|+|z+y|-|x+y|-|z|)\\
                                      &\geq 2(|x+z|+|z+y|-|x+y+z|-|z|).
            \end{align*}
            Using inequality $(5)$ with $z\leq x+z$, we conclude that 
            \[2(|x+z|+|z+y|-|x+y+z|-|z|)\geq 0\; ;\]
            therefore, $d'(x,y)\leq d'(x,z)+d'(y,z)$.

            Now we need to check that $(L,d')$ satisfies the axioms of being a metric lattice. Note that $d'(x,0) = |x|$, so $d'$ induces the same rank function as $d$. We have that 
            \begin{equation*}
                \begin{split}
                    &d'(x,y) - d'(x+z,y+z)\\
                    &= (2|x+y| - |x| - |y|) - (2|x+y+z| - |x+z| - |y+z|)\\
                    &= (|x + z| - |x|) + (|y+z| - |y|) - 2(|x+y+z| - |x+y|)\\
                    &= d(x+z,x) + d(y+z,y) - 2\,d(x+y+z,x+y) \geq 0         
                \end{split}
            \end{equation*} by item (2) and Definition \ref{defn:metric-lattice}.2. It follows from item (4) that \[d'(x,y)\leq |x+y| - |xy|.\]

        \item 
            By Definition \ref{defn:metric-lattice}.2 and the triangle inequality, we have that
        \begin{align*}
            d(x,y)&\leq d(x,x+y)+d(y,x+y)&\\
            &=d(x+x,x+y)+d(x+y,y+y)&\\
            &\leq d(x,y)+d(x,y) =2\,d(x,y).
        \end{align*}
        Thus, by the definition of $d'$, we conclude $d(x,y)\leq d'(x,y)\leq 2\, d(x,y)$. \qedhere
    \end{enumerate}
\end{proof}

\begin{prop}\label{prop:semimodular-upgrade}
    If $\cc L$ is a metric lattice, then we have that
    \begin{equation}
        |x+y| + |z| \leq |x+z| + |y+z|
    \end{equation}
    for all $x,y,z\in L$.
    As a consequence, 
    \begin{equation}\label{eq: strg-sm-mod}
       d(x,y) \leq d'(x,y) \leq |x+y| + |x+z| + |y+z| - |x| - |y| - |z|
    \end{equation}
    for all $x,y,z\in L$.
\end{prop}

\begin{proof}
    Since $z\leq z+x$, we can apply Proposition \ref{prop:metric-rank}.5 to get
    $$|z|+|y+z+x|\leq|z+y|+|z+x|.$$
    By Proposition \ref{prop:metric-norm} we know that $|x+y|\leq|y+z+x|$, so we conclude that 
    $$|x+y|+|z|\leq|x+z|+|y+z|$$ for all $x,y,z,\in L$.
    This result, along with Proposition \ref{prop:metric-rank}.3 and Proposition \ref{prop:metric-rank}.7, gives us
    \begin{align*}
        d(x,y)\leq d'(x,y)&=2|x+y|-|x|-|y|\\
        &=|x+y|+(|x+y|-|x|-|y|)\\
        &\leq (|x+z|+|y+z|-|z|)+(|x+y|-|x|-|y|)\\
        &=|x+y| + |x+z| + |y+z| - |x| - |y| - |z|.
    \end{align*}
    for all $x,y,z\in L$.
\end{proof}

\begin{remark}
    From Proposition \ref{prop:metric-rank}.2 we see that the metrics $d$ and $d'$ agree on any chain (totally ordered subset) of $L$. By the proof of item (7) in the same, we have that if $d$ and $\tilde d$ are metrics that induce a metric lattice structure on the same lattice and for which the rank functions agree, then $d(x,y)\leq 2\tilde d(x,y)$. In general, the set of metrics on a lattice can be quite large. For example, take any Borel regular measure $\mu$ of full support on $[0,1]$. We say that two Borel sets $A,B$ are equivalent if $\mu(A\Delta B)=0$, and write $[A]$ for the equivalence class of $A$. For any measure $\nu$ in the same absolute continuity class as $\mu$, i.e., $\mu(A)=0 \Leftrightarrow \nu(A)=0$, the equivalence classes are the same, and each one induces a metric lattice structure on the same lattice of equivalence classes of Borel sets by $d_\nu([A],[B]) = \nu(A \cup B) - \nu(A\cap B)$.

    %\rmkts{QUESTION: Is there an example of a metric lattice for which $d$ and $d'$ are different?}
\end{remark}

\begin{remark}\label{rmk:triple}
    Let's define
    \begin{equation}
        [x,y,z]_d := |x+y| + |x+z| + |y+z| - |x| - |y| - |z|
    \end{equation}
    and note this expression is symmetric in the variables $x,y,z$ and that $[x,y,z]_d = [x,y,z]_{d'}$. From the equation (\ref{eq: strg-sm-mod}) of Proposition \ref{prop:semimodular-upgrade} we see that for all $x,y,z\in L$
    \begin{equation*}
        \max\{d(x,y), d(y,z), d(x,z)\}\leq [x,y,z]_d
    \end{equation*}
    for any metric lattice. Further, by Proposition \ref{prop:metric-rank}.2 we have that
    \begin{equation}
        \begin{split}
            &d(x+y,0) - d(z,0) + d(x+z,x) + d(y+z,y)\\
            &= |x+y| - |z| + (|x+z| - |x|) + (|y+z| - |y|) = [x,y,z]_d.
        \end{split}  
    \end{equation}
\end{remark}

\noindent Thus, we have the following
\begin{cor}\label{cor:semimodular-upgrade}
    Every metric lattice $\cc L$ satisfies the following inequality:
    \begin{equation}
          d(x,y) \leq |x+y| +  d(x+z,x) + d(y+z,y) - |z|
    \end{equation}
    for all $x,y,z\in L$.
\end{cor}
This last inequality is a strengthening of the semi-modular condition (3) from Definition \ref{defn:metric-lattice} since we can get back the semimodular condition by substituting $z=xy$. 

Note that $d'(x,y)\leq |x+y| + d'(x+z,x) + d'(y+z,y) - |z|$ is equivalent to $d'$ satisfying the triangle inequality.

%%%%%%%%%%%%%%%%%%%%%%%%%%%%%%%%%%%%%%%%%%%%%%%%%%%%%%%%%%%%%%%
\subsection{Modularity}
%%%%%%%%%%%%%%%%%%%%%%%%%%%%%%%%%%%%%%%%%%%%%%%%%%%%%%%%%%%%%%%%

In a lattice $L$, a pair $(x,y)$ of elements is called a \emph{modular pair} if $xy + z = x(y+z)$ for all $z\leq x$. Note that this relation need not be symmetric: if it is, we will call the lattice \emph{semi-modular}. As will be seen below, in section \ref{Sec:mm-meet-def}, the detailed study of modularity from the metric lattice perspective occupies a prominent place in the theory as modularity is essentially equivalent to continuity of the meet operation.

\begin{defn}\label{defn:metric-mod-pair}
    Given a metric lattice $\cc L$, we say that $(x,y)\in L^2$ is a \emph{metrically modular pair} if 
    \[|x+y| + |xy| = |x| + |y|;\] equivalently, if $d'(x,y) = |x+y| - |xy|$. Note that, unlike modularity of a pair, metric modularity is a symmetric relation.
    We say that $x\in L$ is \emph{metrically modular} if $(x,y)$ is a metrically modular pair for all $y\in L$ and that $\cc L$ is \emph{metrically modular} if every pair $(x,y)\in L^2$ is a metrically modular pair.
\end{defn}

\begin{prop}
    For a metric lattice, every metrically modular pair $(x,y)$ is a modular pair.
\end{prop}

\begin{proof}
    Let's suppose that $(x,y)$ is a metrically modular pair. To prove that they are a modular pair we need to show that for all $z\leq x, z+yx=(z+y)x$. We already know that for all $z\leq x$ we have the inequalities $z+yx\leq (z+y), x$, hence $z+yx\leq (z+y)x$. Now using Proposition \ref{prop:metric-norm}.1 with $yx\leq y$ and Proposition \ref{prop:metric-rank}.4  we have that for all $z\leq x$
    \begin{equation*}
        \begin{split}
            |(z+y)x|&\leq |z+y|+|x|-|z+y+x|\\
            &=|z+y|+|x|-|y+x|\\
            &=|z+y|+|x|-(|x|+|y|-|xy|)\\
            &=|xy|+|z+y|-|y|\leq |z+yx|.
        \end{split}
    \end{equation*}
    We conclude that $|(z+y)x|=|z+yx|$; thus, as $|\cdot|$ is strictly increasing, it follows that $(z+y)x=z+yx$.
\end{proof}

\begin{remark}
    We say that a metric lattice $\cc L$ is \emph{strongly metrically semi-modular} if 
    \begin{equation}
        |x| + |y| + |z| \leq |x+z| + |y+z| + |xy|
    \end{equation}
    for all $x,y,z\in L$; equivalently, if
    \begin{equation}
        |x+y| - |xy|\leq [x,y,z]_d
    \end{equation}
    for all $x,y,z\in L$.
    It is clear that $|x+y| - |xy| \geq \inf_z [x,y,z]_d$; thus, strong semi-modularity is equivalent to 
    \begin{equation}
        |x+y| - |xy| = \inf_z [x,y,z]_d
    \end{equation}
    for all $x,y\in\cc L$.
    In fact, we have that a metric lattice $\cc L$ is strongly metrically semi-modular if and only if it is metrically modular.

    Indeed, assume $\cc L$ is strongly metrically semi-modular. Setting $y=z$, we get that $|x|+|y|\leq|x+y|+|xy|$ for all $x,y\in L$. By Proposition \ref{prop:metric-rank} we have the reverse inequality, so $\cc L$ is metrically modular.
    On the other hand, if $\cc L$ is metrically modular, we have $d'(x,y) = |x+y| - |xy|$; therefore, by Proposition \ref{prop:semimodular-upgrade} we have $|x+y|-|xy|\leq [x,y,z]_d$.
\end{remark}

\begin{defn}
    Let $\cc L$ be a lattice. We say that a function $\rho: L\times L\to \bb R_{\geq 0}$ is a \emph{quasi-metric} if:
    \begin{enumerate}
        \item $\rho(x,y) = \rho(y,x)$;
        \item $\rho(x,y)=0$ if and only if $x=y$;
        \item $\rho(x,y)\leq \rho(x,z)+\rho(y,z)$ when $z\leq x$ or $z\leq y$.
    \end{enumerate}
\end{defn}

\begin{prop}
    Let $\cc L$ be a metric lattice, and let 
    \[\de(x,y) := |x| + |y| - 2|xy|.\] The following are true.
    \begin{enumerate}
        \item $\de(x,y)$ is a quasi-metric.
        \item $\de(x,0) = |x|$ for all $x\in L$.
        \item $d'(x,y)\leq \de(x,y)$.
        \item $\de$ is a metric if and only if $\cc L$ is metrically modular.
    \end{enumerate}
\end{prop}

\begin{proof}
    We begin with the proofs of the numbered assertions.
\begin{enumerate}
    \item First note that since $xy\leq x,y$, we have that $|xy|\leq |x|,|y|$, so $\de(x,y)\geq 0$. It is clear that $\de(x,y)=\de(y,x)$ and that if $x=y$, then $\de(x,y)=0$. Now, if $x\neq y$, we have that $x\neq xy$ or $y\neq xy$; therefore, $|x|>|xy|$ or $|y|>|xy|$ and we get that $\de(x,y)>0$. Thus, we conclude that $\de(x,y)=0$ if and if $x=y$. Now to conclude that $\de$ is a quasi-metric we need check the last condition.

    If $z\in L$ is such that $z\leq x$ or $z\leq y$, then $|xz|\leq |xy|$ or $|yz|\leq|xy|$ and given that $|xz|,|yz|\leq |z|$ we conclude that $|xz|+|yz|\leq |z|+|xy|$. Using this inequality we get
    \[|x|+|y|+2|xz|+2|yz|\leq |x|+|z|+|y|+|z|+2|xy|,\]
    hence
    \begin{equation*}
        \begin{split}
            \de(x,y)&=|x|+|y|-2|xy|\\
            &\leq |x|+|z|-2|xz|+|y|+|z|-2|yz|=\de(x,z)+\de(y,z).
        \end{split}
    \end{equation*}
    Thus, $\de$ is a quasi-metric.
    
    \item It is clear that $\de(x,0)=|x|+|0|-2|0|=|x|-0=|x|$ for all $x\in L$.
    
    \item By Proposition \ref{prop:metric-rank}.4 we know that for all $x,y\in L$
    $$|x+y|+|xy|\leq |x|+|y|.$$
    We can reorder the terms to get
    \[2|x+y|-|x|-|y|\leq|x|+|y|-2|xy|;\]
    thus, $d'(x,y)\leq \de(x,y)$.

    \item For the last assertion, if $\de$ satisfies the triangle inequality, then for all $x,y,z\in L$ we have that
    \[|x|+|y|-2|xy|\leq |x|+|z|-2|xz|+|z|+|y|-2|yx|,\]
    which shows that
    \[|xz|+|yz|\leq|z|+|xy|.\]
    Taking $z=x+y$ we have that $|x|+|y|\leq|x+y|+|xy|$. We already know that $|x+y|+|xy|\leq|x|+|y|$, so we conclude that all pairs $x,y\in L$ are metrically modular pairs.
    
    In the other direction, if $\cc L$ is metrically modular, then $\de(x,y) = d'(x,y)$, so is a metric by Proposition \ref{prop:metric-rank}.6. \qedhere
    \end{enumerate}
\end{proof}

%%%%%%%%%%%%%%%%%%%%%%%%%%%%%%%%%%%%%%%%%%%%%%%%%%%%%%%%%%%%%%%
\subsection{Complete metric lattices}
%%%%%%%%%%%%%%%%%%%%%%%%%%%%%%%%%%%%%%%%%%%%%%%%%%%%%%%%%%%%%%%

Metric completeness is another crucial aspect of our study, as it allows to deduce algebraic structure from analytic arguments. The next result is the key observation which underpins the usage of metric completeness throughout.
\begin{lem}\label{lem:net-convergence}
    If $\cc L$ is a complete metric lattice, then any monotone net converges.
\end{lem}
\begin{proof}
  Let $\cc A$ be a directed set and suppose $(x_a)_{a\in\cc A}$ is an increasing net. Then $(|x_a|)_{a\in \cc A}$ is increasing and bounded by $1$, hence $\alpha := \lim_{\cc A} |x_a|$ exists. Given $\e>0$, if $\alpha-\e < |x_a| \leq |x_{b}|\leq \alpha$ for $a\leq b$, then $d(x_a,x_{b}) = |x_b| - |x_a|< \e$ by Proposition \ref{prop:metric-rank}.2. Thus, $(x_a)_{a\in \cc A}$ is a Cauchy net and $x^* = \lim_{\cc A} x_a$ exists by metric completeness. Similarly, if $(x_a)_{a\in A}$ is a decreasing net, $\lim_{\cc A} x_a$ exists.
\end{proof}

\begin{prop}
\label{lem:sequence}
    If $\cc L$ is a complete metric lattice and $S\subseteq L$ is closed under $+$ (respectively, closed under $\cdot$), then the least upper bound (resp., greatest lower bound) of $S$ exists and belongs to the closure of $S$.
\end{prop}

\begin{proof}
    Let $\cc F$ be the collection of all finite subsets of $S$. Consider the net $(x_F)_{F\in \cc F}$ where $x_F := \sum_{x\in F} x$. We have that $x$ is an upper bound for $S$ if and only if $x_F\leq x$ for all $F\in \cc F$. As $(x_F)$ is increasing,  $x^* = \lim_{\cc F} x_F$ exists by Lemma \ref{lem:net-convergence}. Clearly $x_F + x_G = x_G$ if $F\subset G$, so passing to limits we have that $x_F + x^* = x^*$ which shows that $x^*$ is an upper bound for $S$. If $S$ is closed under $+$, then $x_F\in S$, so $x^*$ belongs to the closure of $S$.
    
    It now suffices to check that $x^*$ is the least upper bound. If $y$ is another upper bound for $S$, then so is $x^*y$, so without loss of generality we may only consider the case when $x_F\leq y < x^*$ for all $F\in\cc F$. However, we would then have that by Proposition \ref{prop:metric-rank}.2 that $d(x_F,x^*) = |x^*| - |x_F| \geq  |x^*| - |y|>0$, which is a contradiction. 

    The proof when $S$ is closed under $\cdot$ with the greatest lower bound follows nearly identically.
\end{proof}

The following consequence is readily apparent.

\begin{cor}
\label{prop:metric-to-lattice-complete}
    A complete metric lattice is complete as a lattice.
\end{cor}

\begin{remark}
    An important consequence of the previous results is the existence of $\liminf$ and $\limsup$ for any sequence in a complete metric lattice, i.e., given a sequence $(x_n)$ there exist $\limsup x_n:=\lim_{n\to\infty}\sum_{k=n}^\infty x_k$ and $ \liminf x_n:=\lim_{n\to \infty} \prod_{k=n}^\infty x_k $. For more details of their construction see Proposition \ref{prop:metrically-modular-axiom} and Proposition \ref{prop:weak-complement}, respectively.
\end{remark}

\begin{lem}\label{lem:metric-sum}
    Let $\cc L$ be a complete metric lattice. If $(x_i)$ is a sequence in $L$, by Proposition \ref{lem:sequence} we can set $z = \sum_ix_i$ and we have that \[d(z + y, y)\leq \sum_i d(x_i+y,y).\]
\end{lem}

\begin{proof}
    Setting $z_n = \sum_{i=1}^n x_i$, we have that $( z_n)$ is an increasing sequence, so $(z_n)$ is Cauchy and $z_n\to z$ by Lemma \ref{lem:net-convergence}. We have inductively that
    \begin{align*}
        d(z_n+y,y) &= d(z_{n-1} + x_n + y, y)\\
        &\leq d(z_{n-1} + y + x_n, y + x_n) + d(x_n + y,y)\\
        &\leq d(z_{n-1} + y,y) + d(x_n + y, y) \leq \sum_{i=1}^n d(x_i+y,y);
    \end{align*}
    hence, the result obtains in the limit by continuity.
\end{proof}

\begin{prop}\label{prop:metrically-modular-axiom}
    Let $\cc L$ be a complete metric lattice. A pair $x,y\in L$ is metrically modular if and only if for every $n\in \bb N$ there is $z\in L$ so that
    \begin{equation}
        |x| + |y| \leq |x+y| + |z| + \frac{1}{n},\ d(x+z,x) < \frac{1}{n},\ d(y+z,y) < \frac{1}{n}.
    \end{equation}
\end{prop}

\begin{proof}
    The ``only if'' direction follows by choosing $z = xy$ and using Proposition \ref{prop:metric-rank}.4. For the converse, fix $x,y\in L$. Choose $z_i\in L$ for $n = 2^i$ as above, we will proceed to construct $t$ the $\limsup z_i$. Set $t_k = \sum_{i=k}^\infty z_i$ with $t_k^N = \sum_{i=k}^N z_i$, noting that $t_k^N\to t_k$ as $N\to\infty$ by Lemma \ref{lem:net-convergence}. For $k<l$ and $M> N$ we have that $d(t_k^M + t_l^N, t_k^M)=0$, so 
    \[d(t_k + t_l^N, t_k) = \lim_M d(t_k^M + t_l^N, t_k^M)=0\]
    when $k< l$. Taking the limit in $N$, we see that $d(t_k + t_l, t_k)=0$, or $t_l\leq t_k$ when $k< l$. Since $(t_k)$ is  therefore decreasing, it is Cauchy. Set $t = \lim_k t_k$. By Lemma \ref{lem:metric-sum}, we have that $d(t_k+x,x), d(t_k+y,y)\leq \sum_{i=k}^\infty 2^{-i}\to 0$ as $k\to \infty$; thus, $d(t+x,x)= d(t+y,y)=0$ or $t\leq xy$.

     Since $t_k\geq z_k$, we have that
    \[|x|+|y|\leq |x+y| + |t_k| + \frac{1}{2^k}\]
    for all $k\in \bb N$ which implies that 
     \[|x|+|y|\leq |x+y| + |t|\leq |x+y| + |xy|.\]
     The result is now apparent by Proposition \ref{prop:metric-rank}.4.
\end{proof}

%%%%%%%%%%%%%%%%%%%%%%%%%%%%%%%%%%%%%%%%%%%%%%%%%%%%%%%%%%%%%%%%%%
\subsection{Metrically complete semilattices}\label{sec:alt-axiomatization}
%%%%%%%%%%%%%%%%%%%%%%%%%%%%%%%%%%%%%%%%%%%%%%%%%%%%%%%%%%%%%%%%%

In the presence of metric completeness we can describe a metric lattice using only uniformly continuous operations, a point which will be crucial for the purpose of axiomatization as metric structures. 

\begin{defn}
    We define a \emph{pseudo-metric semilattice} to be a structure $\cc P = (P,+,0,1,d)$ where the following hold:
    \begin{enumerate}
        \item $(P,d)$ is a pseudo-metric space;
        \item $0,1\in P$ are distinguished elements;
        \item $+: P\times P\to P$ is a commutative, associative, idempotent ($x+x = x$) operation; and
        \item $0+x = x$ and $1+x =1$ for all $x\in P$.
    \end{enumerate}
    Further, we require the following axioms:
    \begin{enumerate}
         \item $d(1,0) =1$;
        \item $d(x+z,y+z)\leq d(x,y)$; and
        \item $d(x,y)\leq d(x+y,0) + d(x+z,x) + d(y+z,y) - d(z,0)$
    \end{enumerate}
    for all $x,y,z\in P$.

    We say that $\cc P$ is a \emph{metric semilattice} if $(P,d)$ is a metric space and that $\cc P$ is \emph{complete} if $(P,d)$ is a complete metric space.
\end{defn}

\begin{remark}
    The property ``$d(x,y)\leq d(x+y,0) + d(x+z,x) + d(y+z,y) - d(z,0)$ for all $x,y,z\in P$'' implies ``$d(x,y)\leq d(x+y,0) - d(z,0)$ for all $x,y,z\in P$ with $z\leq x,y$.'' Inspection of the proofs of Propositions \ref{prop:metric-norm} and \ref{prop:metric-rank} will convince the reader that these properties are equivalent.
\end{remark}

\begin{prop}\label{prop:exchange-relation}
    Let $\cc P$ be a semilattice and let $f: P\to [0,1]$ be a strictly increasing function with $f(0)=0$ and $f(1)=1$. We have that $d_f(x,y) := 2\,f(x+y) - f(x) - f(y)$ defines a semilattice metric on $L$ if and only if $f$ satisfies the ``exchange relation'' 
    \[f(x+y) + f(z)\leq f(x) + f(y+z)\]
    whenever $z\leq x$.
\end{prop}

\begin{proof}
    If $d_f$ is a semilattice metric, then $f$ satisfies the exchange relation by adapting the proof of Proposition \ref{prop:metric-rank}.5. Conversely, $d_f$ satisfies the triangle inequality if and only if $f(x+y) + f(z) \leq f(x+z) + f(y+z)$. 
    Applying the exchange relation to $x+z$, $y$, and $z$, we obtain
    \[f(x+y+z) + f(z)\leq f(x+z) + f(y+z)\]
    for all $x,y,z\in L$.
    We have that $d_f(x+z,y+z)\leq d_f(x,y)$ if and only if
    \[2\,f(x+y+z) + f(x) + f(y) \leq 2\,f(x+y) + f(x+z) + f(y+z).\] This may be obtained from two applications of the previous inequality, permuting the variables $x,y,z$ appropriately.
\end{proof}

\begin{remark}\label{rmk:exchange-relation}
    For a semilattice $\cc P$ and a map $f:P\to \bb R$, the key inequality used in the proof of Proposition \ref{prop:exchange-relation} is that for all $x,y,z\in P$
    \[f(x+y+z) + f(z) \leq f(x+z) + f(y+z).\]
    This property is referred to as \emph{strong subadditivity} in the literature \cite{choquet, niculescu-survey}. We direct the reader to Proposition \ref{prop:cnd-implies-metric} below for a related result. For a lattice, the exchange relation and strong subadditivity are equivalent to $f$ being submodular, that is, $f(x+y) + f(xy)\leq f(x) + f(y)$.
\end{remark}

\begin{lem}\label{lem:metric-quotient}
    Let $\cc P$ be a pseudo-metric semilattice. Let $(P',d')$ be the metric quotient of $(P,d)$ and define $\cc P'= (P', +', [0], [1], d')$ where $[x]+'[y] := [x+y]$. Then $\cc P'$ is a metric semilattice.
\end{lem}

\begin{proof}
    It suffices to check that $+': P'\times P'\to P'$ is well defined. Letting $[x_1] = [x_2]$ and $[y_1] = [y_2]$, we have that
    \begin{align*}
        d'([x_1+y_1],[x_2+y_2]) &= d(x_1+y_1,x_2+y_2)\\
        &\leq d(x_1+y_1,x_2+y_1) + d(x_2+y_1,x_2+y_2)\\
        &\leq d(x_1,x_2) + d(y_1,y_2) = 0.
    \end{align*}
    Therefore, $[x_1+y_1] = [x_2+y_2]$ and the join operation is well defined.
\end{proof}

\begin{lem}\label{lem:join-ordering}
    If $\cc P$ is a pseudo-metric semilattice, then declaring $x\leq y$ if $y =x+y$ defines a partial order on $P$ where $0$ and $1$ are the minimal and maximal elements, respectively, and where $x+y$ is the least upper bound of $\{x,y\}$. Additionally, $d(x,y)\leq |y|- |x|$ when $x\leq y$. (We define $|x| = d(x,0)$ as above in the lattice case.)
\end{lem}

\begin{proof}
    First let's prove that $\leq$ is a partial order. It is clear that for all $x\in\cc P,x\leq x$ since $+$ is idempotent, $\leq$ is reflexive. If for a pair $x,y\in\cc P, x\leq y$ and $y\leq x$, then $y=x+y=y+x=x$, so $\leq$ is antisymmetric. Lastly, for all $x,y,z\in \cc P$ if $x\leq y$ and $y\leq z$, then $z=y+z$ and $y=x+y$; therefore, $z=y+z=(x+y)+z=x+(y+z)=x+z$, so $x\leq z$ and we conclude that $\leq$ is transitive. In conclusion, it is a partial order. 
    
    By definition 1.21 (4) it is clear that for all $x\in\cc P,0\leq x\leq 1$; thus, $0$ and $1$ are the minimal and maximal elements, respectively. By idempotency, commutativity and associativity of $+$, we get $x+y=x+(x+y)=y+(x+y)$, and thus $x,y\leq x+y$. If $x,y\leq z$, then $(x+y)+ z = x + (y+z) = x+ z = z$, so $x+y\leq z$. Finally, if $x\leq y$, we can use the axiom (3) about $d$ in Definition 1.21 to obtain
    \begin{align*}
        d(x,y) &\leq d(x+y,0) + d(x+x,x) + d(y+x,y) - d(x,0)\\
        &= d(y,0)+d(x,x)+d(y,y)-d(x,0)=|y|-|x|. \qedhere
    \end{align*}
\end{proof}

The next proposition follows in exactly the same way as Proposition \ref{lem:sequence}.

\begin{prop}\label{prop:pre-lattice-complete}
    If $\cc P$ is a complete metric semilattice, then every set $S\subset P$ has a least upper bound. If $S$ is closed under $+$, then the least upper bound belongs to the closure of $S$.
\end{prop}

By Corollary \ref{cor:semimodular-upgrade} we have that every metric lattice is an example of a metric semilattice. We now prove a partial converse.

\begin{prop}\label{prop:complete-is-lattice}
    Every complete metric semilattice $\cc P$ is a complete metric lattice with $0,1$ being the minimal and maximal element, respectively, and $+$ being the lattice join operation.
\end{prop}

\begin{proof}
Let us consider
\begin{equation*}\label{eq:complete-is-lattice}
    \phi_{x,y}(z) := d(z+x,x) + d(z+y,y).
\end{equation*}
Then $\phi_{x,y}(z)=0$
if and only if $z$ is a lower bound for $\{x,y\}$.
The sum of two common lower bounds for $x$ and $y$ is again a common lower bound, as $z+x = z'+x =x$ implies that $z+z'+x = x$. Moreover, the set of all lower bounds is closed by continuity of $\phi_{x,y}$, so the set of common lower bounds for $x$ and $y$ must have a largest element by Proposition \ref{prop:pre-lattice-complete}. We define $xy$ as this element. The function $(x,y)\mapsto xy$ is clearly commutative and idempotent. We have that $t\leq (xy)z$ if and only if $t\leq xy,z$ if and only if $t\leq x,y,z$ which shows that it is also associative. The remaining parts of Definition \ref{defn:lattice} are easy to check.
\end{proof}

The following result was proved in \cite[Section 3]{Wilcox}. We present a more direct approach here.

\begin{cor}
    Let $\cc L$ be a metric lattice. We have that the metric completion $\overline{\cc L}$ is a metric lattice and the canonical inclusion $L\hookrightarrow \overline{L}$ is an isometric embedding of metric semilattices.
\end{cor}

\begin{proof}
    If $\cc L$ is a metric lattice, then it is also a metric semilattice. The metric completion of a semilattice is again a metric semilattice as the $+$ operation is uniformly continuous. \qedhere %The inclusion $L\to \overline{L}$ is isometric, so clearly order preserving.
\end{proof}

\begin{prop}\label{prop:hom-closed}
    Let $\cc L$ and $\cc L'$ be complete metric lattices and $\phi: \cc L\to \cc L'$ a continuous semilattice homomorphism, that is $\phi(0)=0$, $\phi(1)=1$, and $\phi(x+y) = \phi(x)+\phi(y)$ for all $x,y\in L$. We have that the subsemilattice $\phi(L)\subset L'$ is a complete metric lattice.
\end{prop}

\begin{proof}
    By Proposition \ref{prop:pre-lattice-complete} we need only show that $\phi(L)$ is closed.
    First suppose that $(\phi(x_n))$ is an increasing sequence, so that $\phi(x_n)\to z$ for some $z\in L'$ by Lemma \ref{lem:net-convergence}. Setting $y_n = \sum_{k=1}^n x_n$ and $y = \sum_{k=1}^\infty x_k$, we have that $y_n\to y$ and so $\phi(x_n)=\phi(y_n)\to \phi(y)$ by continuity. Therefore, $z = \phi(y)$. 

    Now let $\phi(x_n)\to z$ be arbitrary. By passing to a subsequence, we may assume that $d(\phi(x_n),z)< \frac{1}{2^n}$. For $n< L$, we set $y_{n,L} = \phi(\sum_{k=n}^L x_k) = \sum_{k=n}^L \phi(x_k)$ and $y_n = \phi(\sum_{k=n}^\infty x_k)$ so that by the previous paragraph $y_{n,L}\to y_n$ as $L\to\infty$ and by Lemma \ref{lem:metric-sum} we have that $d(y_n+z,z)\to 0$ as $n\to \infty$. 

    Setting $w_n = \sum_{k=n}^\infty x_k\in L$, we have that $(w_n)$ is a decreasing sequence, so has a limit $w\in L$. We have $y_n = \phi(w_n)\to \phi(w) =: y\in \phi(L)$ by continuity. Therefore, by the previous paragraph we have that $d(y+z,z)=0$ or $y\leq z$. On the other hand, $y_n \geq \phi(x_k)$ for all $k>n$, so $y_n = y_n + \phi(x_k)\to y_n+z$ as $k\to\infty$. Thus, $y = y+z$, which implies that $y\geq z$. Altogether this shows that $z\in \phi(L)$, which establishes the result. \qedhere
\end{proof}

\section{Examples}\label{sec:examples}

In this section we detail various constructions of lattice metrics. For the case of finite Boolean lattices, there is some overlap with results obtained by Lov\'asz in \cite[section 9]{lovasz-submod-setfunction} which the authors were unaware of during the completion of the manuscript. However, the perspective presented here is more general and in some ways complementary to those results. 

\begin{example}
    For a finite set $E$, let $\cc B(E)$ be the lattice of subsets of $E$ under the normalized Hamming metric $d(A,B) := \frac{|A\Delta B|}{|E|}$ for $A,B\in 2^E$.
\end{example}

\begin{example}\label{ex:partition-lattice}
    Let $E$ be a finite set. Given $x,y$ partitions of $E$ we can define a partial order by $x\leq y$ if and only if $x$ refines $y$, that is, if for all $B\in x$, exists a $B'\in y$ such that $B\subseteq B'$. Let $\cc P(E)$ be the lattice of partitions of $E$ with the metric \[d(x,y)=\frac{\#x+\#y-2\#(x+y)}{|E|-1},\]
    where for a partition $\# x$ denotes the number of blocks. When $E = [[n]] := \{1,.\dotsc,n\}$, we denote $P(E)$ by $P_n$.
\end{example}

\begin{example}
    Let $E$ be a finite set. We say that $r: 2^E\to \bb N_0$ is a \emph{rank function} if the following hold for all $A,B\in 2^E$:
    \begin{enumerate}
        \item $r(A)\leq |A|$;
        \item $A\subset B$ implies that $r(A)\leq r(B)$;
        \item $r(A\cup B) + r(A\cap B)\leq r(A) + r(B)$.
    \end{enumerate}
    There is a one-to-one correspondence between rank functions on a set $E$ and \emph{matroids} on the same set: See \cite[Section 1.3]{Oxley}.
\end{example}

For a rank function define the pseudo-metric $D_r(A,B) := 2\,r(A\cup B) - r(A) - r(B)$.

\begin{lem}
    If $r$ is a rank function on $E$, then $D_r(A,B)\leq |A \Delta B|$.
\end{lem}

\begin{proof}
    It follows from submodularity that $D_r(A,B)\leq r(A\cup B) - r(A\cap B)$. Using submodularity and disjointness, we have that 
    \[r(A\cup B)\leq r(A\setminus B) + r(A\cap B) + r(B\setminus A),\]
    so altogether
    $D_r(A,B) \leq r(A\setminus B) + r(B\setminus A)\leq |A\setminus B| + |B\setminus A| = |A\Delta B|$.
\end{proof}

\begin{defn}
    We say that a rank function $r: 2^E\to \bb N_0$ is \emph{$k$-sparse} if 
    \[\# E\leq k\cdot r(E).\]
\end{defn}

\begin{example}
    If $G = (V,E)$ is a $k$-regular graph and $r: 2^E\to \bb N_0$ is given by $r(A)$ is the maximal cardinality of the edges of a forest in $(V, A)$, then $r$ is $k$-sparse.
\end{example}

For a rank function $r$ on $E$, let's define the normalized rank distance by $d_r(A,B) = D_r(A,B)/r(E)$. Let $\cc L_r(E) = (L_r(E), <,+,\cdot, d_r)$ be the quotient metric lattice associated to the pseudo-metric lattice $(2^E, d_r)$. Since $2^E$ is a finite lattice, the fact that $\cc L_r(E)$ is a metric lattice follows from Lemma \ref{lem:metric-quotient} and Proposition \ref{prop:complete-is-lattice}. The lattice $\cc L_r(E)$ is referred to as the \emph{lattice of flats} of the associated matroid. The following proposition is evident.

\begin{prop}
    For a rank function $r$ on $E$ which is $k$-sparse, we have that the quotient map $q: 2^E\to L_r(E)$ from $\cc B(E)$ with the normalized Hamming metric to $\cc L_r(E)$ induces a surjective, $k$-Lipschitz metric semilattice homomorphism $\cc B(E)\to \cc L_r(E)$.
\end{prop}

\begin{defn}
    Let $\cc P$ be a semilattice. We say that a function $\rho: P\to \bb R$ is \emph{positive definite} (respectively, \emph{positive semidefinite}) if for all $x_1,\dotsc, x_n\in P$, the matrix 
    \[R_{ij} := [\rho(x_i + x_j)]\] is positive definite (resp., positive semidefinite).
\end{defn}

\begin{example}
    For $x\in P$, define $\rho_x: P\to \{0,1\}$ by $\rho_x(y) = 1$ if $y\leq x$ and $\rho_x(y)=0$ otherwise. Then $\rho_x$ is positive semidefinite.
\end{example}

The next lemma begins to show how positive (semi)definiteness imposes strong restrictions on the function $\rho$.

\begin{lem}\label{lem:psd-decreasing}
    If $\rho: P\to \bb R$ is positive semidefinite, then $\rho$ is positive and nonincreasing, with $\rho$ being strictly decreasing if $\rho$ is positive definite. 
\end{lem}

\begin{proof}
    We have $\rho(x) = \rho(x+x)\geq 0$ by positive definiteness. If $y > x$, then the matrix
    \[\begin{bmatrix} \rho(x) & \rho(y)\\\rho(y) & \rho(y)\end{bmatrix}\]
    is positive semidefinite, which implies that $\rho(x)\geq \rho(y)$ with strict inequality in the positive definite case.
\end{proof}

Let $\cc P = (P,+,0,1)$ be a complete semilattice (thus, every subset of $P$ has a least upper bound). We define the \emph{order topology} on $P$ to be the topology generated by the collection of closed sets 
\[C_x := \{y\in P: x+y = y\} = \{y\in P: y\geq x\}\]
for each $x\in L$. Let $\cc C$ be the $\sg$-algebra generated by the sets $\{C_x: x\in P\}$.

\begin{example}\label{ex:psd-measure}
    If $\mu$ is a measure on $(P, \cc C)$, then $\rho(x) := \mu(C_x)$ is positive semidefinite since, using that $x+y$ is the least upper bound of $\{x,y\}$, 
    \[\rho(x+y) = \mu(C_{x+y}) = \mu(C_x\cap C_y) = \int 1_{C_x} 1_{C_y} d\mu.\]
    Indeed, choosing $x_1,\dotsc,x_n$ in $P$ and $g_1,\dotsc, g_n\in\bb R$. Setting $g = \sum_{i=1}^n g_i 1_{C_{x_i}}$, we have that 
    \[\sum_{i,j} g_jg_j\rho(x_i+x_j) = \int\sum_{i,j} g_ig_j 1_{C_x}1_{C_y} d\mu = \int g^2 d\mu \geq 0.\]
\end{example}

If $\cc P$ is a finite semilattice, then for every function $f: P\to \bb R$, there is a unique function $f^\mu: P\to \bb R$ so that
\[f(x) = \sum_{y\in C_x} f^\mu(y).\]
The function $f^\mu$ is known as the \emph{M\"obius inversion} of $f$. We refer the reader to \cite[Section IV.2]{aigner} or \cite{rota-i} for a detailed treatment of this topic. Setting $n = |P|$, we fix an enumeration $x_1,x_2,\dotsc,x_{n}$ of $P$ so that $x_i\leq x_j$ only if $i\leq j$. We define the $n\times n$ \emph{zeta matrix} of $\cc P$ by 
\[\zeta_{ij} := \delta_{x_j,x_i+x_j} \]
where $\de_{ab}$ is the Kronecker delta. That is, $\zeta_{ij} = 1$ if $x_i\leq x_j$ and $0$ otherwise. Since $\zeta$ is upper triangular with ones on the main diagonal, it is invertible. For a sequence $a =(a_1,\dotsc,a_{n})$, let $D_a$ be the diagonal matrix with entries $[D_a]_{ii} = a_i$. The following identity is due independently to Lindstr\"om \cite{lindstrom} and Wilf \cite{wilf}:
\begin{equation}\label{eq:wilf}
    [f(x_i+x_j)]_{i,j=1}^{n} = \zeta D_{f^\mu} \zeta^t.
\end{equation}

The following result which provides a converse to Example \ref{ex:psd-measure} was first observed by Mattila and Haukkanen in the positive definite case \cite{mattila}. We give a simple proof that extends to the positive semidefinite case as well. 

\begin{prop}\label{prop:measure-rep-psd}
    Let $\cc P$ be a finite semilattice. We have that $\rho: P\to \bb R$ is positive semidefinite (resp., positive definite) if and only if there is a measure $m_\rho$ on $P$ (resp., a measure of full support on P) so that
    \[\rho(x) = m_\rho(C_x).\]
\end{prop}

\begin{proof}
    Example \ref{ex:psd-measure} verifies the `if' direction. For the converse, it is a standard fact that if $A\in M_n(\bb R)$ is positive semidefinite, then $BAB^t$ is as well for any $B\in M_n(\bb R)$, with the same holding for positive definiteness in the case that $B$ is invertible. Thus, $\zeta^{-1}[\rho(x_i+x_y)](\zeta^{-1})^t = D_{\rho^\mu}$ is a positive (semi)definite diagonal matrix by equation (\ref{eq:wilf}). This is equivalent to $\rho^\mu\geq 0$ in the positive semidefinite case and $\rho^\mu>0$ in the positive definite case. For $A\subseteq P$, we then define $m_\rho(A):= \sum_{x\in A} \rho^\mu(x)$. The result now follows by M\"obius inversion.
\end{proof}

\begin{remark}
    More generally, integral representation theorems for positive definite functions on semilattices give analogous measure representations under suitable topological regularity assumptions; see, for example, \cite{Berg}.
\end{remark}

%\begin{remark}
%    If $\cc P$ is a complete metric lattice which is first countable, then standard sampling theory arguments can be used to show that the same is true for any continuous positive semidefinite function with regard to a $\sg$-finite Radon measure on $\cc C$.
%\end{remark}

\begin{example}
    Let $K\in M_n(\bb C)$ be a Hermitian matrix with spectrum in $[0,1]$. And let $L = 2^{[[n]]}$ be the lattice of subsets of $[[n]] = \{1,\dotsc,n\}$. By a result of Lyons \cite{Lyons} there is a measure $\mu_K$ on $L$ so that
    \[\mu_K(C_S) = \det(K_S)\]
    for all $S\subseteq [[n]]$ where $K_S$ is the principal submatrix of $K$ given by deleting all entries in rows or columns not indexed by $S$. Note that $1 = \mu(L) =\mu(C_{\emptyset}) = \det(K_{\emptyset})$ by convention.

    The existence of such a measure can be reduced to case where $P$ is an orthogonal projection. Indeed, any Hermitian, contractive $n\times n$ matrix can be realized as a principal submatrix of a $2n\times 2n$ orthogonal projection and positive semidefiniteness of kernels is closed under restriction. 
    
    Suppose $P$ is of rank~$r$, and define
    \[
        \rho(S):=\det(P_S), \qquad S\subseteq[[n]].
    \]
    By viewing $P = VV^*$ for a isometry $V: \mathbb C^r\to \mathbb C^n$ we have by the Cauchy--Binet formula that
    \[
        \det(P_{S}) = \sum_{R\supseteq S,\ |R|=r} \det(P_{R}),
    \]
    hence $\rho(S) = m(C_S)$ for
    \[
        m(R) =
        \begin{cases}
        \det(P_R), & |R|=r,\\
        0, & |R|\neq r.
        \end{cases}
    \]
\end{example}

\begin{defn}
    We say that a function $\eta: P\to \bb R$ is \emph{conditionally negative definite} if 
    \[\sum_{i,j} {c_ic_j}\cdot\eta(x_i + x_j) \leq 0\]
    when $\sum_i c_i = 0$ for all $x_1,\dotsc,x_n\in P$.
\end{defn}

\begin{remark} 
    By Schoenberg's theorem \cite[Proposition 11.2]{Roe}, $\eta$ is conditionally negative definite if and only if $x\mapsto \exp(-t\eta(x))$ is positive semidefinite for all $t\geq0$. Therefore, we have that every conditionally negative definite function is increasing.
\end{remark}

For $\eta$ conditionally negative definite, define
\[d_\eta(x,y) := 2\eta(x+y) - \eta(x) - \eta(y),\]
which is again conditionally negative definite with $d_\eta(x,x)=0$ and $d_\eta(x,y)\geq 0$ for all $x,y\in P$. We note that
\[|x|_\eta := d_\eta(x,0) = \eta(x) - \eta(0)\geq 0\]
is again conditionally negative definite and that
\[d_\eta(x,y) = 2|x+y|_\eta - |x|_\eta - |y|_\eta = d'_\eta(x,y).\]

\begin{prop}\label{prop:cnd-implies-metric}
    For $\eta$ conditionally negative definite we have that $(\cc P,d_\eta)$ is a pseudo-metric semilattice. Moreover, if $\exp(-\eta)$ is positive definite, then $d_\eta$ is a semilattice metric on $\cc P$.
\end{prop}

\begin{proof}
    We write $|x|$ for $|x|_\eta$ and $d(x,y)$ for $d_\eta(x,y)$ for brevity. If $\exp(-\eta)$ is positive definite, then $|x|$ is strictly increasing, from which it easily follows that $d(x,y)>0$ if $x\not= y$.
    
    By Proposition \ref{prop:exchange-relation} and Remark \ref{rmk:exchange-relation}, we need only check that strong subadditivity holds, that is,
    \[|x+y+z|+|z|\leq |x+z|+|y+z|.\]

    Consider the matrix
    \[K := \begin{bmatrix}
        |x+z| & |x+y+z| & |x+z| & |x+y+z|\\
        |x+y+z| & |y+z| & |y+z| & |x+y+z|\\
        |x+z| & |y +z| & |z| & |x+y+z|\\
        |x+y+z| & |x+y+z| & |x+y+z| & |x+y+z|
    \end{bmatrix}
    \]
    which is $|a_i+a_j|$ where $a_1=x+z$, $a_2 =y+z$, $a_3= z$, and $a_4 = x+y+z$. Since $|x|$ is conditionally negative definite, we have $\xi K \xi^t\leq 0$ where $\xi = [1,1,-1,-1]$. We calculate
    \begin{equation*}
        \begin{split}
            \xi K\xi^t &=\\ 
           &=|x+z| + |x+y+z| - |x+z| - |x+y+z|\\
        &+ |x+y+z|  + |y+z| - |y+z| - |x+y+z|\\
         &- |x+z| -  |y +z|  + |z|  + |x+y+z|\\
         &- |x+y+z| - |x+y+z| + |x+y+z| + |x+y+z|\\
            &= -|x+z| -|y+z| + |z| + |x+y+z|,
        \end{split}
    \end{equation*}
    which proves the assertion. \qedhere

\end{proof}

\begin{example}
    If $\rho:P\to \bb R$ is positive semidefinite, then $x\mapsto a - \rho(x)$ for $a\in \bb R$ is clearly conditionally negative definite.
    %, and again by Schoenberg's theorem $x\mapsto \exp(-t(a-\rho(x))$ is positive semidefinite for $t\geq 0$. 
    Thus, if $(\Om, \cc B, \mu)$ is a probability space, then $X\mapsto \mu(X^c)$ is positive semidefinite, and $X\mapsto \mu(X) = 1 - \mu(X^c)$ is conditionally negative definite.
    
    Accordingly, let $(\Om,\cc B)$ be a Borel space, and let $\Phi$ be a finite measure on $(\Om\times \Om, \cc B\otimes \cc B)$. The map $\rho_\Phi: \cc B\ni X\mapsto \Phi(X^c\times X^c)$ is positive semidefinite, since 
    \[\rho_\Phi(X\cup Y) = \Phi\left((X\cup Y)^c\times (X\cup Y)^c\right) = \Phi((X^c\times X^c)\cap (Y^c\times Y^c)).\] Therefore,
    \[\eta_\Phi: \cc B\ni X\mapsto \Phi(\Om\times \Om) - \rho_\Phi(X) = \Phi(\Om\times X) + \Phi(X\times X^c)\]
    is conditionally negative definite and $\exp(-t\eta_\Phi)$ is positive semidefinite for $t\geq 0$.
    Moreover, we can see that $d_\Phi(X,Y) = \eta_\Phi(X\cup Y) - \eta_\Phi(X) - \eta_\Phi(Y)$ gives the lattice of Borel sets the structure of a pseudo-metric pre-lattice.
\end{example}

\begin{question}
    In \cite[section 9]{lovasz-submod-setfunction} it is shown that for finite Boolean lattices conditional negative definiteness coincides with a sequence of inequalities which strengthen submodularity first considered by Choquet \cite{choquet}. These inequalities already imply, in the metric case, that elements of the lattice are uniquely complemented: see Remark \ref{rem:submod-choquet} below. It would be interesting to determine if there is a similar characterization of conditionally negative definite functions on semilattices in general.
\end{question}

Let $\cc P$ be a finite semilattice and let $x_1,\dotsc,x_n$ be an enumeration of the elements of $P$ such that $x_i \leq x_j$ only if $i\leq j$. Given a function $f: P\to \bb R$, we say that a subset $S\subseteq [[n]]$ is a \emph{chain} if $\{x_i: i\in S\}$ is totally ordered. We say that an $n\times n$ matrix $A$ is \emph{totally $\cc P$-non-negative} if for all chains $S,T$ with $|S| = |T|$ we have that $\det(A_{S,T}) \geq 0$, where $A_{S,T}$ is the submatrix of $A$ consisting of the rows $S$ and columns $T$. If $\cc P = [[n]]$ with $i+j = \max\{i,j\}$, then total $\cc P$-positivity corresponds with the usual notion of $k$-total non-negativity of matrices \cite{ando-tp}.

We say that $f: P\to \bb R_{\geq 0}$ is \emph{totally $\cc P$-non-negative} if the associated matrix \[[f(x_i+x_j)]_{i.j=1}^n\] is totally $\cc P$-non-negative. We note that even though such a matrix is symmetric, it is not clear that a variant of Sylvester's criterion would apply to show that a totally $\cc P$-non-negative matrix is positive semidefinite. However, considering the $3\times 3$ principal submatrix corresponding to a triple $z\leq x,y$, 
\[\begin{bmatrix}
    f(z) & f(x) & f(y)\\
    f(x) & f(x) & f(x+y)\\
    f(y) & f(x+y) & f(y)
\end{bmatrix},\]
We see that total $\cc P$-non-negativity implies that $f$ is non-increasing and consequently $f$ satisfies the celebrated ``FKG inequality'' \cite{FKG}
\[f(x+y)f(z) \geq f(x)f(y)\]
whenever $z\leq x,y$. Note that the determinant of the $2\times 2$-upper right block is necessarily non-positive if $f$ is non-increasing, but is not considered unless $x <  y$, in which case it is $0$. Evidently, if $f$ is the characteristic function of $P\setminus C_x$ for any $x\in P$, then $f$ is totally non-negative as the non-zero entries are constant and form a rectangle in the upper left corner of the associated matrix. 

It would be interesting to completely characterize such matrices. For $\cc P = [[n]]$ with $i+j = \max\{i,j\}$ this is true for all $f: [[n]]\to \bb R_{\geq 0}$ which are non-increasing \cite[example 7.I(f)]{ando-tp}. However, the example of $2^n\ni S\mapsto \det(K_S)$ above shows that this is false in general as $\det(K_{S\cup T})\det(K_{S\cap T})\leq \det(K_S)\det(K_T)$ for all matrices $K$ that are positive semidefinite by Kotelyanskii's inequality \cite[Theorem 7.8.9]{horn}.

\section{Model Theoretic Aspects}\label{sec:model-theory}

We now turn to studying the model theory of metric lattices. We define a language $\ff L$ as follows.

%%%%%%%%%%%%%%%%%%%%%%%%%%%%%%%%%%%%%%%%%%%%%%%%%%%%%%%%
\subsection{Language and axioms for metric lattices}
%%%%%%%%%%%%%%%%%%%%%%%%%%%%%%%%%%%%%%%%%%%%%%%%%%%%%%%%%%%

\begin{enumerate}
    \item There is a single sort $(X,d)$ consisting of a complete metric space of diameter $1$.
    \item There are constant symbols $0$ and $1$.
    \item A function symbol $+:X\times X\to X$ which is a contraction where $X\times X$ is equipped with the $\ell^1$-metric, $d_1((x,y),(w,z)) = d(x,w) + d(y,z)$.
\end{enumerate}

For axioms we require the following collection of sentences $T_{ML}$:
\begin{enumerate}
    \item $d(0,1)-1$;
    \item $\sup_x d(x+0,x) + d(x+1,1)$;
    \item $\sup_x d(x+x,x)$;
    \item $\sup_{x,y} d(x+y,y+x)$;
    \item $\sup_{x,y,z} d((x+y)+z,x+(y+z))$;
    \item $\sup_{x,y,z} (d(x+z,y+z)\dminus d(x,y))$;
    \item $\sup_{x,y,z} \left([d(x,y) + d(z,0)] \dminus [d(x+y,0) + d(x+z,x) + d(y+z,y)]\right)$.
\end{enumerate}

Clearly, every model of $T_{ML}$ is complete metric semilattice. The following proposition summarizes the main results of section \ref{sec:alt-axiomatization}.

\begin{prop}
    We have that $\ff L$-structures $\cc M$ and $\cc N$ are models of $T_{ML}$ if and only if $\cc M$ and $\cc N$ are complete metric lattices. Furthermore, the $\ff L$-structure homomorphisms $\phi:\cc M\to \cc N$ are exactly the contractive maps so that $\phi(0)=0$, $\phi(1)=1$, and $\phi(x+y) = \phi(x)+\phi(y)$ for all $x,y\in\cc M$ and $\phi(\cc M)\subseteq\cc N$ is a model of $T_{ML}$.
\end{prop}

\begin{proof}
    This follows directly from Propositions \ref{prop:complete-is-lattice} and \ref{prop:hom-closed}.
\end{proof}

%%%%%%%%%%%%%%%%%%%%%%%%%%%%%%%%%%%%%%%%%%%%%%%%%%%%%%%%%%%%%%%%%%%%%
\subsection{Axioms for metrically modular lattices}\label{Sec:mm-meet-def}
%%%%%%%%%%%%%%%%%%%%%%%%%%%%%%%%%%%%%%%%%%%%%%%%%%%%%%%%%%%%%%%%%%%%

Consider the sentence
\begin{equation}
    \sg_{\text{mod}} := \sup_{x,y}\inf_z\max\{(|x| + |y|) \dminus (|x+y| + |z|), d(x+z,x), d(y+z,y)\}.
\end{equation}
and the theory 
\begin{equation}\label{eq:t-mml}
    T_{MML} := T_{ML}\cup\{\sg_{\text{mod}}\}.
\end{equation}

\begin{prop}\label{prop:t-mml-language}
    We have that an $\ff L$-structure $\cc M$ models $T_{MML}$ if and only if it is an metrically modular complete metric lattice.
\end{prop}

\begin{proof}
    This follows directly from Proposition \ref{prop:metrically-modular-axiom}.
\end{proof}

Our objective will be to show that the \textit{meet} operation is definable in $T_{MML}$, which in our context means that we need to prove that the predicate $P(x,y,z):=d(xy,z)$ is definable.

\begin{lem}\label{lem:metric-mod-contraction}
    If $\cc M$ is a metric lattice and $(x,z),(y,z)$ are metrically modular pairs, then $d'(xz,yz)\leq d'(x,y)$ and $d(xz,yz)\leq 2\,d(x,y)$.
\end{lem}

\begin{proof}
    We follow the proof given in \cite[Lemma 21]{Sachs}. Applying Proposition \ref{prop:metric-rank}.4 to $x+y$ and $z$, we get $|z(x+y)|\leq|z|+|x+y|-|z+x+y|$. Since $(z,y)$ is a metrically modular pair, i.e. $|zy|=|z|+|y|-|z+y|$, it follows that
    \begin{align*}
        |z(x+y)|-|zy|&\leq |z|+|x+y|-|z+x+y|-|z|-|y|+|z+y|\\
                    &= |x+y|-|y|+|z+y|-|z+x+y|
                    \leq |x+y|-|y|.
    \end{align*}
    Using a similar argument, we can show that $|z(x+y)|-|zx|\leq|x+y|-|x|$; therefore,
    \[2|z(x+y)|-|zx|-|zy|\leq 2|x+y|-|x|-|y|=d'(x,y).\]
    Now observe that as $zx,zy\leq z(x+y)$, we have $zx+zy\leq z(x+y)$, so $|zx+zy|\leq|z(x+y)|$ and we can conclude $d'(zy,zx)=2|zx+zy|-|zx|-|zy|\leq d'(x,y)$. This proves the first part of the lemma.

    For the second part, we use Proposition \ref{prop:metric-rank}.7:
    \[d(xz,yz)\leq d'(xz,yz)\leq d'(x,y)\leq 2\,d(x,y).\qedhere\]
\end{proof}

\begin{remark}\label{rmk:meet-ctrl-fails}
    Let 
    \begin{equation}\label{eq:modular-pair}
        \vp(x,y) := \inf_z\max\{(|x| + |y|) \dminus (|x+y| + |z|), d(x+z,x), d(y+z,y)\},
    \end{equation}
    noting that for complete metric lattices $\vp(x,y)=0$ is equivalent to $(x,y)$ being a metrically modular pair by Proposition \ref{prop:metrically-modular-axiom}.  In light of the previous lemma, it is natural to ask whether $\vp(x,z),\vp(y,z)<\e$ implies that $d(xz,yz)\leq 2d(x,y)+K\e$ for some fixed constant $K$. The answer turns out to be `no'.
    Indeed, we will show that for every $K$ there is a finite partition lattice $P_\bullet$, a real number $\varepsilon>0$, and partitions $x,y,z\in P_\bullet$ such that $\varphi(x,z),\varphi(y,z)<\varepsilon$ and $d(xz,yz)>2d(x,y)+K\varepsilon$.
    
    Given $K$, take a natural number $n> K+2$, set $\varepsilon=\frac{1}{3n-2}$, and consider the finite partition lattice of the set $\{1,2,3,...,3n\}$, which we will note as $P_{3n}$. Here the norm of a partition $x$ is $|x|:=\frac{3n-\#x}{3n-1}$ where $\#x$ is the number of blocks in the partition $x$, and the distance of two partitions is $d(x,y):=2|x+y|-|x|-|y| = d'(x,y)$.

    Consider the following partitions
    \begin{align*}
        x&:=\{\{1,2,3,...,2n\},\{2n+1,...,3n\}\}\\
        y&:=\{\{1,2,3,...,n\},\{n+1,n+2,...,3n\}\\
        z&:=\{\{i,i+n,i+2n\}\}_{i=1}^n.
    \end{align*}
    We compute the partitions $x+z,y+z, x+y, xz, yz$ and $xz+yz$ which we will need for later.
    \begin{align*}
        x+z&=\{\{1,2,...,3n\}\}=y+z=x+y\\
        xz&=\{\{i,i+n\}\}_{i=1}^n\cup\{\{j\}\}_{j=2n+1}^{3n}\\
        yz&=\{i\}_{i=1}^n\cup\{\{j,j+n\}\}_{j=n+1}^{2n}\\
        xz+yz&=z.
    \end{align*}
    It is clear that $|x+z|=|y+z|=|x+y|=1$. Further,
    \[|x|=|y|=\frac{3n-2}{3n-1},\quad
        |z| =\frac{2n}{3n-1},\quad
        |xz|=|yz|=\frac{n}{3n-1}.\]
    Now, let's check that $\varphi(x,z),\varphi(y,z)<\varepsilon$. Observe that $|x|+|z|\leq |x+z|+|z|$, \[ |x+z|-|x|=1-\frac{3n-2}{3n-1}=\frac{1}{3n-1},\] and $|z+z|-|z|=0$, so $\vp(x,z)< \frac{1}{3n-1}<\e$.
    Similarly $\varphi(y,z)<\varepsilon$.
    However, we have that
    $$d(xz,yz)=2|xz+yz|-|xz|-|yz|=\frac{4n}{3n-1}-\frac{n}{3n-1}-\frac{n}{3n-1}=\frac{2n}{3n-1}$$
    and 
    \begin{equation*}
        \begin{split}
            2d(x,y)+K\varepsilon &=2(2|x+y|-|x|-|y|)+K\varepsilon\\
            &=2\left(2-\frac{3n-2}{3n-1}-\frac{3n-2}{3n-1}\right)+K\varepsilon\\
            &=\frac{4}{3n-1}+\frac{K}{3n-2}.
        \end{split} 
    \end{equation*}
    Since $n>K+2$, we conclude that
    $$d(xz,yz)=\frac{2n}{3n-1}>\frac{4}{3n-1}+\frac{2K}{3n-1}>\frac{2}{3n-1}+\frac{K}{3n-2}=2d(x,y)+K\varepsilon.$$

    In general picking $n>K+S$ we can use this counterexample to get that $\varphi(x,z)<\varepsilon, \varphi(y,z)<\varepsilon$ and $d(xz,yz)>Sd(x,y)+K\varepsilon$.
\end{remark}

\begin{lem}\label{lem:meet-ultrp}
    Let $\{\cc M_i: i\in I\}$ be a collection of metrically modular complete metric lattices, and $\cc U$ a non-principal ultrafilter on $I$. Let $\cc M = \prod_{i\in \cc U} \cc M_i$ be the ultraproduct. For $x,y\in \cc M$ with representatives $x = (x_i)$ and $y = (y_i)$ we have that $xy = (x_iy_i)$.
\end{lem}

\begin{proof}
    Clearly $x_iy_i\leq x_i,y_i$, hence $(x_iy_i)\leq x,y$ and $(x_iy_i)\leq xy$. In the other direction, suppose that $z\leq x,y$ with representative $z = (z_i)$. We therefore have that $d_i(x_i+z_i,x_i)\to 0$ and $d(y_i+z_i,y_i)\to 0$ as $i\to \cc U$. By metric modularity, we have that 
    \[d_i(x_iz_i,z_i) = |z_i| - |x_iz_i| = |x_i+z_i| - |x_i| = d_i(x_i+z_i, z_i)\]
    with the corresponding equations holding for $d_i(y_iz_i,z_i)$. By Lemma \ref{lem:metric-mod-contraction} and the triangle inequality 
    \begin{equation*}
        d_i(x_iy_iz_i,z_i)\leq d_i(x_iy_iz_i,y_iz_i) + d_i(y_iz_i,z_i)\leq 2(d_i(x_iz_i,z_i) + d_i(y_iz_i,z_i))\to 0
    \end{equation*}
    as $i\to \cc U$. Therefore, $z= (x_iy_iz_i)\leq (x_iy_i)$ so $xy\leq (x_iy_i)$.
\end{proof}

 In order to show the next result we will invoke the Beth Definability Theorem which we discussed in the preliminaries.

%Using this Corollary we can finally prove the following proposition.
\begin{prop}\label{prop:meet-uc}
    The predicate $P(x,y,z) = d(xy,z)$ is definable for the elementary class $\Mod(T_{MML})$ of metrically modular complete metric lattices.
\end{prop}

\begin{proof}
    First, we prove that for all $\mathcal{M}\in \Mod(T_{MML})$ the predicate $P^\mathcal{M}(x,y,x):=d(xy,z)$ satisfies the conditions of Corollary \ref{cor:beth-expanded-language}. Observe that given $x,y,z,w\in \mathcal{M}$ we can apply the Lemma \ref{lem:metric-mod-contraction} plus the triangle inequality to get
    \[d(xy,zw)\leq d(xy,yz)+d(yz,zw)\leq 2d(x,z)+2d(y,w).\]
    Therefore, the meet operation is uniformly continuous. Now note that if we take 
    \[x_1,x_2,y_1,y_2,z_1,z_2\in\mathcal{M},\] using the triangle inequality we obtain
    $$d(x_1y_1,z_1)\leq d(x_1y_1,z_2)+d(z_1,z_2)\leq d(x_1y_1,x_2y_2)+d(x_2y_2,z_2)+d(z_1,z_2)$$
    and 
    $$d(x_2y_2,z_2)\leq d(x_2y_2,z_1)+d(z_1,z_2)\leq d(x_1y_1,x_2y_2)+d(x_1y_1,z_1)+d(z_1,z_2);$$
    therefore, 
    $$|d(x_1y_1,z_1)-d(x_2y_2,z_2)|\leq d(x_1y_1,x_2y_2)+d(z_1,z_2).$$
    Finally, we have that
    $$|P^\mathcal{M}(x_1,y_1,z_1)-P^\mathcal{M}(x_2,y_2,z_2)|\leq 2\left(d(x_1,x_2)+d(y_1,y_2)\right)+d(z_1,z_2).$$
    Since the predicates $P^\mathcal{M}$ are thus uniformly continuous with common modulus of continuity for every $\mathcal{M}\in \Mod(T_{MML})$, we can take the class 
    \[\mathcal{C}':=\{(\mathcal{M},P^\mathcal{M}):\mathcal{M}\in \Mod(T_{MML})\}\]
    of structures for our expanded language with a predicate for $P$, which is interpreted in each $\mathcal{M}\in \Mod(T_{MML})$ as $P^\mathcal{M}$. The class $\mathcal{C}'$ is axiomatizable by the following theory in our expanded language
    \[T':=T_{ML}\cup\left\{\sup_x\sup_y\left(d(0,x+y)+P(x,y,0)-d(x,0)-d(y,0)\right)\right\},\] so by Corollary \ref{cor:beth-expanded-language} we conclude that the meet function is axiomatizable in the theory of metrically modular complete metric lattices. \qedhere
\end{proof}

    In fact, we can can show that $d(z,xy)$ is bounded above by an explicit formula in $(x,y,z)$, which directly shows that $d(z,xy)$ is definable.
    We thank the reviewer for the next result and its proof.
    
\begin{prop}
    If $\cc M$ is a metrically modular complete metric lattice, then for all $x,y,z\in M$ we have 
    $$d(z, xy)\leq 4d(x, x+z)+4d(y,y+z)+(|x|+|y|\dminus (|x+y|+|z|)).$$
\end{prop}

    \begin{proof}
        By the proof of Lemma \ref{lem:meet-ultrp} and Proposition \ref{prop:meet-uc} we note that
        $$d(xyz,z)\leq 2d(xz,z)+d(yz,z)=2d(x+z,x)+2d(y+z,y),$$
        also
        $$d(xyz,xy)=|xy|-|xyz|=|xy|-|z|+|z|-|xyz|=(|xy|\dminus|z|)+d(z,xyz).$$
        Thus,
    \begin{align*}
        d(z,xy)&\leq d(xyz,z)+d(xyz,xy) \\
        &\leq 2d(xyz,z)+(|xy|\dminus|z|)\\
        &\leq 4d(x,x+z)+4d(y,y+z)+(|x|+|y|\dminus(|x+y)+|z|)). \qedhere
    \end{align*}

    \end{proof}

%%%%%%%%%%%%%%%%%%%%%%%%%%%%%%%%%%%%%%%%%%%%%%%%%%%%%%%%%%%%%%%%%%
\subsection{Axioms for distributive and complemented lattices}
%%%%%%%%%%%%%%%%%%%%%%%%%%%%%%%%%%%%%%%%%%%%%%%%%%%%%%%%%%%%%%%%

Since the predicate $d(xy,z)$ is $T_{MML}$-equivalent to a definable predicate or formula in $\ff L$ (due to the general definition of formula we are using), we know there is formula $\phi_m$ such that for every model $\cc M$ of $T_{MML}$ and all $x,y,z\in M, \phi_m(x,y,z)=d(xy,z)$. Using this we can axiomatize the theory of distributive metrically modular complete metric lattices.

To do so, we define the sentence
\begin{equation}
    \sigma_{\text{dist}}:=\sup_{x,y,z}\inf_{t,w} \max\{\phi_m(x,y,t),\phi_m(x,z,w),\phi_m(x,y+z,t+w)\}
\end{equation}
%and
%\[\sigma_{dist2}:=\sup_{x,y,z}\inf_w\max\{\phi_m(y,z,w),\phi_m(x+y,x+z,x+w)\}\]
and consider the theory 
\begin{equation}\label{eq:t-dml}
    T_{DML}:=T_{MML}\cup\{\sigma_{\text{dist}}\} = T_{ML} \cup\{\sg_{\text{mod}},\sg_{\text{dist}}\}.
\end{equation}

\begin{prop}\label{prop:t-dml-language}
    Let $\cc M$ be a $\ff L$-structure, then $\cc M$ models $T_{DML}$ if and only if it is a distributive metrically modular complete metric lattice.
\end{prop}

\begin{proof}
    It is clear that if $\cc M$ is a distributive metrically modular complete metric lattice, then it is a model of $T_{DML}$. Now let's suppose that $\cc M$ is a model of $T_{DML}$. By Proposition \ref{prop:t-mml-language} we know that $\cc M$ is a metrically modular complete metric lattice, so we just need to prove that $\cc M$ has the distributive property. 
    
    Let us fix $x,y,z\in \cc M$. Since $\cc M$ models $T_{DML}$, it satisfies $\sigma_{\text{dist}}$, so for every $n\in \bb N$ there are a $t_n$ and a $w_n$ such that 
    \[\max\{ d(xy,t_n), d(xz,w_n),d(x(y+z),t_n+w_n)\}<\frac{1}{n}.\]
    Therefore, $t_n\rightarrow xy$, $w_n\rightarrow xz$, and $t_n+w_n\rightarrow xy+xz$ as $n\rightarrow\infty$, where the first two limits are immediate and the last follows from these by uniform continuity of the join operation (Remark \ref{rmk:uniform-continuity}).
    Finally we observe that for every $n\in\bb N$ 
    \[d(x(y+z),xz+xz)\leq d(x(y+z),t_n+w_n)+d(t_n+w_n,xy+xz);\]
    thus, $d(x(y+z),xy+xz)=0$,
    or equivalently $x(y+z)=xy+xz$.
    This implies that $x+yz=(x+y)(x+z)$; thus, $\cc M$ is a distributive metrically modular complete metric lattice. 
\end{proof}

\begin{defn}
    For $\cc L$ a lattice, we say that $x\in\cc L$ is \emph{complemented} if there exist a $y\in\cc L$ such that $x+y=1$ and $xy=0$. We refer to $y$ as a complement of $x$. We say that $\cc L$ is a \emph{complemented lattice} if every element of $\cc L$ is complemented.
\end{defn}

\begin{lem}\label{lem:metric-complement}
    Let $\cc L$ be a metric lattice, and fix $x\in L$. If $d'(x,y)=1$, then $y$ is a complement of $x$. If $y\in L$ is metrically modular, then the converse is also true. 
\end{lem}

\begin{proof}
    If $d'(x,y)=1$, then $1=d'(x,y)=2|x+y|-|x|-|y|\leq|x+y|-|xy|$; thus, $1+|xy|\leq|x+y|\leq 1$. We conclude that $|x+y|=1$ and $|xy|=0$,\; therefore, $x+y=1$ and $xy=0$.

    Since $y$ is metrically modular, then $d'(x,y)=|x+y|-|xy|$. Thus, if $y$ is a complement of $x, d'(x,y)=|x+y|-|xy|=1-0=1$. 
\end{proof}

\begin{remark}
    For the converse statement in the previous result, the assumption of metric modularity is necessary. 
    To see this, consider the finite partition lattice of a set with $4$ elements, $P_4$, and the partitions 
    \[x:=\{\{1,2\},\{3,4\}\}, \quad y:=\{\{1,3\},\{2,4\}\}.\]
    Since $x+y=1,xy=0, |x|,|y|=\frac{4-2}{3}=\frac{2}{3},$ and $|x+y|+|xy|=1<\frac{4}{3}=|x|+|y|$, we observe that $P_4$ is not metrically modular and $y$ is a complement of $x$. However, $d(x,y)=2|x+y|-|x|-|y|=2-\frac{2}{3}-\frac{2}{3}=\frac{2}{3}<1$.
\end{remark}

For the following result we need this lemma.

\begin{lem}\label{lem:distance-product-estimate}
    Let $\cc M$ be a distributive metrically modular complete metric lattice. If $(x_i)$ is a sequence in $\cc M$, then setting $z=\Pi_i x_i$ we have that
    $$d(y+z,1)\leq 2\sum_i d(y+x_i,1)$$
    and 
    $$d'(y+z,1)\leq \sum_i d'(y+x_i,1)$$
    if the sum converges.
\end{lem}

\begin{proof}
    Setting $z_n=\Pi_{i=1}^n x_i$, we have that $(z_n)$ is a decreasing sequence, so $(z_n)$ is Cauchy and $z_n\rightarrow z$ by Lemma \ref{lem:net-convergence}. Using Lemma \ref{lem:metric-mod-contraction} and the distributive property we have inductively that
    \begin{align*}
        d(z_n+y,1) &= d(z_{n-1}x_n+y,1) \\
        &= d((z_{n-1}+y)(x_n+y),1)\\
        &\leq d((z_{n-1}+y)(x_n+y),z_{n-1}+y)+d(z_{n-1}+y,1)\\
        &\leq d(z_{n-1}+y,1)+2d(x_n+y,1)\leq 2\sum_{i=1}^n d(x_i+y,1)\, ;
    \end{align*}
    hence, the result obtains in the limit by continuity. The argument works the same with $d'$ replacing $d$.
    %\begin{align*}
    %    d(z_n+y,1) &= d(z_{n-1}x_n+y,1) \\
    %    &= d((z_{n-1}+y)(x_n+y),1)\\
    %    &\leq d((z_{n-1}+y)(x_n+y),z_{n-1}+y)+d(z_{n-1}+y,1)\\
    %    &\leq d(z_{n-1}+y,1)+d(x_n+y,1)\leq \sum_{i=1}^n d(x_i+y,1)
    %\end{align*}
%    
\end{proof}

\begin{defn}
    Let $\cc L$ be a metric lattice. We say that $x\in L$ is \emph{weakly complemented} if $\sup_y d'(x,y)=1$. We say that $\cc L$ is \emph{weakly complemented} if $\inf_x\sup_y d'(x,y)=1$.
\end{defn}

\begin{prop}\label{prop:weak-complement}
    Let $\cc M$ be a distributive, metrically modular complete metric lattice. We have that $x\in\cc M$ is weakly complemented if and only if $x$ is complemented, and that complement is unique.
\end{prop}

\begin{proof}
    We know that complemented implies weakly complemented by Lemma \ref{lem:metric-complement}, so we only need to prove the converse.

    Since $x\in \cc M$ is weakly complemented, we can choose for every $i\in \bb N$ a $z_i\in \cc M$ such that $1-2^{-i}<d'(x,z_i)\leq 1$. Given that for every metric lattice the inequality $d'(x,y)\leq |x+y|-|xy|\leq 1$ holds, we can conclude that for every $i\in \bb N, 1-2^{-i}<|x+z_i|-|xz_i|$, and thus, 
    \begin{equation*}
        0\leq d'(1,x+z_i)=1-|x+z_i|<2^{-i}-|xz_i|=2^{-i}-d'(xz_i,0)\leq 2^{-i}.
    \end{equation*} 
    Therefore, for every $n\in\bb N$, we have that $d'(x+z_n,1),\, d'(xz_n,0)<2^{-n}$. Let's construct $t=\liminf z_i$.

    We define $t_k^N:=\Pi_{i=k}^N z_i$ and $t_k:=\Pi_{i=k}^\infty z_i$, noting that $t_k^N\rightarrow t_k$ as $N\rightarrow \infty$ by Lemma \ref{lem:net-convergence}. For $k<l$ and $M>N$ we have that $d'(t_k^Mt_l^N, t_k^M)=0$, so $\lim_{M\rightarrow\infty}d'(t_k^Mt_l^N,t_k^M)=d'(t_kt_l^N,t_k)=0$ when $k<l$ by Lemma \ref{lem:metric-mod-contraction}. Now, taking the limit as $N$ tends to $\infty$, we have that $d'(t_kt_l,t_k)=0$ when $k<l$, so $(t_k)$ is increasing, therefore Cauchy. Setting $t=\lim_k t_k$, from $t\geq t_k$ and Lemma \ref{lem:distance-product-estimate}, we observe that for every $k\in\bb N$ we have that 
    \begin{equation*}
        d'(x+t,1)\leq d'(x+t_k,1)\leq \sum_{i=k}^\infty d'(x+z_i,1)<\sum_{i=k}^\infty 2^{-i}.
    \end{equation*}
    Using that $t_k\leq z_k$ for every $k$ we get
    \[d'(xt,0)\leq d'(xt,xt_k)+d'(xt_k,0)\leq d'(t,t_k)+d'(xz_k,0).\]
    In either case the right side of the inequality tends to $0$ as $k$ tends to $\infty$, therefore $d(x+t,1)=d(xt,0)=0$. In other words $x+t=1$ and $xt=0$, which shows that $t$ is a complement of $x$. To prove uniqueness suppose $t'$ is another complement of $x$. Using the distributive property we have that
    \[t=1t=(x+t')t = xt+t't = t't = t't+xt' = t'(x+t) = 1t' =t',\]
    and the result obtains.
\end{proof}

\begin{remark}
    There are cases where complements \emph{do not} exist. For example, consider the lattice $L_n:=B_n\cup\{1^+\}$ where $B_n$ is the boolean algebra with $n$ atoms, with minimum $0$, maximum $1$ and $1^+$ is an additional element above $1$. By setting, $|\cdot|: L_n\to[0,1]$ with 
    \[
        |a|:=\frac{\#\text{ number of atoms below } a}{n+1}
    \]
    for all $a\in B_n$, and $|1^+|=1$, we have that $L_n$ is a distributive, metrically modular complete metric lattice where the only elements with complements are $0$ and $1^+$. Also note that in this example $d=d'$ and $\inf_x\sup_y d'(x,y)=\frac{n}{n-1}<1$.
\end{remark}

%\begin{remark}
%    Relatedly, there are metrically modular lattices where every element is uniquely complemented but which are not distributive. For instance one can take the lattice of orthogonal projections of a finite-dimensional inner product space or, more generally, a tracial von Neumann algebra $(M,\tau)$ with $|p| := \tau(p)$. A projection $p\in M$ is distributive if and only if it belongs to the center.
%\end{remark}

\begin{remark}\label{rem:submod-choquet}
    Consider the inequality 
    \[|x+y+z| + |x| + |y| + |z| \geq |x+y| + |x+z| + |y+z|,\] which we may rephrase as $2|x+y+z|\geq d'(x,z) + d'(y,z) + d'(x,y)$. It is easy to see that this implies that if both $x$ and $y$ complement $z$, then $x=y$. 
    
    Uniqueness of complements does not hold for metric lattices in general (see Lemma \ref{lem:partitions-have-selectors}), though by \cite{dilworth} every lattice is a sublattice of a uniquely complemented lattice. Consider the partition lattice $P_3$, which is not uniquely complemented. In this case the inequality also fails to hold. Consider $x_1:=\{\{1\},\{2,3\}\}, x_2:=\{\{2\},\{1,3\}\}$ and $x_3:=\{\{3\},\{1,2\}\}$, observe that $|x_1|=|x_2|=|x_3|=\frac{1}{2}$ and $|x_1+|x_2|=|x_1+x_3|=|x_2+x_3|=|x_1+x_2+x_3|=1$. Therefore, \[|x_1+x_2+x_3|+|x_1|+|x_2|+|x_3|=\frac{5}{2}<3=|x_1+x_2|+|x_2+x_3|+|x_3+x_1|.\]
    
\end{remark}

\begin{question}
      For each $S\subseteq [[n]]$, we write $x_S := \sum_{i\in S} x_i$, with $x_\emptyset := \prod_i x_i$. For which lattices does the ``inclusion/exclusion'' inequality hold that
    \[\sum_{S\in 2^{[[n]]}} (-1)^{|S|} |x_S|\geq 0?\]
    This condition for an $n$-tuple is witnessed by 
    \[
        \sup_{\bar x}\inf_z\left(\sum_i d(x_i+z,x_i) + \left(-|z| - \sum_{\emptyset\not=S\in 2^{[[n]]}} (-1)^{|S|} |x_S|\right)_+\right) =0,
    \]
    so defines an elementary class. This condition is formally dual, exchanging meets and joins, with Hrushovski's notion of ``flatness'' for pregeometries \cite{hrushovski}.
\end{question}

%%%%%%%%%%%%%%%%%%%%%%%%%%%%%%%%%%%%%%%%%%%%%%%%%%%%%%%%%%%%%%%%%%%
\subsection{Metric lattices and Boolean algebras}
%%%%%%%%%%%%%%%%%%%%%%%%%%%%%%%%%%%%%%%%%%%%%%%%%%%%%%%%%%%%%%%%%%%

%\begin{remark}
%\label{rem:boolean-algebra}
    A lattice which is distributive and complemented is referred to as a \emph{Boolean algebra}. Thus, it follows from the last result that a distributive metrically modular metrically complemented complete metric lattice is a Boolean algebra.
    %A metric lattice just needs to be distributive and complemented as a lattice to be a Boolean Algebra.     
%\end{remark}

\begin{notation}\label{not:complement}
    For a Boolean algebra $\cc A$ we will denote the (unique) complement of $x\in\cc A$ by $x^c$.
    %Given a lattice $\cc L$, if for every $x\in\cc L$ the complement exists and is unique we will denote it by $x^c$.
\end{notation}

\begin{lem}\label{lem:complement-isometry}
    Let $\cc M$ be a distributive, metrically modular complete metric lattice. For every $x,y\in\cc M$ we have that $d'(x^c,y^c) = d'(x,y)$.
\end{lem}

\begin{proof}
    From the definition, we know that $\cc M$ is a Boolean algebra. For $x\in\c M$ we observe by metric modularity that $1=|x+x^c|-|xx^c|=|x|+|x^c|$. Therefore, by metric modularity and De Morgan's laws we get that
    \begin{align*}
        d'(x,y)&= |x+y|-|xy|= (1-|xy|)-(1-|x+y|)=\\
        &=|(xy)^c|-|(x+y)^c|=|x^c+y^c|-|x^cy^c|=d'(x^c,y^c). \qedhere
    \end{align*}
\end{proof}

Now let's consider the sentences
\begin{equation}
    \sg_{\text{wcom}} := \inf_x\sup_y (1-d(x,y))
\end{equation}
and
\begin{equation}
    \sigma_{d=d'}:=\sup_{x,y}|d(x,y)-d'(x,y)|
\end{equation}
and the theory 
\begin{equation}
    T_{BML}:=T_{DML}\cup\{\sigma_{\text{wcom}}, \sigma_{d=d'}\}.
\end{equation}

Next, we will prove that this theory and the following $L^{pr}$-theory $\mathsf{PR}$ found in \cite[Section 4]{Mtfps} are equivalent.

\begin{defn}
    The continuous signature $L^{pr}$ is defined as follows.
    \begin{itemize}
        \item $\mu$ as an unitary predicate symbol with modulus of uniform continuity $\Delta_\mu(\varepsilon)=\varepsilon$.
        \item $\cap$ and $\cup$ as binary operations with modulus of uniform continuity $\Delta_\cap(\varepsilon)=\Delta_\cup(\varepsilon) =\frac{\varepsilon}{2}$. $\cdot^c$ as an unitary operation symbol with modulus of uniform continuity $\Delta_c(\varepsilon)=\varepsilon$.
        \item $1$ and $0$ as constant symbols.
    \end{itemize}
    The $L^{pr}$-theory is formed by the following $L^{pr}$-conditions:
    \begin{itemize}
        \item Boolean algebra axioms: As stated in \cite[Section 4]{Mtfps}, each one of these is the $\forall$-closure of an equation between terms, so it can be expressed in our setting as a condition. Example: From the axiom $\forall x,y (x\cap y= y\cap x)$ we get the condition $\sup_{x,y}d(x\cap y,y\cap x)$.
        \item Measure axioms:
        \begin{itemize}
            \item $\mu(0)=0$ and $\mu(1)=1$
            \item $\sup_{x,y}(\mu(x\cap y)\dminus \mu(x))$
            \item $\sup_{x,y}(\mu(x)\dminus\mu(x\cup y))$
            \item $\sup_{x,y}|(\mu(x)\dminus\mu(x\cap y))-(\mu(x\cup y)\dminus\mu(y))|$           
        \end{itemize}
        \item Identification of $d$ and $\mu$: $\sup_{x,y}|d(x,y)-\mu(x\Delta y)|$
    \end{itemize}
\end{defn}
We are considering $\ff L\subset L^{pr}$ where the constant symbols are the same and $+$ corresponds to the $\cup$ symbol.

\begin{prop}\label{prop:boolean-elementary}
    If $\cc M$ is a $\ff L$-structure, then $\cc M$ is a model of $T_{BML}$ if and only if $M$ (viewed as a $L^{pr}$-structure) is a model of $\mathsf{PR}$.
\end{prop}

\begin{proof}
    Let's suppose that $\cc M\models T_{BML}$. First let's check that $\cc M$ can be seen as a $L^{pr}$-structure. It is clear that $\cc M$ has a single sort, diameter $1$, a couple of distinguished elements $0,1$ and a function $+$ with modulus of uniform continuity $\Delta(\varepsilon)=\varepsilon/2$, which is the interpretation of $\cup$. We need to find interpretations for $\cap,\cdot ^c$ and the measure $\mu$. Using Proposition \ref{prop:metric-rank}.1, Lemma \ref{lem:metric-mod-contraction}, and Lemma \ref{lem:complement-isometry} it is easy to check that $|\cdot|$, $\cdot$, and $\cdot ^c$ each have the modulus of continuity necessary to take them as interpretations of $\mu$, $\cap$, and $\cdot^c$ respectively. Thus, we can view $\cc M$ as a $L^{pr}$ structure.

    We now need to check that $\cc M$ interpreted as an $L^{pr}$ structure satisfies the axioms of $\mathsf{PR}$. Since $\cc M$ is a Boolean algebra, it is clear that it satisfies the axioms about Boolean algebras, it also satisfies the axiom about the connection between $\mu$ and $d$ because by $\sigma_{d'}$ we get that for every $x,y\in\cc M$
        \[d(x,y)=d'(x,y)=|x+y|-|xy|=|x\Delta y|,\]
        where $x\Delta y:= xy^c+x^cy$
        
        For the measure axioms, we can check that $|0|=0$ and $|1|=1$. By Proposition 2.4, the sentence $\sigma_{\text{mod}}$ and the behavior defined for $+$ and $\cdot^c$ it follows that $\cc M$ satisfies the axioms of measure. Therefore, it is a model of $\mathsf{PR}$.\\

        Now let's suppose that $\cc M$ is a $L^{pr}$-structure. It is clear that we can view $\cc M$ as a $\ff L$ structure, when we forget about the interpretations of $\cap,\cdot^c$ and $\mu$. (Note, however, that from the axiom $(3)$ of $\mathsf{PR}$ it follows that $\mu=|\cdot|$.) Now let's check that it satisfies the axioms of $T_{BML}$. We proceed sequentially.
        \begin{itemize}
            \item [$T_{ML}$)] The first five axioms follow easily from the axioms of $\mathsf{PR}$, the sixth axiom follows from the inequality $(x+z)\Delta(y+z)\leq x\Delta y$ in Boolean algebras plus the measure and metric axioms of $\mathsf{PR}$ theory. The last axiom follows from the modularity expressed by the measure axioms and Proposition \ref{prop:semimodular-upgrade}. Hence $\cc M$ satisfies $T_{ML}$.
            \item [$T_{MML})$] From the measure axioms it is clear that $\cc M$ is metrically modular; thus, by Proposition \ref{prop:t-mml-language} it models $T_{MML}$
            \item [$T_{DML}$)] This follows from Proposition \ref{prop:t-dml-language} and the Boolean algebra axioms.
            \item [$T_{BML}$)] This follows from metric modularity, the Boolean algebra axioms, and Proposition \ref{prop:weak-complement}. \qedhere

        \end{itemize}

\end{proof}

\begin{prop}\label{Prop:clousure of modularity}
    Let $\cc M$ be a complete metric lattice. If $\cc N\subseteq \cc M$ is a sublattice which is either: (1) metrically modular, distributive, and metrically complemented; or (2) Boolean, then $\overline{\cc N}$ is a sublattice with the same properties.
\end{prop}

\begin{proof}
    First, let's prove that every $\overline{\cc N}$ is metrically modular. Given $x,y\in\overline{\cc N}$, we know by uniform continuity of the $\ff L$-formula $\vp$ defined in equation (\ref{eq:modular-pair}), that for every $\ve>0$ there is a $\delta>0$ such that if $d(x,x'),d(y,y')<\delta$, then $|\vp(x,y)-\vp(x',y')|<\ve$. Since $x,y\in \overline{\cc N}$ and $\cc N$ is metrically modular, for each $\ve>0$ there are some $x_\ve,y_\ve\in\cc N$ such that $d(x,x_\ve),d(y,y_\ve)<\delta$ and $\vp(x_\ve,y_\ve)=0$. It follows from this that $|\vp(x,y)-\vp(x_\ve,y_\ve)|=|\vp(x,y)|<\ve$ for every $\ve>0$, so we conclude that $\vp(x,y)=0$ for any $x,y\in\overline{\cc N}$, as desired.

    Now observe that for any $x,y\in \overline{\cc N}$, there are Cauchy sequences $(x_n),(y_n)$ such that $x_n\to x$ and  $y_n\to y$. From the uniform continuity of the join operation on $\cc M$ it follows that $x_n+y_n\to x+y$ on $M$, in addition $x_n+y_n\in \cc N$ for each $n$ since $\cc N$ is a sublattice. So we can conclude that $x+y\in\overline{\cc N}$. Though the meet operation isn't necessarily uniformly continuous, by \ref{prop:meet-uc} we know that it is uniformly continuous on metrically modular sets. Since $\overline{\cc N}$ is metrically modular and $x,y,x_n,y_n\in\overline{\cc N}$, we have that $x_ny_n\to xy$ and $x_ny_n\in \cc N$ for each $n$, concluding that $xy\in\overline{\cc N}$. So $\overline{\cc N}$ is a sublattice.

    Finally, let's prove that it is distributive and metrically complemented. Taking $x,y,z\in\overline{\cc N}$ and sequences $(x_n),(y_n),(z_n)$ on $\cc N$ such that $x_n\to x, y_n\to y$ and $z_n\to z$ it is clear by the uniform continuity of the join and meet operations that $x_n(y_n+z_n)\to x(y+z)$ and $z_ny_n+x_nz_n\to xy+xz$. However, as $\cc N$ is boolean and $x_n,y_n,z_n\in\cc N$ for each $n$, so it follows that $x_n(y_n+z_n)=x_ny_n+x_nz_n$ and we conclude that $x(y+z)=xy+xz$ for any $x,y,z\in\overline{\cc N}$, or in other words that $\overline{\cc N}$ is distributive. On the other hand, we know that $\sup_y d(x,y)=1$ for every $x\in\cc N$, as that expression is uniformly continuous, we get that $\sup_y d(x,y)=1$ for every $x\in\overline{\cc N}$. We conclude that $\overline{\cc N}$ is also metrically complemented, and therefore the proposition is proved.  \qedhere

\end{proof}

%%%%%%%%%%%%%%%%%%%%%%%%%%%%%%%%%%%%%%%%%%%%%%%%%%%%%%%%%%%%%%%%

\section{Partition Lattices}\label{sec:partition}
\subsection{Finite partition lattices}
We begin by fixing some notation and terminology which will be crucial for this section.

\begin{defn}
    Let $P_n$ denote the lattice of partitions of the set $[[n]] := \{1,\dotsc,n\}$ with the following order: Given $x,y$ partitions of $[[n]]$, $x\leq y$ if and only if for all $A\in x$ there is a $B\in y$ such that $A\subset B$. Given a partition $x\in P_n$ we define
    \begin{enumerate}
        \item $\langle x\rangle$ as the number of blocks of $x$ that are singletons,
        
        \item $[x]$ as the number of blocks of $x$ that have at least two elements,

        \item $\#x = \langle x\rangle + [x]$ as the number of blocks of $x$,
        
        \item $|x|=\frac{n-\#x}{n-1}$ as the norm of $P_n$, and

        \item for $S\subseteq [[n]]$, $i(x,S)$ as the number of blocks of $x$ which have non-empty intersection with $S$. (Note $i(x,[[n]]) = \#x$.)
    \end{enumerate}
        
    \noindent If $[x]\leq 1$, then $x$ is called a \emph{singular partition}. 
    Every non-zero singular partition is defined by its unique block with more than two elements, which we will call the \emph{basic block} of $x$, in which case we will define $B_x\subseteq [[n]]$ to be the basic block of $x$. By convention $B_0 = \{1\}$ is the basic block of $0$. We will denote the set of singular partitions in $P_n$ by $\Sg_n$.
\end{defn}

Recall that $P_n$ can be viewed as a metric lattice when equipped with the metric
\[d(x,y) := \frac{\#x + \#y - 2\#(x+y)}{n-1}=d'(x,y).\]
We will implicitly assume that $P_n$ is equipped with this metric.

The following formulas are readily apparent.
\begin{lem}\label{lem:singular-partition-lemma}
    For $x,y\in P_n$ with $y$ singular, we have that
    \begin{equation*}
        \#(x+y)=\#x-i(x,B_y)+1
    \end{equation*}
    and 
    \begin{equation*}
        \#xy = \#y +i(x,B_y) -1.
    \end{equation*}
\end{lem}

\begin{lem}\label{lem:distance-to-singular}
    For any partition $x\in P_n$ we have that
    \[d(x,\Sg_n) := \inf_{y\in \Sg_n} d(x,y) = \frac{[x]-1}{n-1}.\]
\end{lem}

\begin{proof}
    %In the partition lattice $P_n$ the distance between two partitions $x,y$ are given by the formula $d(x,yy)=2|x+y|-|x|-|y|=2\frac{n-\#(x+y)}{n-1}-\frac{n-\#x}{n-1}-\frac{n-\#y}{n-1}=\frac{\#x+\#y-2\#(x+y)}{n-1}$. 
    We fix $x\in P_n$. If $y$ is a singular partition, then by Lemma \ref{lem:singular-partition-lemma} we get that
    \[d(x,y)=\frac{\#y-\#x+2i(x,B_y) -2}{n-1}.\] We will use this formula to show that a closest singular partition to $x$ is the one whose basic block is the union of all of the non-singleton blocks of the partition $x$. 
    
    Let's denote this partition by $\pi_x$. Since $x\leq \pi_x$, it is clear that \[d(x,\pi_x)=\frac{\#x-\#\pi_x}{n-1}=\frac{[x]-1}{n-1}.\] 
    Setting $B_{\pi_x}$ to be the basic block of $\pi_x$, for another singular partition $y$ with basic block $B_y$, let $k$ be the number of elements in $B_{\pi_x}\setminus B_y$ and $l$ be the number of elements in $B_y\setminus B_{\pi_x}$. Since $B_{\pi_x}$ is the disjoint union of blocks of $x$ of size at least two, the removal of the elements $B_{\pi_x}\setminus B_y$ decreases $i(x,B_{\pi_x})$ by at most $\lfloor k/2 \rfloor$, while the addition of elements from $B_y\setminus B_{\pi_x}$ increases $i(x,B_{\pi_x})$ by exactly $l$ since every block of $x$ intersecting $B_y\setminus B_{\pi_x}$ is a singleton. Therefore, we have that
    \[i(x,B_y) \geq i(x,B_{\pi_x}) + l - \lfloor k/2 \rfloor.\]
    It follows that 
    \begin{equation*}
        \begin{split}
            d(x,y) &= \frac{\#y-\#x+2i(x,B_y) -2}{n-1}\\
            &= \frac{(\#\pi_x +k -l) - \#x + 2i(x,B_y) -2}{n-1}\\
            &\geq \frac{\#\pi_x + (k-l) - \#x + 2i(x,B_{\pi_x}) + 2(l - \lfloor k/2 \rfloor) -2}{n-1}\\
            &\geq \frac{\#\pi_x - \#x + 2i(x,B_{\pi_x})-2}{n-1} = d(x,\pi_x). 
        \end{split}
    \end{equation*}
    In this way the result obtains. \qedhere
\end{proof}

The next results will show that the set of modular partitions is exactly $\Sigma_n$ and that it is a definable set with respect to the ``theory of finite partition lattices'', $T_{FPL}$; see Definition \ref{def-T_FPL}.

\begin{lem}\label{lem:singleton-difference}
    For all partitions $x,y\in P_n$ with $x\leq y$ we have that 
    \begin{equation*}
        \langle x\rangle - \langle y\rangle\leq 2\left(\#x - \#y\right)
    \end{equation*}
\end{lem}

\begin{proof}
    We begin by noting that $x\leq y$ implies that $\#x\geq \#y$ and $\langle x\rangle \geq \langle y\rangle$, however $[x] - [y]$ can be negative. Next, observe that
    \begin{equation}
        [y] \leq [x] + \frac{\langle x\rangle - \langle y\rangle}{2}
    \end{equation}
    since a block of size at least two of $y$ either contains a block of size at least two of $x$ or is obtained by merging at least two singleton blocks of $x$. As $\langle y\rangle$ singleton blocks must remain, there are at most $\langle x\rangle - \langle y\rangle$ blocks which can be merged.
    Rearranging this equation we get that
    \begin{equation*}
        - \left(\frac{\langle x\rangle  - \langle y\rangle}{2}\right)\leq  [x] - [y],
    \end{equation*}
    hence
    \begin{equation*}
        \frac{\langle x\rangle - \langle y\rangle}{2}\leq \bigl(\langle x\rangle - \langle y\rangle\bigr) + \bigl([x] - [y]\bigr) = \#x - \#y.
    \end{equation*}
    The result thus obtains. \qedhere
\end{proof}

%\rmkts{Applying this lemma to $x = x$, $y = x+t$, and with $m \geq \#x - \#(x+t)$, we have that $\langle x\rangle - \langle x+t \rangle \leq 2m$ or $\langle x\rangle - 2m \leq \langle x+t\rangle$, which was the missing inequality below.}\\

Recall the formula 
\begin{equation*}
    \vp(x,y) := \inf_z\max\{(|x| + |y|) \dminus (|x+y| + |z|), d(x+z,x), d(y+z,y)\},
\end{equation*}
which measures how far $(x,y)$ is from being a modular pair. 

\begin{prop}\label{prop:distance-to-singular-controls-modularity}
    For any  $x\in P_n$, we have that 
    \[d(x,\Sg_n)\leq 48\sup_{y\in P_n}\vp(x,y).\]
where $\vp$ is the metric-modular pair equation defined in (\ref{eq:modular-pair})
\end{prop}

\begin{proof}
    Given a partition $x\in P_n$, by choosing two distinct elements from any block of size at least two, we can write
    \[x=\{\{a_i,b_i\}\sqcup R_i\}_{i=1}^{[x]}\cup\{\{c_i\}\}_{i=1}^{\langle x\rangle}.\] Using this notation, we define a partition $x^*$ in the following way. If $[x]$ is even, then
        \[x^*:=\{\{a_{2i-1},b_{2i}\}\sqcup R_{2i},\{a_{2i},b_{2i-1}\}\sqcup R_{2i-1}\}_{i=1}^{[x]/2}\cup\{\{c_i\}\}_{i=1}^{\langle x\rangle}.\]
    If $[x]$ is odd, then
    \[x^*:=\{\{a_{2i-1},b_{2i}\}\sqcup R_{2i},\{a_{2i},b_{2i-1}\}\sqcup R_{2i-1}\}_{i=1}^{([x]-1)/2}\cup\{\{a_{[x]},b_{[x]}\}\sqcup R_{[x]}\}\cup\{\{c_i\}\}_{i=1}^{\langle x\rangle}.\]
    
    We observe that 
    $$x+x^*=\{\{a_{2i-1}, a_{2i}, b_{2i-1},b_{2i}\}\sqcup R_{2i-1}\sqcup R_{2i}\}_{i=1}^{[x]/2}\cup\{\{c_i\}\}_{i=1}^{\langle x\rangle},$$
    if $[x]$ is even, and
    $$x+x^*=\{\{a_{2i-1}, a_{2i}, b_{2i-1},b_{2i}\}\sqcup R_{2i-1}\sqcup R_{2i}\}_{i=1}^{([x]-1)/2}\cup\{\{a_{[x]},b_{[x]}\}\sqcup R_{[x]}\}\cup\{\{c_i\}\}_{i=1}^{\langle x\rangle},$$
    if $[x]$ is odd. Thus, we can compute and note that
    \begin{equation}\label{eq:distance-to-singular-1}
        \langle x^*\rangle=\langle x\rangle,\quad \langle x+x^*\rangle=\langle x\rangle,\quad [x^*]=[x],\quad [x+x^*]=\lceil[x]/2\rceil.
    \end{equation} 
    %[x^*]=[x]$, $\langle x^*\rangle=\langle x\rangle$, $[x+x^*]=\lceil[x]/2\rceil$, and $\langle x+x^*\rangle=\langle x\rangle$.\\
    Using this notation we will write 
    \[C_i := \{\{a_{2i-1},b_{2i-1}\}\sqcup R_{2i-1},\{a_{2i},b_{2i}\}\sqcup R_{2i}\}\] and
    \[C_i^* := \{\{a_{2i-1},b_{2i}\}\sqcup R_{2i},\{a_{2i},b_{2i-1}\}\sqcup R_{2i-1}\}\]
    and refer to these as the \emph{i-th pair} in $x$ and $x^*$, respectively.

    We claim that
    \begin{equation}
        \vp(x,x^*) \geq \frac{1}{48} \frac{[x]-1}{n-1},
    \end{equation}
    which establishes the result by Lemma \ref{lem:distance-to-singular}.

    Since \[|p| = d(p,0) = \frac{n-\#p}{n-1}\] for any partition $p\in P_n$, and by the definition of $\vp$ (line \ref{eq:modular-pair}), the previous condition is equivalent to showing that for all $t\in P_n$, we have that
    \begin{equation}\label{eq: def-mm-eq}
        \max\{\#(x+x^*)+\#t-\#x-\#x^*,\quad \#x-\#(x+t),\quad \#x^*-\#(x^*+t)\}\geq \frac{[x]-1}{48}
    \end{equation}

    We assume that $\#x-\frac{[x]-1}{48}<\#(x+t)$ and $\#x^*-\frac{[x]-1}{48}<\#(x^*+t)$.
    Setting \[m := \left\lceil \frac{[x]-1}{48}\right\rceil\geq 1\] we have that 
    \begin{equation}\label{eq:distance-to-singular-10}
        \#x-m\leq\#(x+t), \quad \#x^*-m\leq\#(x^*+t)
    \end{equation} as all other terms are integers.

    Since $\langle p \rangle \geq \langle q\rangle$ whenever $p\leq q$, we have that
    \begin{equation*}
        \langle x\rangle + [x] - m = \#x-m \leq \#(x+t) =\langle x+t\rangle + [x+t]\leq \langle x\rangle + [x+t].
    \end{equation*}
    By similar reasoning for $x^*$, we conclude using line (\ref{eq:distance-to-singular-1}) that
    \begin{equation}\label{eq:distance-to-singular-2}
        [x]-m = [x^*]-m \leq [x+t],\  [x^*+t].
    \end{equation}

    Applying Lemma \ref{lem:singleton-difference} to the pairs $x\leq x+t$ and $x^*\leq x^*+t$ and using (\ref{eq:distance-to-singular-1}) and (\ref{eq:distance-to-singular-10}), we have that 
    \begin{equation}\label{eq:distance-to-singular-3}
        \langle x\rangle - 2m = \langle x^*\rangle - 2m \leq \langle x+t\rangle,\ \langle x^*+t\rangle.
    \end{equation}
    
    Next, note that by the inequalities (\ref{eq:distance-to-singular-3}) that there are at least $\langle x\rangle - 2m$ singleton blocks of $x$ which still appear as singleton blocks of $x+t$. It follows that these must also be singleton blocks of $t$, so 
    \begin{equation}
        \langle t\rangle \geq \langle x \rangle -2m = \langle x^*\rangle - 2m.
    \end{equation}

    Finally, we note that from line (\ref{eq:distance-to-singular-2}) that there are $[x]-m$ blocks of $x$ of size at least two which belong to pairwise distinct blocks of $x+t$. We denote an analogous set of blocks of $x$ by $D$. Similarly, we denote an analogous set of blocks of $x^*$ by $D^*$. For each index $i\in \{1,\dotsc,\lfloor [x]/2\rfloor\}$ such that both blocks of the pair $C_i$ of $x$ are elements of $D$ and both blocks of the pair $C_i^*$ are elements of $D^*$, the elements $\{a_{2i-1}, b_{2i-1}, a_{2i}, b_{2i}\}$ must belong to distinct blocks of $t$ which aren't among the $\langle x\rangle-2m$ singleton blocks chosen in the previous paragraph. We will call these pairs the ``repeated pairs''.
    
    For $x$, at most $m$ blocks of size at least two do not belong to $D$, so there are at least 
    \[\left\lfloor\frac{[x] - 2m}{2}\right\rfloor = \lfloor [x]/2\rfloor - m\]
    pairs $C_i$ among the blocks of $D$. Similarly there are at least $\lfloor [x^*]/2\rfloor - m =\lfloor [x]/2\rfloor - m$ pairs $C_i^*$ among the blocks of $D^*$. Therefore, the number of repeated pairs is at least $\lfloor [x]/2\rfloor - 2m$. Using this along with the estimate (\ref{eq:distance-to-singular-3}) we conclude that
    \begin{equation}\label{eq:distance-to-singular-4}
        \begin{split}
            \#t &\geq 4\left(\frac{[x]-1}{2} - 2m\right)+\langle x\rangle -2m\\ 
            &= 2[x] + \langle x\rangle - 10m - 2\geq 2[x] + \langle x\rangle - 12m,
        \end{split}   
    \end{equation}
    since $m\geq 1$.
    
    It follows from lines (\ref{eq:distance-to-singular-1}) and (\ref{eq:distance-to-singular-4}) that
    \begin{align*}
        &\#(x+x^*)+\#t-\#x-\#x^*\\
        &\geq\frac{[x]-1}{2}+\langle x\rangle +\#t-2([x]+\langle x\rangle)\\
        &\geq \frac{[x]-1}{2}+\langle x\rangle +\biggl[ 2[x] + \langle x\rangle - 12m\biggr] -2([x]+\langle x\rangle)\\
        &=\frac{[x]-1}{2}- 12m\\
        &\geq \frac{[x]-1}{2}-12\,\frac{[x]-1}{48} =\frac{[x]-1}{4} >\frac{[x]-1}{48},
    \end{align*}
    which suffices to show that the equation \ref{eq: def-mm-eq} is satisfied, proving the result. \qedhere
\end{proof}

\begin{cor}\label{cor:metrically-modular-equals-singular}
    If $P_n$ is a finite partition lattice, then $x\in P_n$ is metrically modular if and only if $x$ is singular.
\end{cor}

\begin{proof}
    The `` only if'' direction follows directly from Proposition \ref{prop:distance-to-singular-controls-modularity}. For the ``if'' direction, we fix a singular partition $x\in P_n$ and an arbitrary partition $y$. By Lemma \ref{lem:singular-partition-lemma} we have that $\#(x+y) + \#xy = \#x + \#y$, hence $|x+y| + |xy| = |x| + |y|$. \qedhere
   
\end{proof}

\begin{cor}\label{cor:metric-modular-ultraproduct-singular}
    Let $\cc U$ be a non-principal ultrafilter on $\bb N$ and $P = \prod_{\cc U} P_n$ be an ultraproduct of finite partition lattices. We have that $x\in P$ is metrically modular if and only if $x$ is an ultraproduct of singular partitions.
\end{cor}

\begin{proof}
    Suppose $x\in P$ is metrically modular, and let $(x_n)\in \prod_{n\in \bb N} P_n$ be a representative sequence for $x$. By \L{}os' Theorem we have that $\lim_{n\to\cc U} \sup_{y\in P_n}\vp(x_n,y)=0$; hence, by Proposition \ref{prop:distance-to-singular-controls-modularity} we have that $\lim_{n\to\cc U} d(x_n, \Sg_n)=0$. Thus, we can find $x_n'\in \Sg_n$ so that $\lim_{n\to\cc U} d(x_n,x_n')= 0$.
    
    In the other direction, by Corollary \ref{cor:metrically-modular-equals-singular} every singular element of $\Sg_n$ is metrically modular, hence so is every element of $\prod_{n\in\cc U} \Sg_n$ by Proposition \ref{prop:distance-to-singular-controls-modularity} and \L{}os' Theorem.
\end{proof}

\begin{defn}
    Let $P_n$ be the finite partition lattice of a set with $n$ elements. A sublattice $A$ is called \emph{singular} if every $x\in A$ is a singular partition.
\end{defn}
\begin{defn}
    Let $\cc L=(L,<,+,\cdot,0,1)$ be a lattice. An element $x\in L$ is called and atom if and only if $x\neq 0$ and for all $y\in L$, if $0\leq y\leq x$, then $y=0$ or $y=x$.
\end{defn}

\begin{lem}\label{lem:maximal-singular-boolean}
    Let $A$ be a maximal singular distributive sublattice of $P_n$.
    The following claims are true:
    \begin{enumerate}
        \item There must exist an element $a^*\in [[n]]$ such that $a^*$ belongs to the basic block of every non-zero partition of $A$. 
        
        \item The atoms of $A$ are the singular partitions $x_e$ for $e\in [[n]]-\{a^*\}$, where $x_e$ is the singular partition with basic block $\{a^*,e\}$. 

        \item $A\models T_{BML}$.
        
        \item $A$ is a maximal boolean sublattice of $P_n$.
        
        \item For every singular partition $x\in P_n$, we have that $d(x,A) \leq \frac{1}{n-1}$.
    \end{enumerate}
\end{lem}

\begin{proof}
    We prove each claim in turn.
    
        \begin{enumerate}
            \item Let us suppose the claim is false, then $\bigcap_{x\in A,x\neq0} B_x=\emptyset$. Since for all $x\in A$, there is an atom $x'\leq x$ such that $B_{x'}\subset B_x$ we have 
            \[
            \bigcap_{x\in A, x \text{ atom}} B_x=\bigcap_{x\in A,x\neq0} B_x=\emptyset.
            \]
            Now, given atoms $x,y\in A, x\neq y, B_x\cap B_y\neq\emptyset$, otherwise $x+y$ would not be a singular partition, and as $xy=0$, then $|B_x\cap B_y|=1$. Thus, as $\bigcap_{x\in A, x \text{ atom}} B_x=\emptyset$ there must be different atoms $x,y,z\in A$ with basic blocks $B_x$, $B_y$, and $B_z$, respectively, such that $B_x\cap B_y=\{a\}, B_x\cap B_z=\{b\}, B_y\cap B_x=\{c\}$ and $B_x\cap B_y\cap B_z=\emptyset$.  Observe that $x(y+z)\neq0$ because $\{a,b\}\subset B_x\cap(B_y\cup B_z)=B_{x(y+z)}$, but as $x,y,z$ are atoms we have that $xy+xz=0+0=0\neq x(y+z)$. Hence $A$ wouldn't be distributive.\\

            \item Let $A^*$ be the sublattice generated by all such $x_e$. It's clear that $A^*$ is Boolean as it can be seen to be isomorphic to the lattice $2^{[[n]]-\{a^*\}}$ by associating to each $X\subseteq [[n]]-\{a^*\}$ the singular partition with basic block $X\cup\{a^*\}$.  By the previous claim there is an $a^*$ which belongs to the basic block of every non-zero singular partition in $A$, hence $A\subseteq A^*$. Equality then follows by maximality of $A$.
            \\

            \item It is clear that the isomorphism of $A^*$ with the Boolean lattice $2^{[[n]]-\{a^*\}}$ given in the previous item is isometric with respect to the canonical lattice metric on $2^{[[n]]-\{a^*\}}$ induced by the counting measure.\\
            
            \item For any partition $x$ not belonging to $A$ there is a block which contains a pair of elements $\{e,e'\}\subseteq [[n]]-\{a^*\}$.
            We have that $x(x_e+x_{e'})\neq xx_e+xx_{e'}$ because the pair $\{e,e'\}$ is included in a block of $x(x_e+x_{e'})$ but $e$ and $e'$ are in different blocks of the partition $xx_e+xx_{e'}$. Therefore, $A\cup \{x\}$ cannot generate a Boolean sublattice by failure of distributivity; hence, $A$ is maximal Boolean.\\
            %Thus, we conclude that a partition with those atoms isn't just maximal among the singular boolean sublattices but among boolean sublattices too.
            \item Let $x\in P_n$ be a singular partition with basic block $B_x$. We choose $y\in A$ to be the singular partition with basic block $B_x\cup\{a^*\}$ and observe that 
            \[d(x,y)=\frac{\#x-\#y}{n-1}\leq\frac{\#x-(\#x-1)}{n-1}=\frac{1}{n-1}. \qedhere\]
        \end{enumerate}
    \end{proof}
    
Now, we can use this lemma to prove our result.

\begin{prop}\label{prop:modular-ultraproduct}
    Let $\cc U$ be a non-principal ultrafilter on $\bb N$. For each $n\in\bb N$ choose $A_n\subseteq P_n$ a maximal singular sublattice. We have that $\prod_{n\in\cc U} A_n$ is a Boolean sublattice of $\prod_{n\in\cc U} P_n$ and is equal to the set of metrically modular elements of $\prod_{n\in\cc U} P_n$.
\end{prop}

\begin{proof}
    We begin by showing that every metrically modular element of $\prod_{n\in\cc U} P_n$ belongs to $\prod_{n\in\cc U} A_n$. Indeed, if $x$ is metrically modular, by Corollary \ref{cor:metric-modular-ultraproduct-singular} there is representative sequence $(x_n)$ with $x_n\in P_n$ singular for $\cc U$-generic $n$. Hence, by Lemma \ref{lem:maximal-singular-boolean}.5 we have that $x$ has a representative sequence $(x_n')$ with $x_n'\in A_n$ for $\cc U$-generic $n$, and $x\in\prod_{n\in \cc U} A_n$. On the other hand it is clear that if $x\in\prod_{n\in\cc U}A_n$, then $x$ is singular again by Lemma \ref{cor:metric-modular-ultraproduct-singular}. Finally, $\prod_{n\in \cc U} A_n$ is a Boolean algebra because by Lemma \ref{lem:maximal-singular-boolean} we know that each $A_n\models T_{BML}$ so $\prod_{n\in\cc U}A_n\models T_{BML}$, and therefore it is a boolean sublattice of $\prod_{n\in\cc U} P_n$.
\end{proof}

\subsection{Pseudofinite partition lattices}

\begin{defn}\label{def-T_FPL}
    We write $T_{FPL}$ as the theory of the set $\{P_n :n\in\bb N\}$, that is $\cc M\models T_{FPL}$ if and only if $\sg^{\cc M} =0$ for any sentence so that $\sup_n \sg^{P_n} =0$. We then say that $\cc M$ is a \emph{pseudofinite partition lattice}. Note that any ultraproduct of finite partition lattices is a pseudofinite partition lattice.
\end{defn}
\noindent 

In this section we investigate the properties of $T_{FPL}$. Let $\cc M\models T_{ML}$, and consider the functor $\cc M\mapsto \mu(\cc M)$ which sends $\cc M$ to its collection of metrically modular elements. (Note that $\mu(\cc M)$ is always non-empty since $0$ and $1$ are modular.)  We begin with the following observation which is a direct consequence of Proposition \ref{prop:distance-to-singular-controls-modularity} and Corollary \ref{cor:metrically-modular-equals-singular}.

\begin{lem}\label{lem:mu-definable}
    We have that $\cc M\mapsto \mu(\cc M)$ is a definable functor in $T_{FPL}$.
\end{lem}

\begin{proof}
    Indeed, the quoted results show that every ``approximately modular'' element of any finite partition lattice is uniformly close, independent of the size of the partition lattice, to a singular partition and that singular partitions are modular. 
\end{proof}

\begin{remark}
    By Theorem \ref{thm:definability}, it follows that we may quantify over modular elements in $T_{FPL}$. 
\end{remark}

\begin{prop}\label{prop:boolean-core}
    If $\cc M\models T_{FPL}$ and $\cc M$ is infinite, then $\mu(\cc M)$ is a Boolean sublattice. 
\end{prop}

\begin{proof}
    It is well known (see, for instance, \cite[Lemma 16.2.4]{goldbring}) that if $\cc M\models T_{FPL}$, then it admits an elementary embedding into an ultraproduct $\prod_{i\in U} P_{n_i}$. As $\cc M$ has density character at least $\aleph_0$, we must have that $\lim_{i\in U} n_i =\infty$. Since modularity and being a Boolean sublattice are definable properties by Lemma \ref{lem:mu-definable} and Proposition \ref{prop:boolean-elementary}, respectively, the result follows by Proposition \ref{prop:modular-ultraproduct}.
\end{proof}

\begin{question}
    Is $\mu(\cc M)$ definable for an arbitrary model of $T_{ML}$?
\end{question}

    Regarding this question, note that the formula $\vp$ (equation (\ref{eq: def-mm-eq})) is not enough to define or control the meet in $T_{ML}$, as it was showed in Remark \ref{rmk:meet-ctrl-fails}, but it is enough to show the definability of $\mu(\cc M)$ in models of $T_{FPL}$.

\subsection{Sets of selectors}\label{sec: set-sel}
\begin{defn}
    Let $\cc L$ be a metric lattice, and let us fix $x\in L$. We say that an element $y\in\mu(\cc L)$ is a \emph{selector} for $x$ if $y$ complements $x$. 
\end{defn}

As we will see, the set of selectors play an interesting role in the context of $T_{FPL}$. In the following results, we show that the set of selectors of a partition is a non-empty definable set and that such sets allows us to recover the partitions by comparing the usual distance $d(x,y)$ with the Hausdorff distance between sets of selectors. 

The following is a straightforward consequence of Lemma \ref{lem:metric-complement}.

\begin{lem}\label{lem:selector}
    We have that $y$ is a selector for $x$ if and only if $x+y=1$ and $|x| + |y| = 1$.
\end{lem}

\begin{lem}\label{lem:partitions-have-selectors}
    Let $P_n$ be a finite partition lattice. We have that every element $x\in P_n$ admits a selector. The set of selectors is the set of all singular partitions whose main block intersects each block of $x$ at exactly one element.
\end{lem}

\begin{proof}
    Given a partition $x$ and a singular partition $y$ of $P_n$ whose main block $B_y$ intersects each block of $x$ at exactly one element we observe that $B_y$ contains all singletons of $x$ and since it intersects every other block of $x$, it is clear that $x+y=\{[[n]]\}=1$. On the other hand, we observe that $|B_y|=\langle x\rangle+[x]=\#x$; therefore, $\#y= n-\#x+1$. Consequently,
    \[|x|+|y|=\frac{n-\#x}{n-1}+\frac{n-(n-\#x+1)}{n-1}=\frac{n-\#x+\#x-1}{n-1}=\frac{n-1}{n-1}=1,\]
    so this direction follows by Lemma \ref{lem:selector}.

    Now suppose $y$ is a selector of $x$. As, by Corollary \ref{cor:metrically-modular-equals-singular}, $y$ is singular, and $xy=0, x+y=1$, we know that $$1=\#(x+y)=\#x-i(x,B_y)+1$$ 
    and 
    $$n=\#xy=\#y+i(x,B_y)-1=n-|B_y|+1+i(x, B_y)-1$$

    From the first equality it follows that $\#x=i(x,B_y)$, that is the main block intersects each block of $x$, and from the second that $\#x=i(x, B_y)=|B_y|$, so the main block of $y$ has exactly $\#x$ elements. This concludes the proof.
\end{proof}

The main result of this section is that the same is true for any pseudofinite partition lattice.

\begin{prop}\label{prop:tfpl-selector}
    If $\cc M\models T_{FPL}$, then every element $x\in M$ admits a selector. 
\end{prop}

In order to prove this proposition, we will first require a few intermediate results.

\begin{lem}\label{lem:almost-selector-to-selector}
    Let $\cc M\models T_{FPL}$, and let $x\in\cc M$ and $y\in\mu(\cc M)$. We have that
     \[\inf_{z\in\mu(\cc M)} \max\left\{d(z,y)\dminus d(x+y,1), d(x+z,1)\right\}=0.\]
\end{lem}

\begin{proof}
    We claim that for any finite partition lattice $P$ and any $x\in P$ and $y\in \mu(P)$, there is $z\in \mu(P)$ so that $d(y,z)\leq d(x+y,1)$ and $d(x+z,1)=0$. From this, the result easily follows.

    Let $P_{n+1}$ be the finite partition lattice on $n+1$ elements, and let us fix a partition $x$ and a singular partition $y$. By Corollary \ref{cor:metrically-modular-equals-singular} we know that every element of $\mu(P_{n+1})$ is a singular partition. Now, using Lemma \ref{lem:singular-partition-lemma} we can compute
     \[d(x+y,1)=1-|x+y|=1-\frac{n+1-\#(x+y)}{n}=\frac{\#x-i(x, B_y)}{n}.\] 
    Defining $k := \#x- i(x,B_y)$, we have that $d(x+y,1)=\frac{k}{n}$. Observe that from Lemma \ref{lem:singular-partition-lemma} it also follows that $x+y=1$ if and only if $i(x,B_y)=\#x$. By the definition of $i(x,B_y)$ there are $k$ blocks of $x$ that do not intersect $B_y$; denote them by $C_1,..., C_k$. Taking $c_i\in C_i$ for $1\leq i\leq k$ we define a new singular partition $z$ where $B_z:= B_y\cup\{c_i\}_{i=1}^k$. We see that $i(x,B_z)=\#x, y\leq z$, and $\#z=\#y-k$, so we conclude that $x+z=1$ and 
     \[d(z,y)=|z|-|y|=\frac{n+1-\#z}{n}-\frac{n+1-\#y}{n}=\frac{\#y-(\#y-k)}{n}=\frac{k}{n}\leq d(x+y,1),\]
     finishing the proof.
\end{proof}

\begin{cor}\label{cor:z-covers-y}
    Let $\cc M\models T_{FPL}$. For any $x\in M$, $y\in\mu(\cc M)$ and $\epsilon>0$, there is $z\in \mu(M)$ so that $x+z=1$ and $d(y,z)\leq (1+\epsilon)\,d(x+y,1)$.
\end{cor}

\begin{proof}
    Given such $x,y\in M$, if $x+y=1$, the claim is immediately true, so let's suppose $x+y\neq1$, i.e., $d(x+y,1)>0$, and pick $\e>0$ by Lemma \ref{lem:almost-selector-to-selector} we may inductively construct a sequence $(z_n)$ in $\mu(\cc M)$ in the following way.
    \begin{itemize}
        \item Let $z_0 = y$.
        \item Given $z_n$, take $z_{n+1}$ such that $d(z_{n+1},z_n)\leq d(x+z_n,1)+2^{-(n+1)}\epsilon d(x+y,1)$
    and $d(x+z_{n+1},1)\leq 2^{-(n+2)}\epsilon d(x+y,1)$.
    \end{itemize}
    By construction we have that $d(z_{n+1},z_n)\leq2 \cdot2^{-(n+1)}\epsilon d(x+y,1)=2^{-n}\e d(x+y,1)$ for $n>1$. Thus, by the triangle inequality, $(z_n)$ is Cauchy and converges to $z\in\mu(\cc M)$ with $x+z=1$ and $d(z,y)\leq d(x+y,1)+\epsilon d(x+y,1)$, proving the result.
\end{proof}

We define the following $\ff L$-formula for models of $T_{FPL}$:
    \begin{equation*}
        \begin{split}
             \chi_z(y):=\inf_{w\in\mu(\cc M)} \max\{d(z,w)\dminus &|1-|y|-|z||,\:  d(z+w,z),\\
                &d(w+y,1),\:|1-|w|-|y||\}.
        \end{split}
    \end{equation*}

\begin{lem}\label{lem:axiom-behavior-almost-selectors}
    The following formula holds in $T_{FPL}$:
     \[\sup_y\sup_{z\in\mu(\cc M)}\chi_z(y)\dminus 2d(y+z,1)=0.\]
\end{lem}

\begin{proof}
    We claim that for any finite partition lattice $P = P_{n+1}$ and any $y\in P$ and $z\in \mu(P)$, there is $w\in \mu(P)$ so that
    \begin{equation*}
        \begin{split}
            d(w,z) &\leq 2\, d(y+z,1)+|1-|y|-|z||,\\
            d(z+w,z) &= d(y+z,1),\ \textup{and}\\
            d(w+y,1)&=|1-|w|-|y||=0.
        \end{split}
    \end{equation*}
    From this the result easily follows.

     We fix $y$ and $z$ as above, and define the numbers 
     \[k:=\#y-i(y,B_z), \quad l:=|B_z|-i(y,B_z).\]
     Note that \[d(y+z,1)=\frac{\#(y+z)-1}{n}=\frac{\#y-i(y, B_z)}{n}=\frac{k}{n}\]
     and 
     \[|1-|y|-|z||=\left|\frac{n-(n+1-\#y)-(n+1-\#z)}{n}\right|=\left|\frac{\#y-1-(|B_z|-1)}{n}\right|=\frac{|k-l|}{n}.\] 
     By the definition of $i(y,B_z)$ we get that there are $k$ blocks, $C_1,\dotsc,C_k$ of $y$ which do not intersect $B_z$ and $l$ additional elements, $x_1,\dotsc,x_l$ of $B_z$ which belong to blocks whose intersection with $B_z$ contains at least one other element. Taking an element $c_i\in C_i$ from each of the $k$ blocks, as we did in the proof of Lemma \ref{lem:almost-selector-to-selector}, we let $w$ be the singular partition for which $B_w = (B_z\setminus\{x_1,\dotsc,x_l\})\cup\{c_1,\dotsc,c_k\}$. By construction we have that $|B_w|=\#y=|B_z|+k-l$, $y+w=1$, $1=|y|+|w|$, and $z+w$ is a singular partition. Finally, we note that 
     \begin{equation*}
         \begin{split}
             d(z,w)=|z+w|-|zw| &=\frac{|B_{z+w}-B_{zw}|}{n}\\
             &=\frac{|B_z|+k-(|B_z|-l)}{n}\\
             &=\frac{2k+l-k}{n}\leq 2\,d(y+z,1)+|1-|y|-|z||,
         \end{split}
     \end{equation*}
     and
     \[d(z+w,z)=|z+w|-|z|=\frac{|B_{z+w}|-|B_z|}{n}=\frac{|B_z|+k-|B_z|}{n}=\frac{k}{n}= d(y+z,1),\] as desired. \qedhere

\end{proof}

\begin{lem}\label{lem:almost-decreasing-is-cauchy}
    Let $\cc L$ be a metric lattice and $\sum_{n=0}^\infty a_n$ a convergent series of non-negative real numbers. If $(x_n)$ is a sequence in $\cc L$ such that \[d(x_{n+1}+x_n, x_n)\leq a_n\] for all $n\in\bb N$, then $(x_n)$ is Cauchy.
\end{lem}

\begin{proof}
    Let $x_{(n,k)} := x_n + x_{n+1} + \dotsb + x_{n+k}$. We have by the triangle inequality and the property that $d(x+z,y+z)\leq d(x,y)$ that 
     \begin{equation*}
        \begin{split}
            &d(x_n, x_n + x_{n+k})\\
            &\leq \sum_{i=0}^{k-1} d(x_{(n,i)}, x_{(n,i+1)}) + \sum_{j=1}^{k-1} d(x_{(n,k-j)}+x_{n+k}, x_{(n,k-j-1)} + x_{n+k})\\
            &\leq \sum_{i=0}^{k-1} d(x_{n+i}, x_{n+i} + x_{n+i+1}) + \sum_{j=1}^{k-1} d(x_{n+k-j} + x_{n+k-j-1}, x_{n+k-j-1})\\
            &\leq \sum_{i=0}^{k-1} a_{n+i}+\sum_{j=1}^{k-1} a_{n+k-j-1} \leq 2\sum_{i=n}^\infty a_n.
        \end{split}
    \end{equation*}
    We conclude that \[0\leq |x_{n+k} + x_n| - |x_n|\leq d(x_n, x_n + x_{n+k}) \leq 2\sum_{i=n}^\infty a_n,\] hence 
    \begin{equation}\label{eq:sandwich}
        |x_{n+k}| \leq |x_{n+k} + x_n|\leq |x_n| + 2\sum_{i=n}^\infty a_n
    \end{equation}
     for all $n,k\in\bb N$. Setting $\al = \liminf_n |x_n|$, it follows that $\lim_n |x_n| = \al$; thus, $(|x_n|)$ is a Cauchy sequence. It now follows from (\ref{eq:sandwich}) that
    \[\lim_n d'(x_{n+k}, x_n) = \lim_n\left(2|x_{n+k} + x_n| - |x_{n+k}| - |x_n|\right)=0\]
    for all $k\in\bb N$. Thus, $(x_n)$ is Cauchy in $d'$, hence for $d$ as well. \qedhere
\end{proof}

\begin{notation}
    For $\cc M\models T_{FPL}$ and $x\in M$, we write $\G(x)$ for the set of all selectors of $x$ in $M$.
\end{notation}

The next proposition is a quantitative version of Proposition \ref{prop:tfpl-selector}, so easily implies that result. 

\begin{prop}\label{prop:distance-from-almost-selector-to-selector}
    Let $\cc M\models T_{FPL}$ and $\e>0$. If $y\in \cc M, z\in \mu(\cc M)$ and $y+z=1$, then there exists $w\in\mu(\cc M)$ such that $w\in\Gamma(y)$ and $d(z,w)\leq (1+\e)|1-|y|-|z||$.  
\end{prop}

\begin{proof}
    Since $y+z=1$, by Lemma \ref{lem:selector} we have that $z\in \G(y)$ if and only if $|1 - |y| - |z|| = 0$. Thus, without loss of generality we may assume that $|1 - |y| - |z||\not= 0$.  It also follows from Lemma \ref{lem:axiom-behavior-almost-selectors} that $\chi_z(y)=0$; therefore, 
    for every $\epsilon>0$ we can find $w\in\mu(\cc M)$ so that 
    \begin{equation*}
        \begin{split}
            d(z,w) &\leq \left(1+\epsilon\right)|1-|z|-|y||,\\
            d(z+w,z) &\leq \epsilon|1-|z|-|y||\leq \epsilon,\\
            d(w+y,1) &\leq \epsilon|1-|z|-|y||\leq \epsilon,\ \textup{and}\\
            |1-|w|-|y|| &\leq \epsilon|1-|z|-|y||\leq\epsilon.   
        \end{split}  
    \end{equation*}
    By Corollary \ref{cor:z-covers-y} we can find a modular element $w'$ so that $w'+y=1$ and
    \[d(w,w')\leq (1+\epsilon)\,d(w+y,1)\leq (1+\epsilon)\epsilon\left|1-|z|-|y|\right|\leq \epsilon(1+\epsilon).\] 
    Thus, for any $z\in\mu(\cc M)$ with $y+z=1$ and all $\epsilon>0$ we can find $w'\in\mu(\cc M)$ so that $y+w'=1$, 
    \[d(z,w')\leq (1+\epsilon)^2|1-|y|-|z||,\]
    \[\ d(w'+z,z)\leq \epsilon(2+\epsilon),\ \textup{and}\]
    \[|1-|w'|-|y||\leq\epsilon(2+\epsilon).\]

    Starting with $w_0=z$ and some $\e>0$ , we may thus inductively construct a sequence $(w_n)$ in $\mu(\cc M)$ so that $y+w_{n}=1$, $d(w_0,w_1)\leq (1+\frac{\e}{2})(|1-|y|-|w_0||)$ and for all $n>1$
    \begin{equation}
        \begin{split}
            d(w_{n}+w_{n-1},w_{n-1}) &\leq \frac{1}{2^n},\\
            |1 - |y| - |w_n|| &\leq \frac{1}{2^n},\ \textup{and} \\
            d(w_n,w_{n-1}) &\leq \frac{\e}{2^n}||1-|y|-|w_0||.
        \end{split}
    \end{equation}
    By the last inequality, the sequence $(w_n)$ is Cauchy; thus, by completeness $(w_n)$ converges to a selector $w$ of $y$ so that 
    \[d(z,w)=d(w_0,w)\leq\sum_{n=0}^\infty d(w_n,w_{n+1})\leq\left(1+\e\sum_{n=1}^\infty 2^{-n}\right)|1-|y|-|w_0||=(1+\e)|1-|y|-|z||.\] 
    In this way the result obtains.\qedhere
    
\end{proof}

\begin{cor}
    Let $\cc M\models T_{FPL}$, consider $x,y\in M$ and $\e>0$. If $z\in \G(x)$, then there is $w\in\G(y)$ so that $d(z,w)\leq(3+\e)\, d(x,y)$.
\end{cor}

\begin{proof}
    Using that $z\in\Gamma(x)$ and Corollary \ref{cor:z-covers-y}, we know there exist a $z'$ such that $y+z'=1$ and $d(z,z')\leq (1+\e)\,d(z+y,1)\leq (1+\e)(d(z+y,z+x)+d(z+x,1))\leq (1+\e)\,d(x,y)$. Now, using Proposition \ref{prop:distance-from-almost-selector-to-selector} we know there is a $w\in\Gamma(y)$ such that
    \begin{equation*}
        \begin{split}
            d(z',w)&\leq (1+\e)|1-|y|-|z'||\\
            &= (1+\e)|1-|x|-|z|+|x|-|y|+|z|-|z'||\\
            &\leq (1+\e)(|1-|x|-|z||+||x|-|y||+||z|-|z'||)\\
            &\leq (1+\e)(0 +d(x,y)+d(z,z'))\leq (1+\e)(2+\e)\,d(x,y)=(2+3\e+\e^2)d(x,y),
        \end{split}
    \end{equation*}
    so we can conclude that 
    \[d(z,w)\leq d(z,z')+d(z',w)\leq (1+\e)\,d(x,y)+(2+3\e+\e^2)\,d(x,y)=(3+4\e+\e^2)\,d(x,y).\]

    As $\e>0$ is arbitrary, this proves the statement.
\end{proof}

The next result proves that the set of selectors of a given element is definable.

\begin{prop}
    Given $\cc M\models T_{FPL}$ and $x\in M$. For every $y\in M$ we have that 
    \[d(y,\Gamma(x))\leq |1-|x|-|y||+2\,d(y+x,1)+192\sup_w\vp(y,w),\]
    where $\vp$ is defined as in line (\ref{eq:modular-pair}).
\end{prop}
\begin{proof}
    Applying the results from Proposition \ref{prop:distance-to-singular-controls-modularity}, Corollary \ref{cor:z-covers-y} and Proposition \ref{prop:distance-from-almost-selector-to-selector} we have that:
    \begin{itemize}
        \item there is a $y'\in\mu(\cc M)$ such that $d(y,y')\leq 48\sup_w\vp(y,w)$;
        \item there is a $z\in\mu(\cc M)$ such that $x+z=1$ and $d(y',z)\leq (1+\e)\,d(y'+x,1)$;
        \item there is a $z'\in\Gamma(x)$ such that $d(z,z')\leq (1+\e)|1-|x|-|z||$.
    \end{itemize}
    We observe that $d(y,\Gamma(x))\leq d(y,z')\leq d(y,y')+d(y',z)+d(z,z')$. Developing this inequality using the mentioned results we have that for every $\e>0$:
\begin{align*}
    d(y,\Gamma(x))&\leq d(z,z')+d(y',z)+d(y,y')\\
    &\leq (1+\e)|1-|x|-|z||+(1+\e)d(y'+x,1)+d(y,y')\\
    &\leq (1+\e)(|1-|x|-|y||+||y|-|y'||+||y'|-|z||+d(y'+x,1))+d(y,y')\\
    &\leq (1+\e)(|1-|x|-|y||+d(y,y')+d(y',z)+d(y'+x,1))+d(y,y')\\
    &\leq (1+\e)(|1-|x|-|y||+(1+\e)d(y'+x,1)+d(y'+x,1))+(2+\e)d(y,y')\\
    &\leq (1+\e)(|1-|x|-|y||+(2+\e)d(y'+x,1))+(2+\e)d(y,y')\\
    &\leq (1+\e)(|1-|x|-|y||+(2+\e)(d(y+x,1)+d(y+x,y'+x)))+(2+\e)d(y,y')\\
    &\leq (1+\e)(|1-|x|-|y||+(2+\e)(d(y+x,1)+d(y,y')))+(2+\e)d(y,y')\\
    &\leq (1+\e)|1-|x|-|y||+(2+3\e+\e^2)d(y+x,1)+(4+4\e+\e^2)d(y,y')\\
    &\leq (1+\e)|1-|x|-|y||+(2+3\e+\e^2)d(y+x,1)+(4+4\e+\e^2)(48\sup_w\vp(y,w))\\
    &\leq (1+\e)|1-|x|-|y||+(2+3\e+\e^2)d(y+x,1)+(192+192\e+\e^2)\sup_w\vp(y,w),
\end{align*}
Finally, as by letting $\e$ tend to zero we conclude that
\[d(y,\G(x))\leq|1-|x|-|y||+2d(y+x,1)+192\sup_w\vp(y,w),\]
as was desired. \qedhere
\end{proof}
 
\begin{cor}\label{cor:def-of-selectors}
For all $\cc M\models T_{FPL}$ and $x\in M$, $\Gamma(x)$ is definable.
\end{cor}
\begin{proof}
    This follows from the previous result by Theorem \ref{thm:definability}.3.
\end{proof}

We recall that the Hausdorff distance between two subsets $A, B$ of a metric space $(X,d)$ is defined as 
\[d_{Haus}(A,B) := \max\left\{\sup_{x\in A}d(x,B),\ \sup_{y\in B}d(y,A)\right\}.\]

\begin{prop}\label{prop:formula-selector-set-distance}
    Given a finite partition lattice $P_n$ and $x,y\in P_n$, we have that
    \[d_{Haus}(\G(x),\G(y))\leq d(x,y).\]

Furthermore, given $z\in\G(x),w\in\G(y)$, we have that
    \begin{equation}\label{eq:6-30-1}
        d(z,w)= \frac{\#x+\#y-2\,\max\{1,|B_z\cap B_w|\}}{n-1}
    \end{equation}
    and
    \begin{equation}\label{eq:6-30-2}
        d_{Haus}(\G(x),\G(y)))=\frac{\#x+\#y-2\min\{\g(x,y), \g(y,x)\}}{n-1},
    \end{equation}
    where $\g(x,y):=\displaystyle\min_{z\in\G(x)}\max_{w\in\G(y)}|B_z\cap B_w| $.

\end{prop}

\begin{proof}
    Let's start by proving the formula (\ref{eq:6-30-1}). By Corollary \ref{cor:metrically-modular-equals-singular}, we can claim that for every pair of singular partitions $z,w\in P_n$ that
    \[d(z,w)=|z+w|-|zw|=\frac{\# zw-\#(z+w)}{n-1}=\frac{|B_z|+|B_w|-2\,\max\{1,|B_z\cap B_w|\}}{n-1}.\]
    Indeed, if $B_z\cap B_w\neq\emptyset$, then 
    \[\#zw=n-|B_z\cap B_w|+1,\quad \#(z+w)=n-|B_z\cup B_w|+1,\]
    %=n-|B_z|-|B_w|+|B_z\cap B_w|+1$\]
    and if
    %$d(z,w)=(|B_z|+|B_W|-2|B_z\cap B_w|)/n$, 
    and if $B_z\cap B_w=\emptyset$, then 
    \[\#zw=n,\quad \#(z+w)=n-|B_z|-|B_w|+2.\]
    Altogether, this establishes formula (\ref{eq:6-30-1}).
    
    We now proceed to formula (\ref{eq:6-30-2}). Fixing $z\in \Gamma(x)$ we see that
    \begin{equation*}
        \begin{split}
            d(z,\Gamma(y)) &=\min_{w\in\Gamma(y)}\frac{|B_z|+|B_w|-2\,\max\{1,|B_z\cap B_w|\}}{n-1}\\
            &=\min_{w\in\Gamma(y)}\frac{\#x+\#y-2\,\max\{1,|B_z\cap B_w|\}}{n-1}\\
            &=\frac{\#x+\#y-2\max\{1,\max_{w\in\Gamma(y)}|B_z\cap B_w|\}}{n}.
        \end{split}
    \end{equation*}
    Observe that we can always find a $w\in\Gamma(y)$ such that $|B_z\cap B_w|\geq \#(x+y)\geq 1$ because we can eliminate elements of $B_z$ until we get a selector of $x+y$ which will have a main block of cardinality $\#(x+y)$; we can then add elements to get a selector $w$ of $y$. Therefore, we can conclude that, for every $z\in\Gamma(x),$
    \[d(z,\Gamma(y)) = \frac{\#x+\#y-2\,\max_{w\in\Gamma(y)}|B_z\cap B_w|}{n}.\]
    Formula (\ref{eq:6-30-2}) now follows easily.
    Finally, note that $\g(x,y), \g(y,x)\geq \#(x+y)$ by the observations made in the previous paragraph; thus, we can conclude that $d_{Haus}(\G(x),\G(y))\leq d(x,y)$ as desired. \qedhere
    
\end{proof}

\begin{cor}\label{cor:bound-dH-by-d}
    Given $\cc M\models T_{FPL}$ and $x,y\in M, d_{Haus}(\G(x),\G(y))$ is definable and 
    \begin{equation}\label{eq:prop-6-30}
        \cc M\models \sup_{x,y}\bigl(d_{Haus}(\G(x),\G(y))\dminus d(x,y)\bigr)=0.
    \end{equation}
\end{cor}

\begin{proof}
    By Corollary \ref{cor:def-of-selectors}, we know that for any $x,y\in M$, $\G(x)$ and $\G(y)$ are definable; thus, the expression
        \begin{equation*}
        \begin{split}
            d_{Haus}(\G(x),\Gamma(y)) &=\max\left\{\sup_{w\in \G(x)}d(w,\G(y)),\ \sup_{z\in \G(y)}d(z,\G(x))\right\}\\
            &=\max\left\{\sup_{w\in \G(x)}\inf_{z\in \G(y)}d(w,z),\ \sup_{z\in \G(y)}\inf_{w\in\G(x)}d(z,w)\right\}
        \end{split}
    \end{equation*}
    is a definable predicate in $(x,y)$. Finally, by Proposition \ref{prop:formula-selector-set-distance}, we know that (\ref{eq:prop-6-30}) holds for all finite partition lattices, proving the statement. \qedhere
\end{proof}

\begin{remark}\label{rmk:strc-inq-Haus-dist}
    Equality generally fails to hold. Consider partitions $x,y$ of $P_6$ defined as follows: $x:=\{\{1,4\},\{2,3\},\{5,6\}\}$ and $y:=\{\{1,2\},\{4,5\},\{3\},\{6\}\}$. Note that \[d(x,y)=\frac{\#x+\#y-2\#(x+y)}{5}=\frac{3+4-2}{5}=1.\] However, we can check that for any $z\in\G(x)$ there exists a $w\in\Gamma(y)$ such that $|B_z\cap B_w|=2$, similarly for any $w\in\G(y)$. Therefore, by Proposition \ref{prop:formula-selector-set-distance} we conclude that $d_{Haus}(\G(x),\G(y))=\frac{4+3-2\cdot 2}{5}=\frac{3}{5}<1=d(x,y)$.
\end{remark}

However, we can, in fact, bound $d(x,y)$ using a multiple of $d_{Haus}(\G(x),\G(y))$. Before presenting the result, we introduce the following notation and remarks.
\begin{defn}
    In $P_n$, given $x\in P_n$ and $S\subset [[n]]$ we define the restriction of $x$ to $S$ as $x_S:=\{B\cap S: B\in x\}$.
\end{defn}
If we impose some conditions on $S$, the restrictions will have good behavior:
\begin{remark}
    Consider $x,y,z\in P_n$ with $z:=\{B_1,..., B_m\}\geq x,y$. Then the following claims are true:
    \begin{itemize}
        \item $\#x=\sum_{i=1}^m \#x_{B_i}$,
        \item $t\in\G(x)$ if and only if there exists $t_i\in \Gamma(x_{B_i})$ for $i=1,...,m$ such that $B_t= \sqcup_{i=1}^m B_{t_i}$,
        \item $\g(x,y)=\sum_{i=1}^m \g(x_{B_i},y_{B_i})$.
    \end{itemize}
\end{remark}

\begin{prop}\label{prop:selectors-hausdorff}
    Given $\mathcal{M}\models T_{FPL}$ and $x,y\in M$, we have that 
    \[d(x,x+y)\leq 2\, d_{Haus}(\Gamma(x),\Gamma(y)).\]
    Consequently, 
    \[d(x,y)\leq 4\,d_{Haus}(\G(x),\G(y))\leq 4\,d(x,y)\]
    for all $x,y\in M$.
\end{prop}

\begin{proof}
    We will show that the inequality holds for all finite partition lattices, from there it follows that $T_{FPL}$ implies the statement. Let's consider the finite partition lattice $P_n$ and $x,y\in P_n$.
    
    Firstly, observe that for each $B\in x+y$ if $x_B\leq y_B$, then $\#y_B=1$, and thus 
    \[\g(x_B,y_B) = \g(y_B,x_B) =1.\]
    Considering the partition $\{H,H^c\}\geq x+y$, where 
    \[H=\bigcup\{B\in x+y: x_B\leq y_B \vee y_B\leq x_B\},\] and using Proposition \ref{prop:formula-selector-set-distance} we note that
    \[\g(x_H,y_H) = \g(y_H,x_H) = \#(x+y)_H,\]
    \begin{equation*}
        \begin{split}
            d(x,x+y) &=\frac{\#x-\#(x+y)}{n-1}\\
            &=\frac{\#x_H+\#x_{H^c}-\#(x+y)_H-\#(x+y)_{H^c}}{n-1},
        \end{split}
    \end{equation*}
    and 
    \begin{equation*}
        \begin{split}
            d_{Haus}(\G(x),\G(y)) &=\frac{\#x+\#y-2\min\{\g(x,y),\g(y,x)\}}{n-1}\\
            &=\frac{\#x_H +\#y_H -2\#(x+y)_H}{n-1}\\
            &\quad + \frac{\#x_{H^c} +\#y_{H^c}- 2\min\{\g(x_{H^c},y_{H^c}), \g(y_{H^c},x_{H^c})\}}{n-1}
        \end{split}
    \end{equation*}
    Since it is clear that $\#x_H-\#(x+y)_H\leq 2(\#x_H+\#y_H-2\#(x+y)_H)$, we need to focus on the $H^c$ part. Therefore, by replacing $x$ and $y$ by $x_{H^c}$ and $y_{H^c}$, we can assume
    without loss of generality that for all $B\in x+y, x_B\nleq y_B$ and $y_B\nleq x_B$. 
    We need to show 
    \begin{equation*}
        d(x,x+y) \leq \frac{2\#x+2\#y-4\min\{\g(x,y),\g(y,x)\}}{n-1}=2\,d_{Haus}(\G(x),\G(y)).
    \end{equation*}
    Recalling that $d(x,x+y) = (\#x - \#(x+y))/(n-1)$, this is equivalent to showing
    \begin{equation}
        4\min\{\g(x,y),\g(y,x)\}\leq \#x+2\#y+\#(x+y).
    \end{equation}
    We will establish this through a series of claims. Before proceding further, let us introduce some definitions. 

    Under the assumption that $x_B\not\leq y_B$ and $y_B\not\leq x_B$ for all $B\in x+y$, we define a graph $\cc G_x = (\cc V_x, \cc E_x)$ as follows. The vertices of $\cc G_x$ are the blocks of $x$. Two blocks $A,B\in x$ are connected by an edge if there is some block $C\in y$ so that $A\cap C$ and $B\cap C$ are both nonempty, and $A\in x$ has a self loop if there is $Y\in y$ so that $Y\subseteq A$. By our assumptions, every vertex is incident to at least one edge that is not a self loop. We note that this graph is not simple in the sense that we allow two blocks to be joined by multiple edges, one for each block of $y$ which intersects each of the blocks non-trivially. We thus have an edge labeling function $\ell_x: \cc E_x\to y$ which sends each edge to the block of $y$ that induces it. For an edge $e\in \cc E_x$, we refer to $\ell_x(e)$ as the \emph{label} of $e$.

    We fix a maximum matching, i.e., a matching of maximal cardinality,  $P:=\{\{A_i, B_i\}\}_{i=1}^{N}$ of the vertices of $\cc G_x$ and let $\{\{C_j\}\}_{j=1}^{S}$ enumerate the unmatched vertices. We consider the quotient graph $\widetilde {\cc G}_x = (\widetilde {\cc V}_x, \widetilde {\cc E}_x)$ of $\cc G_x$ obtained by identifying the vertices corresponding to $A_i$ and $B_i$ for each $i=1,\dotsc, N$ and leaving the $C_j$ for. $j=1,...,S$. We partition $\widetilde {\cc V}_x$ into two sets $U_x = \{u_1,\dotsc, u_N\}$ and $V_x = \{v_1,\dotsc,v_S\}$ corresponding, respectively, to the sets of matched vertex pairs and unmatched vertices.
    
    Before continuing, we bring attention to the following delicate distinction: an edge $e\in\cc E_x$ is incident to $\{A,B\}\subset \cc G_x$ if and only if $e$ is incident to both $A$ and $B$, while an edge $e'\in\widetilde{\cc E}_x$ is incident to $\{A,B\}$ if and only if $e'$ is incident to $A$ or $B$. If this follows from the context in which set of edges we are working with, we will not mention it. 
    
    We define the function $L_x: \widetilde{\cc V}_x\to 2^y$ by sending each vertex $v\in V_x$ to the set of all labels of edges in $\cc E_x$ incident to it, and each vertex $u=\{A,B\}\in U_x$ to the set of all labels of edges in $\cc E_x$ incident to $\{A,B\}$, that is incident to both $A$ and $B$ . For each subset $S\subseteq \widetilde{\cc V}_x$, we define its support $\supp(S)\subseteq 2^x$ to be the collection of the representative blocks of the elements of $S$. For example, if $u = \{A,B\}\in U_x$, $\supp(\{u\}) = \{A,B\}$.
    Note that there can be multiple edges with the same label incident to a vertex $v$.

    Interchanging the roles of $x$ and $y$, we can define the graph $\widetilde{\cc G}_y$ similarly, with $U_y, V_y, L_y$, etc., having the same interpretations.

    \begin{claim}\label{lem:graph-behavior}
        With the same notation as above, the following statements are true.
        \begin{enumerate}
            \item For $\{A, B\}:=u\in U_x$ and $v:=\{C\},v':=\{C'\}\in V_x$, distinct, if there is an edge $e\in \cc E_x$ incident to $\{A, C\}$, then there is no edge $e'\in\cc E_x$ incident to $\{B, C'\}$.
            
            \item For $v,v'\in V_x$, distinct, we have that $L_x(v)\cap L_x(v') = \emptyset$.
            
            \item For $u,u'\in U_x$ and $v,v'\in V_x$, all distinct, if there are edges $e, e'$ so that $e$ is incident to $\{u,v\}$ and $e'$ is incident to $\{u',v'\}$, then $L_x(u)\cap L_x(u') = \emptyset.$
            
        \end{enumerate}
    \end{claim}

    \begin{proof} We prove each assertion in turn.
    
        \begin{enumerate}
             \item If there was an edge $e'\in\cc E_x$ incident to $\{B, C'\}$, then the set 
             \[(P\setminus\{\{A,B\}\})\cup\{\{A,C\}, \{B, C'\}\}\]
             would be a matching of $\cc G_x$ of greater cardinality, contradicting that $P$ is a maximum matching.
             
            \item Let us denote by $C$ and $C'$ the blocks of $x$ corresponding to $v, v'\in V_x$. If $L_x(v)\cap L_x(v')\neq\emptyset$, then there exist edges $e, e'$ incident to $v$ and $v'$ respectively, such that $\ell_x(e)=\ell_x(e')=Y\in y$. By  definition of $\cc G_x$ this means that $Y\cap C\neq \emptyset$ and $Y\cap C'\neq\emptyset$, so there is an edge $e^*$ incident to $\{v,v'\}$, and $P\cup \{v, v'\}$ would be a matching of $\cc G_x$. This is a contradiction, since $P$ is a maximum matching.
            
            \item Let us denote by $\{A, B\}, \{A', B'\}$ the sets of blocks of $x$ corresponding to $u, u'$, respectively, and by $C, C'$ the blocks of $x$ corresponding to $v, v'$ respectively. If $L_x(u)\cap L_x(u')\neq\emptyset$, then there exist edges $f, f'$ in $\cc G_x$ incident to $\{A, B\}$ and $\{A', B'\}$, respectively, such that $\ell_x(f)=\ell_x(f')=Y$. Therefore, $A,A',B,B'$ all have non-empty intersection with $Y$. As $e$ is incident to $\{u,v\}$ and $e'$ is incident to $\{u',v'\}$, we can assume without loss of generality that $\{A,C\}$ and $\{A',C'\}$ are incident in $\cc G_x$. Therefore, the set 
            \[(P \setminus\{\{A, B\}, \{A', B'\}\})\cup\{\{C, A\},\{B, B'\}, \{C, A'\}\}\]
            is a matching of $\cc G_x$, again contradicting maximality. \qedhere
           
        \end{enumerate}
    \end{proof}

    \begin{claim}\label{cor:match-beh}
        Keeping the same notation, the following statements are true.
        \begin{enumerate}
            \item For $u\in U_x$ and $v,v'\in V_x$, distinct, if $L_x(u)\cap L_x(v)\neq\emptyset$, then there is no edge $e$ incident to $\{v',u\}$.
            \item For $u\in U_x$ and $v_{i_1},..., v_{i_n}\in V_x$, all distinct, if there are edges $e_1,..., e_n$ incident to $\{u, v_{i_1}\},..., \{u, v_{i_n}\}$, respectively, then there exists $X\in u\subset x$ such that $X\cap \ell_x(e_i)\neq\emptyset$ for every $i=1,..., n$.
        \end{enumerate}
    \end{claim}
    \begin{proof}
        The results follow from Claim \ref{lem:graph-behavior}.1
    \end{proof}

    For each $v_j\in V_x$ we define the set $C(v_j):=\{a\in \widetilde{\cc G}_x: L_x(a)=L_x(v_j)\}$ and the set $\fk v_j:= \supp C(v_j)$. Let $\overline{\fk v}_j := \bigcup_{A\in \fk v_j} A$. Note that 
    \[C(v_j)=\{v_j\}\cup\{u\in U_x: L_x(u)=L_x(a)\}\]
    and that $j\neq j'$ implies $C(v_j)\cap C(v_{j'})=\emptyset$, by Claim \ref{lem:graph-behavior}.2. We define a new partition $x'$ to be finest partition greater than $x$ so that each $\overline{\fk v}_j$ is a block of $x'$.
    %$x':= (x-\bigcup_{j=1}^S\fk v_j)\cup \{\bigcup \fk v_j\}_{j=1}^S$. It is clear that $x'\geq x$. 
    Now, we define 
    \[P':=\{\{A, B\}\in P: \forall j\in\{1,.., S
    \},\{A, B\}\notin C(v_j)\}\]
    and note that $|P'|=|P|-\sum_{j=1}^S|C(v_j)-1|$.

    \begin{claim}
        $P'$ is a maximum matching of $\cc G_{x'}$.
    \end{claim}

    \begin{proof}
        Let us suppose that there exists a matching $Q'$ of $x'$ with $|Q'|>|P'|$. By the previous observations, for all $j,j'\in\{1,...,S\}$, with $j\neq j'$, there are no edges in $\cc E_{x'}$ incident to $\{\overline{\fk v}_j, \overline{\fk v}_{j'}\}$; therefore, any pair $\{A, B\}\in Q'$ is of the form $A, B\in x-\bigcup_{j=1}^S\fk v_j$, or $A\in x-\bigcup_{j=1}^S\fk v_j$ and $B\in  \{ \overline{\fk v}_j\}_{j=1}^S$. Let us denote by $Q_1'$ the set of pairs of $Q'$ corresponding to the first case, and by $Q_2'$ the set of pairs corresponding to the second case. We can define a matching $Q$ of $\cc G_x$ the following way:
        \begin{itemize}
            \item Take all pairs in $Q_1'$.
            
            \item For each pair $\{A,B\}\in Q_2$, taking $A= \overline{\fk v}_j$, for some $j$, there exists $B'\in x\cap \fk v_j$ such that $\{B, B'\}$ is an incident edge of $\cc G_x$. Note that by construction $|\fk v_j|= 2|C(v_j)|-1$, and there exists $Y\in y$ such that $Y\cap X\neq\emptyset$ for all $X\in \fk v_j$. Therefore, we can find a perfect matching of $\{B\}\cup \fk v_j$ as a subgraph of $\cc G_x$ with size $|C(v_j)|$.
            
            \item Otherwise, we take all pairs in $C(v_j)$.
        \end{itemize}
        By construction it follows that the cardinal of this matching is $|Q|=|Q'|+ \sum_{j=1}^S|C(v_j)-1|>|P'|+\sum_{j=1}^S|C(v_j)-1|=|P|$, this is a contradiction since $|P|$ is a maximum matching of $\cc G_x$; therefore, $|P'|$ must be a maximum matching of $\cc G_{x'}$. \qedhere
    \end{proof}

    \begin{claim}\label{lem: match-bound}
        There exists a matching of $\cc G_y$ of cardinality at least $\left\lceil\frac{S}{2}\right\rceil$.
    \end{claim}

    \begin{proof}
        Since $P'$ is a maximum matching of $x'$ and $\{\overline{\fk v}_j\}_{j=1}^S$ enumerates its unmatched vertices, we can define the sets $U_{x'}, V_{x'}$ and the function $L_{x'}:\widetilde{\cc V}_{x'}\to 2^y$, and apply Claims \ref{lem:graph-behavior} and \ref{cor:match-beh} to them. Note that $U_{x'}\subset U_x$.
    
        Let us define the sets 
        \[Z_{x'}:=\{v\in V_{x'}: |L_{x'}(v)|\geq 2\}\ \textup{and}\  W_{x'}:= V_{x'}-Z_{x'}=\{v\in V_{x'}: |L_{x'}(v)|=1\}.\]

        Note that for each $v\in W_{x'}$ with $v=\{\overline{\fk v}_j\}$ for some $j$, if $\{Y\}=L_{x'}(v)$, then $\overline{\fk v}_j\subset Y$ and $L_x(v_j)=\{Y\}=L_x(\overline{\fk v}_j)$. Since $x_B\nleq y_B$ for all $B\in x+y$, we know that $Y\neq\overline{\fk v}_j$, and by Claim \ref{lem:graph-behavior}.2. we know that there exists $\{A, B\}=u\in U_{x'}$ such that $Y\cap (A\cup B)\neq\emptyset$. Since $u\in U_{x'}$, we know that $u\notin C(v_i)$, for all $i\in\{1,..., S\}$, so $L_{x'}(u)\neq \{Y\}$. Therefore, we see that for each $v\in W_{x'}$ there exists $u\in U_{x'}$ such that there is an edge $e$ incident to $\{u,v\}$ and $L_{x'}(u)\neq L_{x'}(v)$. So, we can define a function $\pi_{x'}: W_{x'}\to U_{x'}$ such that there exists an edge incident to $\{v, \pi_{x'}(v)\}$ and $L_{x'}(v)\neq L_{x'}(\pi_{x'}(v))$. 
        
        For each $v\in Z_{x'}$, we can define a pair $\{Y_v^0, Y_v^1\}\subseteq L_{x'}(v)$, by definition of $Z_{x'}$ We know that if $v\neq v'$, then $\{Y_v^0, Y_v^1\}\cap\{Y_{v'}^0, Y_{v'}^1\}=\emptyset$ by disjointness of the sets $C(v_j)$; therefore, the set $M_{Z_{x'}}:=\{\{Y_v^0, Y_v^1\}\}_{v\in Z_{x'}}$ is a matching of $\cc G_y$ . 
        
        On the other hand, if $u = \pi_{x'}(v)$ for some $v\in W_{x'}$, we can write $\pi_{x'}^{-1}(u)=\{w_i\}_{i=1}^m$. Writing $L_{x'}(w_i)=\{Y_{w_i}\}$ for $i=1,..., m$, we know, by Corollary \ref{cor:match-beh}.2, that there exists $X\in u\subset x$ such that for all $i\in\{1,..., m\}, X\cap Y_{w_i}\neq\emptyset$. By disjointness of the sets $C(v_j)$, we have that $i\neq i'$ implies $Y_{w_i}\neq Y_{w_{i'}}$ and by the construction of $\pi_{x'}$ and Corollary \ref{cor:match-beh}.1 that $L_{x'}(u)\not\subset\bigcup_{i=1}^m  L_{x'}(w_i)$. Therefore, we can define a matching of $\cc G_y$, $M_u:=\{\{Y_{u,j}^0, Y_{u,j}^1\}\}_{j=1}^{\lceil m/2\rceil}$ where $Y_{u,j}^i:= Y_{w_{2j-i}}$ for $i=0,1$ and $j=1,...,\lfloor m/2\rfloor$ and, if $m$ is odd, $Y_{u,\lceil m/2\rceil}^0:= Y_{w_m}$ and $Y_{u,\lceil m/2\rceil}^1\in L_{x'}(u)-\{Y_{w_i}\}_{i=1}^m$.

        By disjointness of the sets $C(v_j)$, we also know that for all distinct $u,u'$ in the image of $\pi_{x'}$ the respective sets of matched vertices in $M_u$ and $M_{u'}$ are pairwise disjoint (if $\{A,B\}\in M_u$ and $\{B,C\}\in M_{u'}$, then $\{A,B\}\cap \{C,D\}=\emptyset$) and pairwise disjoint from the matched vertices in $M_{Z_{x'}}$; thus, these matchings may be combined to form a matching $M= M_{Z_{x'}}\cup (\bigcup_{u\in Im(\pi_{x'})} M_u)$ of $\cc G_y$.
        %, \bigcup M_{Z_{x'}}\cap \bigcup M_u=\emptyset$ and $\bigcup M_u\cap \bigcup M_{u'}=\emptyset$. 
        %Thus, the set $M:=(\bigcup_{u\in Im(P_{x'})} M_u)\cup M_{Z_{x'}}$ is a matching of $\cc G_y$, and 
        We now calculate 
        %\[|M|=|M_{Z_{x'}}|+\sum_{u\in \pi_{x'}(W_{x'})} |M_u|=| Z_{x'}|+\sum_{u\in \pi_{x'}(W_{x'})} \left\lceil\frac{|P_{x'}^{-1}(u)|}{2}\right\rceil
        %\geq \left\lceil\frac{|Z_{x'}|}{2}\right\rceil+\left\lceil\frac{}{}\right\rceil\]
        \begin{equation*}
        \begin{split}
            |M|&=|M_{Z_{x'}}|+\sum_{u\in \pi_{x'}(W_{x'})} |M_u|\\
            &=| Z_{x'}|+\sum_{u\in \pi_{x'}(W_{x'})} \left\lceil\frac{|\pi_{x'}^{-1}(u)|}{2}\right\rceil\\
            &\geq \left\lceil\frac{|Z_{x'}|}{2}\right\rceil+\left\lceil\frac{\sum_{u\in \pi_{x'}(W_{x'})}|\pi_{x'}^{-1}(u)|}{2}\right\rceil\\
            &\geq \left\lceil\frac{|Z_{x'}|+ |W_{x'}|}{2}\right\rceil=\left\lceil\frac{|V_{x'}|}{2}\right\rceil=\left\lceil\frac{S}{2}\right\rceil
        \end{split}
        \end{equation*}
        which establishes the claim. \qedhere
    \end{proof}

    Now, besides the maximum matching of $\cc G_x$ previously fixed, we fix a maximum matching $\{\{ A_i', B_i'\}\}_{i=1}^M$ of the vertices of $\cc G_y$ and let $\{\{C_j'\}\}_{j=1}^ R$ enumerate the unmatched vertices. 
    Using the maximum matching of $\cc G_x$ we define a singular partition $z^*\in\G(x)$ with main block
    $B_{z^*}:=\{a_i,b_i\}_{i=1}^N\cup \{c_j\}_{j=1}^S$ where $a_i\in A_i, b_i\in B_i$ and $\{a_i, b_i\}\subseteq Y\in y$ for $i=1,...., N$ and $c_j\in C_j$ for $j=1,..., S$. Similarly, we can define a singular partition $w_*\in\G(y)$ using the maximum matching of $\cc G_y$. It is clear by its construction that $\max_{w\in\G(y)}|B_{z^*} \cap B_w|\leq N+S$ and $\max_{z\in\G(x)}|B_z\cap B_{w^*}|\leq M+R$; therefore,
    \[\displaystyle\min\{\g(x,y),\g(y,x)\}\leq N+S, M+R\]

    Since $\#x=2N+S$ and $\#y=2M+R$, we note that
    \[4\displaystyle\min\{\g(x,y),\g(y,x)\}\leq 2(M+R)+2(N+S)=\#x+\#y+ R+S\]
    To finish the proof, observe that by Claim \ref{lem: match-bound} we know that $S\leq 2M$. Thus, we get $R+S\leq R+ 2M\leq\#y+\#(x+y)$ and conclude that \[4\displaystyle\min\{\g(x,y),\g(y,x)\}\leq \#x+2\#y+\#(x+y). \qedhere\]
    
\end{proof}

\subsection{Connections with continuous partition lattices}

Bj\"orner and Lov\'asz \cite{bjorner-lovasz} introduced an inductive limit-type construction on the class of pseudomodular lattices in order to produce examples of continuous limits of lattices. We refer the reader to \cite{Bjorner-continuous} for a broad survey of this topic. In \cite{Bjorner-part} Bj\"orner constructed in this fashion a continuous limit of partition lattice which we will denote by $\Pi_\infty$. This was realized as a sublattice of a lattice of measurable partitions in subsequent work of Haiman \cite{haiman}. 

We now describe the construction of $\Pi_\infty$. We will use $\Pi_n$ to denote the finite partition lattice of the set $\{0,1,...,n\}$. As metric lattices $\Pi_n$ and $P_{n+1}$ are isomorphic. If $kn=m$, then a lattice homomorphism
\[\phi_n^m:\Pi_n\to \Pi_m\] is defined as follows. For a partition $\pi=\{B_0,\dotsc,B_p\}\in \Pi_n$ such that $0\in B_0$, $\phi_n^m(\pi):=\{C_0\}\cup \{C_{ij}\}_{1\leq i\leq p, 0\leq j\leq k-1}$ where $C_{ij}:=\{kb-j|b\in B_i\}$ and $C_0:=\{0,1,\dotsc,m\}-\bigcup_{i,j} C_{ij}$. 

We then define $\Pi_{(\infty)}$ as the direct limit of $(\Pi_{k}, \phi_1^k)$, which is a lattice with a rank function $|\cdot|:\Pi_{(\infty)}\rightarrow[0,1]$ that only takes rational values. We then define $\Pi_\infty$ as the metric completion of $\Pi_{(\infty)}$ and note that $\Pi_{(\infty)}$ is a sublattice of $\Pi_\infty$. Observe that for each $k$ the construction induces a lattice embedding $\phi_{k}:\Pi_k\rightarrow\Pi_{\infty}$.

\begin{question}\label{q:bjorner-model-tfpl}
    Is $\Pi_\infty$ a model of $T_{FPL}$?
\end{question}

We do not have a clear answer to the question. On the level of naive analogy, $\Pi_\infty$ ought to be the right object in the category of ``continuous partition lattices'' corresponding to the hyperfinite II$_1$ factor in the theory of tracial von Neumann algebras. It is known that the hyperfinite II$_1$ factor is not pseudomatricial among tracial von Neumann algebras as pseudomatricial von Neumann algebras do not satisfy ``property $\Gamma$'' in contrast to the hyperfinite II$_1$ factor \cite{fhs-iii}. We develop this analogy further below. For the meantime, we note that it would suffice to show that for every $\fk L$-formula $\psi(\bar x)$ we would have that $\sup_{\bar a_n}|\psi^{\Pi_\infty}(\phi_n(\bar a_n)) - \psi^{\Pi_n}(\bar a_n)|\to 0$ as $n\to \infty$ where $\bar a_n$ ranges over all tuples in $\Pi_n$. This would be trivially satisfied if all embeddings $\phi_n: \Pi_k\to \Pi_\infty$ are elementary. However, this is not the case.

\begin{remark}\label{rem:failure-Bj-embeddings}
    For all positive integers $n,k>1, \phi_n^{kn}(\Pi_n)$ is not an elementary substructure of $\Pi_{kn}$. Moreover, no $\phi_n: \Pi_n\to \Pi_\infty$ is an elementary embedding.
\end{remark}

\begin{proof}
    Let's consider the formula $\psi(x)=\sup_y\vp(x,y)$ where $\vp(x,y)$ is as in equation line (\ref{eq:modular-pair}). Picking a modular element $a\in \Pi_n, a\neq 0$, it is clear that $\psi^{\Pi_n}(a)=0$. By Corollary \ref{cor:metrically-modular-equals-singular} we know that $a$ is a singular partition, so there is a main block $A$ with cardinality at least $2$. By the definition of $\phi_n^{kn}$ we know that $\{A_j\}_{0\leq\leq k-1}\subset\phi_n^{kn}(a)$ where $A_j=\{kb-j|b\in A\}$. Since $\phi_n^{kn}(a)$ has $k$ blocks of cardinal $\geq 2$, it is not a singular partition; thus, by Corollary \ref{cor:metrically-modular-equals-singular} we conclude that $\psi^{\Pi_{kn}}(\phi_n^{kn}(a))\neq 0$. 
\end{proof}

Despite knowing that the embeddings are not elementary, we can still ask whether the embeddings allow us to approximate formulas of $T_{FPL}$ uniformly. In other we are asking if for every $\e>0$ there exists $N\in\bb N$ such that for all $n\geq N$, all $\fk L$-formula $\psi(\bar{x})$ and all $\bar{a}\in \Pi_n, |\psi^{\Pi_n}(\bar{a})-\psi^{P_{Bjorn}}(\vp_n(\bar{a}))|<\e $.

The answer turns out to still be no. Let us pick $\e=\frac{1}{3}$, for each $n\in \bb N$ we can define the $\fk L$-formula with $n$ free variables, $\psi_n(x_1,..., x_n):=\sup_y\min\{d(x_1,y),...,d(x_n,y)\} $. It is clear that for all $n\in\bb N$, if we denote $|\Pi_n|=N$ and $\Pi_n=\{a_i\}_{i=1}^N$, then $(\psi_N(a_1,..., a_N))^{\Pi_n}=0$. However, in $\Pi_{2n}$ we can define the partition $z:=\{\{0\}\}\cup\{\{2i-1, 2i\}\}_{i=1}^n$. Note that for any partition $x=\{B_0,..., B_p\}\in \Pi_n$ we know that $\vp_n^{2n}(x)=\{C_0\}\cup\{C_{ij}\}_{1\leq i\leq p, j=1,0}$ where $C_{ij}=\{2b-j|b\in B_i\}$ and $C_0=\{0,1,..., 2n\}-\bigcup_{i,j} C_{ij}$, so for each $m\in\{1,..., n\}$ if $m\in B_i\in x$ for some $i\in\{1,...,p\}$, then $2m\in C_{i0}$ and $2m-1\in C_{i1}$. This implies that for all $x\in \Pi_n, \vp_n^{2n}(x)+z=\{C_0\}\cup\{C_{i0}\cup C_{i1}\}_{i=1}^p$, and that $\#(\vp_n^{2n}(x)+z)=\#x$. Now, if we compute the distance of $\vp_n^{2n}(x)$ and $z$, we get that
\[d(\vp_n^{2n}(x),z)=\frac{\#\vp_n^{2n}(x)+\#z-2\#(\vp_n^{2n}(x)+z)}{2n}=\frac{2\#x-1+ n+1-2\#x}{2n}=\frac{n}{2n}=\frac{1}{2}\]

for all $x\in\Pi_n$. Therefore, $\min\{d(\vp_n^{2n}(a_1), z),..., d(\vp_n^{2n}(a_N), z)\}=\frac{1}{2}$ and it follows that $\psi_N^{\Pi_\infty}(\vp_n(a_1,..., a_N))\geq \psi_N^{\Pi_{2n}}(\vp_n^{2n}(a_1,..., a_N))\geq\frac{1}{2}>0=\psi_N^{\Pi_n}(a_1,...,a_N)$. We have proved that for all $n\in\bb N$ there exists a formula $\psi_{|\Pi_n|}(x_1,..., x_{|\Pi_n|})$ and elements $a_1,..., a_{|\Pi_n|}\in \Pi_n$ such that $|\psi_{|\Pi_n|}^{\Pi_n}(\bar{a})-\psi_{|\Pi_n|}^{\Pi_\infty}(\vp_n(\bar{a}))|=\psi_{|\Pi_n|}^{\Pi_\infty}(\vp_n(\bar{a}))-\psi_{|\Pi_n|}^{\Pi_n}(\bar{a})>\e=\frac{1}{3}$, so there is no uniform elementary approximation of the $\fk L$-formulas through the embeddings.

\begin{remark}
    One can, in fact, control the behavior for a restricted set of formulas. We say that an $L$-formula is in \emph{prenex form} if it is of the form $Q_{x_1}^1Q_{x_2}^2\dotsb Q_{x_n}^n\psi(\bar x, \bar y)$ where $\psi$ is a quantifier-free formula and each $Q_i$ is either $\sup$ or $\inf$. Every formula is equivalent to one in prenex form. We call a prenex formula $Q_{x_1}^1Q_{x_2}^2\dotsb Q_{x_n}^n\psi(\bar x, \bar y)$ $\forall\exists$ if there is $1\leq i \leq n$ so that $Q_j = \sup$ for all $j\leq i$ and $Q_j = \inf$ for all $j> i$. It is a well-known fact that models of an $\forall\exists$-sentence are closed under inductive limits \cite[Theorem 3.2.3]{chang-keisler}, hence
    \begin{prop}\label{lem: vality-of-AE-sentences}
    If $\sigma\in T_{FPL}$ is an $\forall\exists$ sentence, then $\Pi_\infty\models\sigma$.
    \end{prop}
\end{remark}

A natural question at this point is whether every formula in $T_{FPL}$ is equivalent to an $\forall\exists$ formula. Again, the answer is not totally clear. An interesting test case would be the sentence $\inf_x\sup_y\vp(x,y)\in T_{FPL}$. This sentence is also true in $\Pi_\infty$ by the results of Bj\"orner \cite{Bjorner-part}. 

Perhaps the most fundamental open question is the following.

\begin{question}
    How many models of $T_{FPL}$ exist, up to isomorphism, of density character $\aleph_0$?
\end{question}

In an attempt to answer this question, one can try to formulate invariants which could be shown to hold for certain constructions.

\begin{defn}
    Let $\cc L$ be a metric lattice. We say that a net $(a_n)$ in $L$ is \emph{almost modular} if $\vp(x,a_n)\to 0$ for all $x\in L$, where $\vp$ is as in equation line (\ref{eq:modular-pair}). We say that an almost modular net $(a_n)$ is \emph{trivial} if there exists a net of modular elements $(b_n)$ in $L$ so that $d(a_n,b_n)\to 0$. We say that a metric lattice has \emph{property $\G$} if there exists a non-trivial almost modular net. 
\end{defn}

The terminology ``property $\G$'' is inspired by the theory of von Neumann algebras, specifically from the seminal work of Murray and von Neumann \cite{MvN-iv}, where it is defined as the algebra having an almost central sequence which is bounded away from sequences of central elements. We do not claim any formal connection between the two.

\begin{question}\label{q:bjorner-gamma}
    Does $\Pi_\infty$ have property $\G$?
\end{question}

\begin{question}\label{q:fpl-gamma}
    Do all models of $T_{FPL}$ have property $\G$?
\end{question}

\begin{prop}\label{prop:no-gamma}
    The following are equivalent:
    \begin{enumerate}
        \item No model of $T_{FPL}$ has property $\G$.
        \item For every $\e>0$ there exists $k\in\bb N$ and $\de>0$ so that for all $P_n$, a finite partition lattice, there exists $x_1,\dotsc, x_k$ in $P_n$ so that $\sum_{i=1}^k \vp(x_i,y) < \de$ implies that $d(y,\Sg_n)<\e$.
    \end{enumerate}
\end{prop}

%\begin{proof}
 %    The proof of ``(2) implies (1)'' is straightforward. 
%
%    For the proof of ``(1) implies (2),'' suppose that the second statement is false. That is, there is some $\e>0$ so that for all $k,l\in\bb N$ we have that there is $n_l\in \bb N$ so that for any choice of $x_1,\dotsc,x_k\in P_{n_l}$ there is $y_{l}\in P_{n_l}$ with $d(y_{l},\Sg_{n_l})\geq \e$ and $\sum_{i=1}^K\vp(y_{l},x_i)\leq \frac{1}{k}\e$.
%    Fix a non-principal ultrafilter $\cc U$ on $\bb N$, and consider the ultraproduct lattice $\cc P = \prod_{l\in\cc U} P_{n_l}$. For any finite subset $F = \{x_1,\dotsc,x_k\}$ of $\cc P$ choose representing sequences $x_i = (x_{l,i})_{l\in\bb N}$ and $y_l\in P_{n_l}$ so that $d(y_{l},\Sg_{n_l})\geq \e$ and $\sum_{i=1}^K\vp(y_{l},x_{l,i})\leq \frac{1}{|F|}\e$. Setting $y_F$ to be the class of $(y_l)$ in $\cc P$, we have that $(y_F)$ is a non-trivial almost modular net. \qedhere
%\end{proof}
%\rmkjc{This is my clean version of the proof. }
\begin{proof}
     The proof of ``(2) implies (1)'' is straightforward. 

    For the proof of ``(1) implies (2),'' suppose that the second statement is false. That is, there is some $\e>0$ so that for all $k,l\in\bb N$ we have that there is $n_{l,k}\in \bb N$ so that for any choice of $x_1,\dotsc,x_k\in P_{n_{l,k}}$ there is $y\in P_{n_{k,l}}$ with $d(y,\Sg_{n_{l,k}})\geq \e$ and $\sum_{i=1}^k\vp(y,x_i)\leq \frac{1}{l}\e$. Specifically, for all $k\in\bb N$, there is $n_k:=n_{k,k}\in\bb N$ such that for any choice of $x_1,\dotsc,x_k\in P_{n_k}$ there is $y\in P_{n_k}$ with $d(y,\Sg_{n_k})\geq \e$ and $\sum_{i=1}^k\vp(y,x_i)\leq \frac{1}{k}\e$.
    Fix a non-principal ultrafilter $\cc U$ on $\bb N$, and consider the ultraproduct lattice $\cc P = \prod_{k\in\bb N} P_{n_k}/\cc U$. For any finite subset $F = \{x_1,\dotsc,x_l\}$ of $\cc P$ choose representing sequences $x_i = (x_{k,i})_{k\in\bb N}$, and for each $k\in \bb N, k\geq l$ we can pick $y_{k,F}\in P_{n_k}$ such that $d(y_{k,F},\Sg_{n_k})\geq \e$ and $\sum_{i=1}^l\vp(y_{k,F},x_{k,i})\leq \sum_{i=1}^l\vp(y_{k,F},x_{k,i})+ (k-l)\vp(y_{k,F},x_{k,l})\leq\frac{1}{k}\e\leq \frac{1}{|F|}\e$, without loss of generality we can pick arbitrary elements of $P_{n_k}$ to define $y_{F,k}$ for each $k< l$.  Setting $y_F$ to be the class of $(y_{k,F})$ in $\cc P$, we have that $d(y_F, \Sg_\cc P)=d(y_F, \Pi_{k\in n} \Sg_{n_k}/\cc U)=\lim_{k,\cc U}d(y_{k,F},\Sg_{n_k})\geq \e$ and $\Sg_{i=1}^l\vp(y_F, x_i)=\lim_{k,\cc U} \sum_{i=1}^l\vp(y_{k,F},x_{k,i})\leq \frac{1}{|F|}\e$. Therefore,  $(y_F)$ is a non-trivial almost modular net. 
                                                                                                                   \end{proof}
By Proposition \ref{prop:no-gamma} an affirmative answer to the following question would imply that no model of $T_{FPL}$ has property $\G$.

\begin{question}
    Do there exist constants $C,K$ so that for all $P_n$, a finite partition lattice, there exists $x_1,\dotsc, x_K$ in $P_n$ so that $d(y,\Sg_n)\leq C\sum_{i=1}^K \vp(x_i,y)$?
\end{question}
\noindent Although this seems like a very strong property, there are analogs of equally unexpected and strong ``spectral gap'' type properties for matrix algebras: See, for instance, \cite{malnormal, vN-matrices}. On the other hand, it is unknown whether there is a uniform spectral gap estimate over all non-$\Gamma$ II$_1$ factors, nor is it known whether the non-$\Gamma$ II$_1$ factors form an axiomatizable class \cite{fhs-iii}.

It would be nice to have a simple description of the class of all pseudofinite partition lattices. 

\begin{question}
    Suppose that $\cc L$ is a complete metric lattice with the following properties: 
    \begin{enumerate}
        \item The set of metrically modular elements is a Boolean lattice;
        \item For every $x\in L$ the set $\G(x)$ of all modular complements of $x$ is nonempty;
        \item For all $x,y\in L$ \[\frac{1}{4}d(x,y) \leq d_{Haus}(\G(x),\G(y))\leq d(x,y).\]  
    \end{enumerate}
    Does $\cc L \models T_{FPL}$?
\end{question}

The properties of the selector sets of partition lattices suggest a possible framework for a ``Boolean'' theory of limits of metric lattices or submodular functions in general, shedding light on a question of Lov\'asz \cite[Problem 8.9]{lovasz-submod-setfunction}.
Let $\fk B = (\cc B, \cup, \cap, \cdot^c, d)$ be a complete metric Boolean lattice. For nonempty closed subsets $P,Q$ of $\cc B$, we write
\[P\cap Q = \{x\cap y : x\in P, y\in Q\}.\] We consider a collection $\G$ of nonempty closed subsets of $\cc B$ with the following properties:
\begin{enumerate}
    \item $\{0\}, \{1\}\in \G$.
    \item For each $P\in \G$, $|x| = |y|$ for all $x,y\in P$, we let $||P||:=|x|$ for any $x\in P$.
    \item For all $P,Q\in \G, \{S\in \G: S\subset P\cap Q\}\neq\emptyset$ and there is a unique $R\in \G$ with $R\subseteq P\cap Q$ and $||R||=\max\{||S||: S\in \G\wedge S\subseteq P\cap Q\}$. We write $R = P\otimes Q$.
    \item $\G$ is closed in the Hausdorff metric.
\end{enumerate}

We observe that $(\G, \otimes)$ is a meet semilattice with $\{0\}$ and $\{1\}$ being the minimal and maximal elements, respectively. Note that for the induced ordering, we have that $Q\leq P$ if and only if for every $x\in Q$ there is $y\in P$ so that $x\leq y$. For $P, Q\in \G$ let $\rho(P,Q) : =  \|P\| +  \|Q\| - 2\|P\otimes Q\|$. 

\begin{prop}
    We have that $(\G, \otimes, \rho)$ is a complete metric semilattice.
\end{prop}

\begin{proof}
    It is straightforward to see that $\G$ is metrically complete follows from $\G$ being closed in the Hausdorff metric. Let $P,Q,R\in \G$ with $R\leq P$ and $x\in P, y\in Q, z\in R$. We have that 
    \begin{equation*}
        |x\cap y| + |z| \leq |x| + |y\cap z|
    \end{equation*}
    whenever $x\geq z$. As this can always be arranged by assumption, it follows that
    \begin{equation*}
        \|P\otimes Q\| + \|R\| \leq \|P\| + \|Q\otimes R\|.
    \end{equation*}
    The result now follows by an adaptation of the proof of Proposition \ref{prop:exchange-relation} to meet semilattices.
\end{proof}

\begin{remark}
     Let $\cc P$ be an infinite pseudofinite partition lattice. Due to our previous work in section \ref{sec: set-sel}, we believe it should be the case that the set $\G :=\{\G(x) : x\in \cc P\}$ satisfies the above properties with respect to the Boolean algebra of modular elements of $\cc P$. Indeed, it can be checked that $\{0\}=\G(1), \{1\}=\G(0)$ so (1) is satisfied, also (2) follows from Lemma \ref{lem:partitions-have-selectors} and (4) from Proposition \ref{prop:selectors-hausdorff}. Given the length of the proof, we have decided to omit the details for the satisfaction of (3), but we mention some key points. Note that by the techniques used in the proof of Proposition \ref{prop:formula-selector-set-distance} it follows that $\G(x+y) \subseteq \G(x)\cap \G(y)$ and $1-|x+y|=||\G(x+y)||=\max\{||\G(z)||: \G(z)\subseteq \G(x)\cap\G(y)\}$ for all finite partition lattices. From here, using the theory of selectors and the definability of the meet in the context we can prove that (3) is satisfied by $\cc P$, and thus $\G(x)\otimes \G(y) = \G(x+y)$ for all $x,y\in P$. 
     
     If given a metric lattice $\cc L$, we define the opposite lattice $\overline{\cc L}$ by $\bar x + \bar y = \overline{xy}$, $\bar x\cdot \bar y = \overline{x+y}$, and $d(\bar x, \bar y) = d(x,y)$ for all $x,y\in L$, we have that $\overline{\cc L}$ is a meet metric semilattice. In this way $\overline{\G}$ is isometrically isomorphic to $\cc P$.
\end{remark}

\begin{question}
    Is every complete metric semilattice isometrically isomorphic to one of this form?
\end{question}

\section{Further Directions and Questions}\label{sec:further}

In their seminal work on continuous limits of lattices, Bj\"orner and Lov\'asz \cite{bjorner-lovasz} isolate the following property as an essential ingredient for their constructions, which they call \emph{pseudomodularity}. We will call a metric lattice $\cc L$ pseudomodular if for all $x,y\in L$ the set $P_{xy} :=\{z\leq y : d(x+z,z) = d(x+y,y)\}$ has a unique least element. 

\begin{question}
    Do the pseudomodular complete metric lattices form an elementary class?
\end{question}

The following definition is inspired by the work of Krivine and Maurey \cite{krivine}.

\begin{defn}
    We say that a metric lattice $\cc L$ is \emph{stable} if 
    \[\lim_i \lim_j |x_i + y_j| = \lim_j \lim_i |x_i + y_j|\] for all sequences $(x_i)$ and $(y_j)$ in $L$ for which both limits exist. 
\end{defn}

\begin{question}
    Is there some ``combinatorial'' description of stable metric lattices?
\end{question}
\noindent Stability of geometric lattices in the context of classical model theory has been studied in \cite{hyttinen}.

Let $G = (V,E)$, be a simple, undirected finite graph which is connected. We can associate to $G$ a subsemilattice $L(G)$ of $P_n$, where $|V| = n$, consisting of the set of all partitions of $V$ whose blocks are the connected components of the subgraph induced by some $E'\subseteq E$. This turns out to be isomorphic to the lattice of flats of the graphical matroid associated to $G$. We note that $L(K_n) = P_n$, where $K_n$ is the complete graph on $n$ vertices. 

\begin{question}
    For which families of finite graphs, do the lattices $L(G)$ form a Fra\"iss\'e class of metric structures as in \cite{ben-fraisse, vignati}?
\end{question}

\noindent Such a class ought to have a canonical ``hyperfinite'' limiting object.

The next question turns to connections between our theory and the limiting theories of combinatorial structures using sampling densities as developed in \cites{kardos, lovasz-large,nesetril,razborov}. If $\cc L$ is a geometric lattice with $n$ atoms $A_{\cc L}$ and $\cc H$ is a fixed geometric lattice with $k$ atoms, we write 
\begin{equation}
    t(\cc H, \cc L) := \binom{n}{k}^{-1}\left|\{S\subset A_{\cc L} : |S|= k, \cc L_S\cong \cc H\}  \right|.
\end{equation} Here $\cc L_S$ denotes the sublattice of $\cc L$ generated by $S$. 

\begin{question}
    For what $\cc H$ does $t(\cc H, \bullet)$ extend to a sentence in the theory of metric lattices?
\end{question}

\noindent As a first step towards answering this, one must consider the following. 

\begin{question}
    If $(\cc L_n)$ and $(\cc M_n)$ are sequences of geometric lattices (with appropriately normalized rank functions) and $\mathcal U$ is an ultrafilter so that $\prod_{\mathcal U} \cc L_n\cong \prod_{\mathcal U} \cc M_n$ isometrically isomorphically as metric lattices, does it hold that $\lim_{\mathcal U} t(\cc H, \cc L_n) = \lim_{\mathcal U} t(\cc H, \cc M_n)$? 
\end{question}

We denote by $NC_n$ the lattice of noncrossing partitions with respect to the usual cyclic order on $\{1,\dotsc,n\}$. Note that $NC_n$ is not a sublattice of the finite partition lattice $P_n$: The meets are the same, but the joins are different. The lattice of noncrossing partitions has a rich combinatorial structure in its own right: see \cite[Lecture 9]{nica} or \cite{simion}. For instance, noncrossing partition lattices are self-dual, while partition lattices are not. However, the rank function on $NC_n$ inherited from $P_n$ is not submodular for $n\geq 4$ \cite[p.\ 370]{simion}.
Moreover, there is a canonical order-reversing lattice anti-isomomorphism $\kappa_n: NC_n\to NC_n$ called the \emph{Kreweras complement}. This poses a serious obstacle to realizing $NC_n$ as a metric lattice for which the metric is compatible with $\kappa_n$, as the order-reversing property together with submodularity ought to give modularity. 

\begin{question}
    Is there a good continuous limiting theory of noncrossing partition lattices?
\end{question}

\noindent  A satisfactory answer to this question likely would have connections with free probability theory \cite{nica}.

Finally, there are ``weak order'' structures on finite permutation groups \cite{takemi}, and finite Coxeter groups more generally \cite{brenti}, which endow these groups with a lattice structures. While these lattice structures need not be geometric, since not all maximal chains need have the same length, they possess a strong property called \emph{semidistributivity}:
\begin{equation}
    \begin{split}
        x+y=x+z &\Rightarrow x+y = x+yz\\
        xy = xz &\Rightarrow xy = x(y+z).
    \end{split}
\end{equation}

\begin{question}
    Is there a metric theory of semidistributive lattices?
\end{question}

\section*{Acknowledgements} 

The authors thank David Jekel for helpful comments on any earlier version of this manuscript and the anonymous referee for many suggestions which improved the exposition. The authors are grateful to the Colombia--Purdue Partnership and Andr\'es Villaveces for facilitating and supporting their collaboration. 

\section*{Funding}

The authors were supported by NSF grant DMS-2055155.

\end{document}